\documentclass[12pt]{article}
\usepackage{enumitem, pgf,tikz, layout,  graphicx,  amssymb,amsmath,amsthm,bm,mathtools,faktor,wrapfig}
\usepackage{microtype}
\usetikzlibrary{patterns}
\usepackage{symlist}
\mathtoolsset{showonlyrefs}  
\usetikzlibrary{arrows,matrix}
\usepackage[total={148mm,190mm}]{geometry}
\usepackage[bitstream-charter]{mathdesign}
\usepackage[english]{babel}
\usepackage[utf8x]{inputenc}
\usepackage[colorinlistoftodos]{todonotes}
\usepackage[colorlinks=true, allcolors=blue,hyperindex,breaklinks]{hyperref}

\author{Johan Andersson\thanks{Email:johan.andersson@oru.se \, Address:Department of Mathematics, School of Science and Technology, {\"O}rebro University, {\"O}rebro, SE-701 82 Sweden. \, This research has been supported by a Crafoord prize research grant.}}
 
\title{Voronin Universality in several  complex variables}

\theoremstyle{plain} 
\newtheorem{thm}{Theorem}  
\newtheorem{lem}{Lemma}
\newtheorem{cor}{Corollary} 
\theoremstyle{definition}

\newtheorem{defn}{Definition}

\newtheorem{voronin}{The Voronin universality theorem}
\newtheorem{garunkstis}{Garunk\v{s}tis' effective version of the Voronin universality theorem}
\newtheorem{lam}{Lamzouri-Lester-Radziwill's  version of the Voronin universality theorem}

\newcommand{\nicefrac}[2]{\leave§ vmode\kern.1em
\raise.5ex\hbox{\the\scriptfont0 #1}\kern-.1em
/\kern-.15em\lower.25ex\hbox{\the\scriptfont0 #2}}
\def\halv{\mathchoice{{ \frac 1 2}}{1/2}{1/2}{1/2}}
\def\cprime{$'$}
\renewcommand{\pod}[1]{\allowbreak\mathchoice
  {\if@display \mkern 18mu\else \mkern 8mu\fi (#1)}
  {\if@display \mkern 18mu\else \mkern 8mu\fi (#1)}
  {\mkern4mu(#1)}
  {\mkern4mu(#1)}
}

\renewcommand{\mod}{ \, \operatorname{mod} \,}
\newcommand{\dddd}{{\delta}}
\newcommand{\C}{{\mathbb C}} 
\newcommand{\R}{{\mathbb R}}
\newcommand{\N}{{\mathbb N}}
\newcommand{\Z}{{\mathbb Z}}
\newcommand{\Q}{{\mathbb Q}}
\newcommand{\ddd}{{\eta}}
\newcommand{\ddddd}{{\xi}}

\newcommand{\NC}{{\mathcal N}}
\newcommand{\BC}{{\mathcal B}}
\newcommand{\Pri}{{\mathcal P}}
\newcommand{\M}{{\mathcal M}}
\newcommand{\cA}{{\mathcal A}}
\newcommand{\cB}{{\mathcal B}}
\newcommand{\cC}{{\mathcal C}}
\newcommand{\val}{{\bm \alpha}}

\newcommand{\va}{{\bm a}}
\newcommand{\ett}{{\bm 1}}
\newcommand{\vb}{{\bm b}}

\newcommand{\vs}{{\bm s}} 
\newcommand{\vv}{{\bm v}}
\newcommand{\vz}{{\bm z}}

\newcommand{\LL}{{M}}
\newcommand{\vt}{{\bm t}}\newcommand{\vx}{{\bm x}}
\newcommand{\vsigma}{{\boldsymbol{\sigma}}}
\newcommand{\norm}[1]{\left \Vert {#1} \right \Vert}
\newcommand{\abs}[1]{{\left| {#1} \right|}} \newcommand{\p}[1]{{\left(
     {#1} \right)}} 
  
\newcommand{\meas}{\operatorname{meas}}

\newcommand{\GCD}{\operatorname{GCD}}

\renewcommand{\Re}{\operatorname{Re}} \renewcommand{\Im}{\operatorname{Im}} 
\begin{document}
 \date{}

\maketitle        
\begin{abstract} 
 We prove the Voronin universality theorem for the   
multiple Hurwitz zeta-function with rational or transcendental parameters in $\C^n$ answering  a question of Matsumoto. In particular this implies that the Euler-Zagier multiple zeta-function is universal in several complex variables and gives the first example of a Dirichlet series that is universal in more than one variable. 
%  We also prove joint and discrete universality results in several complex variables. 
\end{abstract}
   
\tableofcontents  
\listofsymbols 

\section*{Notes for the reader - Changes in  v2 of paper}
\begin{itemize}
\item I have changed the font and size of the paper for  better readability on a 10 inch ebook reader. In order to get the paper properly formatted for this, remove the margins of the pdf-file of the paper or please uncomment the 10th line of the latex-file source and comment out the 11th line. 
\item I have removed the joint universality  theorems and discrete universality theorems in order for the paper to be more focused. These results are a better fit for a sequel. 
\item I have removed the conjectural part on algebraic irrational parameters. When revisiting the argument I can not currently see how even the strong Conjecture 1 of v1 of this paper implies universality for Hurwitz zeta-function with an algebraic irrational parameter.
 \item I have proof read the paper and corrected a number of minor mistakes.
 \item Addition of the references \cite{Aold,Anew1,Anew2,Anew3,Rudin2}
 \item I have slightly changed the statement and proof of Lemma  1
 \item I have changed the discussion of upcoming work on Weyl-group multiple Dirichlet series. This is because I discovered a mistake in my original approach. 
 \item As in v1 of the paper. Comments are appreciated!
 \end{itemize}

 \section{Introduction and Main results}
 
\subsection{Classical Voronin universality}
One of the deep theorems on the Riemann zeta-functions is the following result of Voronin \cite{Voronin2, Voronin0}.
\begin{voronin}  Let 
 $K=\{ s \in \C :|s-3/4| \leq r\}$ for some $r<1/4$, and suppose that $f$ is any continuous zero free function on $K$ that is analytic  in the interior of $K$. Then 
$$\liminf_{T \to \infty} \frac 1 T \mathop{\rm meas} \left \{t \in [0,T]:\max_{s \in K} \abs{\zeta(s+it)-f(s)}<\varepsilon \right \}>0. $$
\end{voronin} 
This theorem has been generalized to a lot of settings. 
In particular the set $K$ can be chosen as any compact set in the strip  
\begin{gather} 
  \label{Ddef} \newsym{$\{s \in \C:1/2<\Re(s)<1\}$}D=
  \{s \in \C:1/2<\Re(s)<1\} 
\end{gather} 
with connected complement (Reich \cite{Reich}, see also  \cite[pp. 18-19]{Steuding}), and it has been proved for different zeta and $L$-functions such as Dirichlet $L$-functions,  more generally a variant of the Selberg class \cite{Steuding} assuming some standard conjectures, and the Selberg zeta-function \cite{DruGarKac}. The Hurwitz zeta-function\footnote{It is customary to assume that $0<\alpha \leq 1$ in which case the Hurwitz zeta-function has a nice functional equation, but we will adopt the convention to allow any $\alpha>0$ in the definition of the Hurwitz zeta-function.} 
  \begin{gather}
	   \newsym{The Hurwitz zeta-function}{\zeta(s,\alpha)}=\sum_{k=0}^\infty (k+\alpha)^{-s} \qquad (\Re(s)>1),
\end{gather} 
which like the Riemann zeta-function $\zeta(s)=\zeta(s,1)$ extends analytically to $\C \setminus\{1\}$  is an especially interesting case since its treatment depends intimately on the Diophantine properties of $\alpha$.  In the case of zeta-functions without Euler product where the Riemann hypothesis does not hold, such as the Hurwitz zeta-function\footnote{Universality for the Hurwitz zeta-function was proved independently by Gonek \cite{Gonek} and Bagchi \cite{Bagchi2}.}
$\zeta(s,\alpha)$ for transcendental or rational $2 \alpha \not \equiv 0 \pmod 1$ the assumption that $f$ is zero-free on $K$ may be removed. It is an open problem whether $\zeta(s,\alpha)$ is universal when  $\alpha$ is an algebraic irrational number.

Another zeta-function that has been considered is the Multiple Hurwitz zeta-function which is defined by
\begin{gather} \label{mhz}
 \newsym{The multiple Hurwitz zeta-function}{\zeta_n(\vs;\val)}=\sum_{0 \leq k_1<k_2< \cdots < k_n} (k_1+\alpha_1)^{-s_1} \cdots  (k_n+\alpha_n)^{-s_n}, \\ \intertext{where} 
\newsym{$(s_1,\ldots,s_n) \in \C^n$}{\vs} =(s_1,\ldots,s_n), \qquad \newsym{$(\alpha_1,\ldots,\alpha_n) \in (\R^+)^n$} {\val}=(\alpha_1,\ldots,\alpha_n),
\end{gather}
and  $\Re(s_j)>1$.  An interesting special case is when $\alpha_j=1$ for $j=1,\ldots,n$ in which case the Multiple Hurwitz zeta function specializes to the Euler-Zagier Multiple zeta-function $\zeta_n(\vs)$. By letting the variables $s_j$ be positive integers we obtain the so called Multiple Zeta Values \cite{Zagier} which have important connections to knot theory, quantum field theory, the Grothendieck-Teichm{\"u}ller group and the theory of mixed motives. 

We will however take a more analytic approach and consider the multiple zeta-function as a function of $n$ complex variables. Atkinson \cite{Atkinson} first proved the meromorphic continuation to $\C^2$  for $\zeta_2(s_1,s_2)$ in 1949. Around the year 2000, Akiyama-Egami-Tanigawa \cite{AKSHITAN} and Zhao \cite{Zhao} independently proved that the Euler-Zagier multiple zeta-function admits a meromorphic continuation\footnote{ A third proof for a general $n$ has been given by Matsumoto \cite{Matsumoto10}.}
 to $\vs \in \C^n$ for any $n \geq 2$.  Akiyama-Ishikawa \cite{AkIs}  showed that the Multiple-Hurwitz zeta-function for general parameters also admits a meromorphic continuation to $\C^n$.  Several authors \cite{Mat10},\cite{MuSi},\cite{KeMa},\cite{MeVi}   have later given different proofs of this result. 

Nakamura \cite{Nakamura} proved Voronin universality for the Multiple Hurwitz zeta-function with respect to the  variable $s_n$ if Re$(s_{n-1})>3/2$ and Re$(s_j)>1$ for $1 \leq j \leq n-2$ under the assumption that the  parameters $\alpha_j$ for $j=1,\ldots,n$ are algebraically independent. Nakamura-Pa{\'n}kowski  \cite[Theorem 4.1, Theorem 4.3]{NakPan} used hybrid universality to prove the same result without assuming that the parameters are algebraically independent, but still assuming $0<\alpha_n \leq 1$, $\alpha \neq \frac 1 2$ is transcendental or rational.   Furthermore Nakamura-Pa{\'n}kowski  \cite[Theorem 4.5] {NakPan} proved that  $\zeta_n(s,\ldots,s;\alpha,\ldots,\alpha)$ is  universal in one complex variable whenever $0<\alpha \leq 1$ is transcendental or rational $\alpha \neq \frac 1 2$. 

\subsection{Voronin universality in $\C^n$}
When it comes to universality in several complex variables much less has been done.  Understanding what a proper $n$-dimensional generalization of the Voronin universality theorem might not even be obvious but a natural definition and proving the existence of a universal function in several complex variables is in fact not too difficult.

This however is something far from giving an explicit example of such a function as a Dirichlet series, and no such example is hitherto known.  With regards to existence of universal functions in $\C$ it was Birkhoff \cite{Birkhoff} who found this already in 1929, predating Voronin's explicit example of a universal function, the Riemann zeta-function, by 46 years, and this is the proof that readily generalizes to several complex variables\footnote{The author is grateful to Paul M. Gauthier for discussions about this.}.

Since the multiple Hurwitz zeta-function  is a function of $n$ variables it is natural to ask   whether this function might be universal in 
$\C^n$, and indeed Matsumoto\footnote{We also thought about this question independently of Matsumoto in 2010, and only found Matsumoto's paper when doing a literature review of  the field} asked this question as one of several open problems in \cite{Matsumoto}.

\subsection{Main results}

 We will be able to answer this question whenever $\alpha_j$ are  transcendental or rational numbers. The proof of the rational case is substantially harder than the transcendental case as we will explain later.
 Our main result is the following Theorem
 \begin{thm} \label{TH1}
  Let $n \geq 2$ and $E \subseteq D^n$ be a Runge domain,  $\val=(\alpha_1,\ldots,\alpha_n)$, where $\alpha_j>0$ for $j=1,\ldots,n$ are transcendental or rational numbers and $f$ be any holomorphic function on $E$. Then for any $\varepsilon>0$, and  compact subset $K \subset E$ we have that
\begin{gather*}
 \liminf_{T \to \infty} \frac 1 {T^n} \meas  \left \{\vt \in [0,T]^n:\max_{\vs \in K} 
\abs{\zeta_n(\vs+i\vt; \val)-f(\vs)}<\varepsilon \right \}>0,
\end{gather*}
where either $\alpha_n$ or $\alpha_{n-1}$ is transcendental, the condition  
\begin{gather} \label{Kkond}
 \min_{\vs \in K} \Re \p{s_{n-1}+s_n}>\frac 3 2 
\end{gather}
holds or $\alpha_n=\frac c d$ and there exist some $q$  coprime to $d$  and non principal character $\chi$ of modulus $q$  such that the Riemann hypothesis is true for the Dirichlet $L$-functions $L(s,\chi \chi^*)$ where $\chi^*$ is any  Dirichlet character mod $d$.
 \end{thm}
When both $\alpha_n$ and $\alpha_{n-1}$ are rational we get a smaller region where we have universality\footnote{This is a fact that we missed at previous presentations of our main result.} unless we assume the Riemann hypothesis for some Dirichlet $L$-functions. For example for $n=2$ and $\alpha_1=\alpha_2=1$ which is the Euler-Zagier case we are sure to get universality in the region where $3/4 <\Re(s_j)<1$ rather than $1/2<\Re(s_j)<1$ for $j=1,2$.  More generally, an immediate consequence (letting $\alpha_j=1$) is the following
\begin{cor}
\label{COR1}
  Let $n \geq 2$ and $E \subseteq D^n$ be a Runge domain  and $f$ be any holomorphic function on $E$. Then for any $\varepsilon>0$, and  compact subset $K \subset E$  that satisfies the condition \eqref{Kkond} we have that
\begin{gather*}
 \liminf_{T \to \infty} \frac 1 {T^n} \meas  \left \{\vt \in [0,T]^n:\max_{\vs \in K} 
\abs{\zeta_n(\vs+i\vt)-f(\vs)}<\varepsilon \right \}>0,
\end{gather*}
where $\newsym{The Euler-Zagier multiple zeta-function}{\zeta_n(\vs)}$ is the Euler-Zagier multiple zeta-function. If the  Riemann hypothesis is true for at least one Dirichlet $L$-function it is no longer necessary to assume the condition \eqref{Kkond} on the set $K$. 
\end{cor} 
 Since we do not assume anything about the function being zero-free when $n \geq 2$ these results implies strong results on the zero-sets of the multiple zeta-functions. In one variable it follows from applying the Voronin universality theorem with functions $f$ that has zeros in the interior of the set $K$ and Rouche's theorem that the Hurwitz zeta-function with a transcendental parameter $\alpha$ or rational $\alpha \not \in \{1/2 , 1\}$ has $\gg T$ zeros in any strip $1/2<\sigma_1 <\Re(s)<\sigma_2<1$ where $|\Im(s)|<T$.  By applying Theorem \ref{TH1} on a function $f$ with zeros in the interior of $K$ we get a corresponding result in several complex variables.
Since zeros of functions in several complex variables are not isolated we may use the area measure of the zero-set, rather than the simple counting measure.
\begin{cor}  \label{korro2}
  Let $\val=(\alpha_1,\ldots,\alpha_n)$ where $\alpha_j$ are transcendental or rational. Let $\vsigma^{\pm}=(\sigma_1^{\pm},\ldots,\sigma_n^{\pm}) \in \R^n$ where $\frac 1 2 \leq \sigma_j^-<\sigma_j^+ \leq1$ for $j=1,\ldots,n$,  and let 
\begin{gather*}
  N_\val(\vsigma^-, \vsigma^+ ;T)= \mu \{\vs \in \C^n: \zeta(\vs;\val)=0, \, 0<\Im(s_j)<T, \,  \sigma_j^-<\Re(s_j)< \sigma_j^+ \}, 
 \end{gather*}
where $\mu$ denote the area measure\footnote{Hausdorff measure of real dimension $2n-2$ when $\C^n$ is identified with $\R^{2n}$.} of the zero-set. Then
\begin{gather*} 
    N_\val(\vsigma^-, \vsigma ^+ ;T) \gg T^n, 
  \end{gather*} 
  where unless we assume the Generalized Riemann Hypothesis or that either $\alpha_n$ or $\alpha_{n-1}$ is transcendental we also need to assume that $\sigma_{n-1}^+ + \sigma_{n}^+  > \frac 3 2$.
\end{cor}
\begin{proof}
Let $\delta:=\frac 1 2 \min_{1 \leq j \leq n}  (\sigma_{j}^+- \sigma_j^-),  \, \sigma_j:= \frac 1 2 (\sigma_j^+ + \sigma_j^-)$ and let  $K \subset  \C^n$ be the closed ball with center at $(\sigma_1,\ldots,\sigma_n)$ and radius $\delta$. Then the result follows by applying Theorem \ref{TH1} on the function  $f(\vz)=(2z_1-2\sigma_1-\delta)/\delta$  and a version of Jensen's formula \cite[Formula (12), p. 385] {Rudin2} for $\C^n$.
\end{proof} 
In particular it is clear that no analogue of the Riemann hypothesis\footnote{This also follows from the one variable universality results of Nakamura-Pa{\'n}kowski \cite{NakPan};  For the Euler-Zagier case see also  Nakamura-Pa{\'n}kowski \cite[Theorem 3.3]{NakPan2} for a clear statement.} holds. In another direction, in \cite{Andersson2} 
 we noticed that if in Voronin's theorem we assume that the set $K$ has empty interior then we do no longer need to assume that the function is zero-free on $K$. Also the condition that $f$ is analytic in the interior is  trivially true so it is sufficient to assume that $f$ is continuous on $K$. In particular it is sufficient to use Lavrentiev's theorem rather than Mergelyan's theorem in the proof of universality. While we do not have a zero-free condition in Theorem \ref{TH1} we may still use what corresponds to Lavrientev's theorem in several complex variable to simplify the statement so that we may approximate any function that is continuous on $K$. In several complex variables it is known by a result of Harvey-Wells\footnote{This generalization of the Stone Weierstrass theorem is a sharper version of a result of Hörmander-Wermer \cite{HWe}, see discussion in \cite[section 8.]{Levenberg}, or \cite{berndtsson} for a different proof.}  \cite{HW} that if $E$  is a {\em totally real\footnote{For the definition see \cite[p.115]{Levenberg}.}}  submanifold of class $C^1$ in an open set in $\C^n$ then for any compact $K \subset E$ and any function $f$ continuous on $K$ may be uniformly approximated by polynomials. An immediate consequence of Theorem \ref{TH1} and the Harvey-Wells theorem is the following.
\begin{cor} \label{kkorro}
  Let $n \geq 2$ and $E \subset D^n$ be a totally real submanifold of class $C^1$, $\val=(\alpha_1,\ldots,\alpha_n)$, where $\alpha_j>0$ for $j=1,\ldots,n$ are transcendental or rational numbers. Then for any compact subset $K \subset E$,  continuous function $f$ on $K$ and  $\varepsilon>0$ we have that
\begin{gather}
 \liminf_{T \to \infty} \frac 1 {T^n} \meas  \left \{\vt \in [0,T]^n:\max_{\vs \in K} 
\abs{\zeta_n(\vs+i\vt; \val)-f(\vs)}<\varepsilon \right \}>0,
\end{gather}
where either $\alpha_n$ or $\alpha_{n-1}$ is transcendental, the condition   \eqref{Kkond} holds or $\alpha_n=\frac c d$ and there exist some $q$  coprime to $d$  and non principal character $\chi$ of modulus $q$  such that the Riemann hypothesis is true for the Dirichlet $L$-functions $L(s,\chi \chi^*)$ where $\chi^*$ is any  Dirichlet character mod $d$.
 \end{cor}
In particular we may choose a compact set $M \subset R^n$ and $K=(\sigma_1,\ldots,\sigma_n)+i M$ in Corollary \ref{kkorro}. If we also for simplicity consider only the Euler-Zagier case we get the following special case:
\begin{cor}
\label{kkorro2}
  Let $n \geq 2$, $\vsigma=(\sigma_1,\ldots,\sigma_n) \in (\frac 1 2,1)^n$ where $\sigma_{n-1}+\sigma_n>3/2$ and let $M \subset \R^n$ be a compact set. Then for any continuous function $f:M \to \C$  and  $\varepsilon>0$  we have that
\begin{gather}
 \liminf_{T \to \infty} \frac 1 {T^n} \meas  \left \{\vt \in [0,T]^n:\max_{\vx \in M} 
\abs{\zeta_n(\vsigma+i (\vt+\vx))-f(\vx)}<\varepsilon \right \}>0.
\end{gather}
\end{cor} 

Finally, there are some new types of universality that occurs in several variables.% We  state the theorems for the continuous  case %Later we will state the corresponding joint universality version,
%but  it should be noted that it is also possible to prove Discrete universality versions of the following results. 
 We will prove two versions of  such a theorem, although certainly it should be possible to prove other versions of these results also. Let $\overline \Q^+$ denote the set of positive algebraic  numbers.  
 \begin{thm} \label{TH4}
  Let $n \geq 2$,  and $1\leq m < n$. Let $\vv_1,\ldots,\vv_m \in (\overline{\mathbb Q}^+)^n$  be linearly independent vectors such that for each $1 \leq j \leq m$ we have that
$$
 \vv_j=(v_{1,j},\ldots,v_{n,j})
$$  
  and that for each $j$ then $v_{1,j},\ldots,v_{n,j}$  are linearly independent over $\Q$   and $E \subseteq \{s \in \C: 1-m/(2n)<\Re(s)<1 \}^n$ be a Runge domain,  $\val=(\alpha_1,\ldots,\alpha_n)$,
 where  $\alpha_j>0$ are rational and $f$ is any holomorphic function on $E$. Then for any $\varepsilon>0$, and  compact subset $K \subset E$ that satisfies \eqref{Kkond} we have that
\begin{gather*}
 \liminf_{T \to \infty} \frac 1 {T^m} \meas  \left \{\vt \in [0,T]^m:\max_{\vs \in K} 
\abs{\zeta_n \left(\vs+i \left(\sum_{j=1}^m t_j \vv_j \right); \val \right)-f(\vs)}<\varepsilon \right \}>0.
\end{gather*}
 \end{thm}
 The condition that $v_{1,j},\ldots,v_{n,j}$ are linearly independent over $\Q$ is essential. However the condition that it is sufficient that they are algebraic follows from using Baker \cite[Theorem 2.1]{baker} and is likely not essential, since for example Schanuel's conjecture is sufficient to remove this condition\footnote{removing this condition is likely  much simpler than
 proving Schanuel's conjecture}. Also it is likely possible to allow transcendental parameters. The most interesting case might be $m=1$, and we state this case for the Euler-Zagier  multiple zeta-function as a Corollary.
 \begin{cor} \label{cor4}
  Let $n \geq 2$,  and  let $\vv=(v_1,\ldots,v_n) \in  (\overline{\Q}^+)^n$ be a vector such that its components $v_1,\ldots,v_n$  are linearly independent over $\Q$ and let $E \subseteq \{s \in \C: 1-1/(2n)<\Re(s)<1 \}^n$ be a Runge domain,  and $f$ be a holomorphic function on $E$. Then for any $\varepsilon>0$ and  compact subset $K \subset E$ we have that
\begin{gather*}
 \liminf_{T \to \infty} \frac 1 {T} \meas  \left \{0 \leq t \leq T:\max_{\vs \in K} 
\abs{\zeta_n(\vs+i \vv t) -f(\vs)}<\varepsilon \right \}>0.
\end{gather*}
 \end{cor}
  We  remark that we get a smaller region where universality holds than  the natural region which should still be $D^n$ like it should be in Theorem \ref{TH1}. To prove universality for the stronger region would not only need the Riemann hypothesis for some Dirichlet $L$-function, but also; for all choices of parameters; better mean square results for the multiple Hurwitz zeta-functions than we manage to prove in Section \ref{sec5}. 
   
   The condition that the components of the vector are linearly independent over the rationals can be replaced by some condition on the parameters $\alpha_j$ and we give an example of such a result. 
  \begin{thm} \label{TH5}
  Let $n \geq 2$,  and $1\leq m < n$. Let $\vv_1,\ldots,\vv_m \in (\Q^+)^n$  be linearly independent vectors and $E \subseteq \{s \in \C: 1-m/(2n)<\Re(s)<1 \}^n$ be a Runge domain, and let  $\val=(\alpha_1,\ldots,\alpha_n)$, where the $\alpha_j$'s are algebraically independent with the exception for at most one rational $\alpha_j$ and $f$ be any holomorphic function on $E$. Then for any $\varepsilon>0$, and  compact subset $K \subset E$ we have that
\begin{gather*}
 \liminf_{T \to \infty} \frac 1 {T^m} \meas  \left \{\vt \in [0,T]^m:\max_{\vs \in K} 
\abs{\zeta_n \left(\vs+i \left(\sum_{j=1}^m t_j \vv_j\right); \val \right)-f(\vs)}<\varepsilon \right \}>0.
\end{gather*}
 \end{thm}  
   Since the conditions of the theorem implies that at least one of the last two parameters are transcendental we no longer need to assume the Riemann hypothesis for some Dirichlet $L$-function in order to obtain a better region of universality. However we still need better mean square results than we are able to prove.  
For example the mean square formula 
\begin{gather*}
  \int_0^T \abs{\zeta_n \p{\frac 1 2 + it v_1,\ldots,\frac 1 2 +it v_n;\val}}^2 dt  \ll_{\varepsilon,\val,{\bf v}} T^{1+\varepsilon}.  \qquad ({\bf v},\val \in (\R^+)^n)
\end{gather*}  
would be sufficient\footnote{A suitable version of the Lindel\"of hypothesis for multiple zeta-functions would certainly also be sufficient. However from an example  of  Kiuchy-Tanagawa-Zhai \cite{KTZ}, see \cite[Remark 1.4]{MatHir} it is not so clear what the correct version of the Lindel\"of hypothesis should be.}  
to prove Theorem \ref{TH5}; and Theorem \ref{TH4} if we also assume GRH; for the full region $D^n$. 

Another consequence of Theorem \ref{TH4} and Theorem \ref{TH5} is a one variable universality result of a more classical flavor
\begin{cor}
    Let $\vv=(v_1,\ldots, v_n) \in (\R^+)^n$, ${\bm w}=(w_1,\ldots,w_n) \in \C^n$ and $\val=(\alpha_1,\ldots,\alpha_n) \in (\R^+)^n$ be such that either $\alpha_j$ are rational,  $v_j$ are algebraic and such that $\{v_1,\ldots,v_n\}$ is linearly independent over $\Q$, or $v_j$ are rational and $\{\alpha_1,\ldots,\alpha_n\}$ are algebraically independent with the possibly exeption of one rational $\alpha_j$. Let  
 $K \subset \bigcap_{j=1}^n \{s \in \C: 1-\frac 1 {2n} <\Re(v_j s+w_j) <1 \}$ be a compact set with connected complement and suppose that $f$ is any continuous function on $K$ that is analytic  in the interior of $K$. Then 
\begin{gather}\liminf_{T \to \infty} \frac 1 T \mathop{\rm meas} \left \{t \in [0,T]:\max_{s \in K} \abs{\zeta_n( {\vv} (s+i t)+{\bm w};\val)-f(s)}<\varepsilon \right \}>0. \end{gather}
\end{cor}
  Nakamura-Pa{\'n}kowski \cite[Theorem 5]{NakPan} proved the corresponding result for $\vv=(1,\ldots,1)$, ${\bm w}= \boldsymbol{0}$ and $\val=(\alpha,\ldots,\alpha)$ with transcendental or rational $\alpha$ which is not covered by the theorem above\footnote{Furthermore they obtain the result in the full range $1/2<\Re(s)<1$.}.  Their proofs use hybrid one variable universality results for the Hurwitz zeta-function and combinatorial relations of type
\begin{align}
   \zeta_2(s,s;\alpha,\alpha)=&\frac 1 2 \zeta(s;\alpha)-\frac 1 2 \zeta(2s;\alpha), \\ \zeta_3(s,s,s;\alpha,\alpha,\alpha)=&\frac 1 6 \zeta(s,\alpha)^3-\frac 1 2 \zeta(s,\alpha)\zeta(2s,\alpha)+\frac 1 3 \zeta(3s,\alpha),  % \zeta_4(s,s,s,s;\alpha,\alpha,\alpha,\alpha)=& \cdots
    \\ &\vdots
\end{align}
An application of Rouche's theorem gives us
\begin{cor}
 Let  $\vv=(v_1,\ldots, v_n) \in (\R^+)^n$, ${\bm w} =(w_1,\ldots,w_n) \in \R^n$ and $\val=(\alpha_1,\ldots,\alpha_n) \in (\R^+)^n$ be such that either $\alpha_j$ are rational,  $v_j$ are algebraic and  such that $\{v_1,\ldots,v_n\}$ is linearly independent over $\Q$, or $v_j$ are rational and $\{\alpha_1,\ldots,\alpha_n\}$ are algebraically independent with the possibly exeption of one rational $\alpha_j$. Then there exists $\gg T$ zeros to the function $f(s)=\zeta_n( s \vv +{\bm w};\val)$ in the set $\cap_{j=1}^n \{s=\sigma+it: 1-\frac 1 {2n} <\sigma_1< v_j \sigma + w_j<\sigma_2 <1, |t_j|<T\}$ whenever the last set is non empty. 
\end{cor}

\section{Some one variable preliminaries}

\subsection{Rational and transcendental parameters}
In this paper we define 
\begin{gather} \label{Kdef}
 \newsym{$\{z \in \C: 1/2+\ddd \leq \Re(z) \leq 1- \ddd, -R \leq \Im(z) \leq R \}$}{\M} =\{z \in \C: 1/2+\ddd \leq \Re(z) \leq 1- \ddd, -R \leq \Im(z) \leq R \}, %\qquad (\ddd, R>0), 
\end{gather}
where $0<\newsym{Distance of the set $\M$ to the critical line}{\ddd}<\frac 1 4$ and $ \newsym{Bounds for the imaginary part of the set $\M$}{R}>0$.
By approximation theorems from complex analysis it is sufficient to prove universality on $\M$ in one variable (Mergelyan's theorem) and $\M^n$ in several variables (Oka-Weil's theorem). In classical proofs of Voronin Universality for the Hurwitz zeta-function of a transcendental parameter the first step is to prove that the Dirichlet polynomial
\begin{gather} \label{ep1}
	 \sum_{k=0}^N a(k+\alpha) (k+\alpha)^{-s}, \qquad |a(k)|=1 
	  \end{gather}
	can approximate any polynomial function $p(s)$ on $\M$. 
In the case of the Riemann zeta-function  we may instead  consider approximations by the finite Euler products
\begin{gather} \label{ep2} \prod_{p \leq N} (1- a(p) p^{-s})^{-1}, \qquad |a(p)|=1.
\end{gather}
 These steps use the celebrated Pechersky rearrangement theorem. The fact that the sets $\{\log(k+\alpha): k \in \N\}$ when $\alpha$ is transcendental and $\{\log p: p \text { prime}\}$ are linearly independent over $\Q$  allows us to use the Kronecker theorem to show that the coefficients $a(k+\alpha)$ and $a(p)$ respectively may be chosen as $1$ if at the same time we replace $s$ by $s+i\tau$ for some real $\tau$. More precisely we  use an equidistribution theorem of Weyl, which in this context may be viewed as an effective version of the Kronecker theorem since it also gives a measure for such $\tau$.
 
   It will turn out useful to not consider Dirichlet polynomial/Dirichlet product truncations properly, but rather the Dirichlet series whose coefficients belong to a special class  of multiplicative unimodular functions. Since completely multiplicative functions are often defined just on positive integers we will first define them in a more general manner, useful for our purposes
\begin{defn}
  Let $X \subset \R^{+}$. We say that a function $a:X \to \C$ is completely multiplicative if for all $x_1,\ldots,x_n \in X$ and $y_1, \ldots,y_m \in X$ such that
$$ \prod_{j=1}^n x_j= \prod_{j=1}^m y_j,$$
then
$$ \prod_{j=1}^n a(x_j)=\prod_{j=1}^m a(y_j).$$
We say that $a$ is unimodular if $|a(x)|=1$ for all $x \in X$.
\end{defn}
We are now ready to define the class of arithmetic function whose generating Dirichlet series will replace the Dirichlet polynomial/Dirichlet product truncations in our proof
\begin{defn}
  We say that a function $a:\N+\alpha \to \C$ is of \newsym{A class of unimodular multiplicative functions}{type $(N,\chi)$} and that $\chi$ is of modulus $q$   if $a$   is a  completely multiplicative unimodular function on $\N+\alpha$ such that
   \begin{enumerate}
    \item If $\alpha
=c/d$ for $(d,c)=1$ is rational then when the function is extended to a completely multiplicative unimodular function on $(\N+\alpha)\cup \N$ then $a(p)=\chi(p)$,
   for primes $p \geq \log N$ and where $\chi$ must be a non-principal Dirichlet character mod $q$ such that $(q,d)=1$ and  such that $L(s,\chi \chi^*)$ is zero-free on $\M$ for any character $\chi^*$ mod $d$.  
   \item If $\alpha$ is transcendental then  $a(k+\alpha)=\chi(k)$ for $k \geq  \log N$ where $\chi(k)=e^{2 \pi i k/q}$ and where $q \geq 2$ is an integer.
    \end{enumerate}
    We say that $\chi$ is permissible with respect to $\alpha$ if the conditions above on $\chi$ are satisfied.
\end{defn}
We remark that the rational case implies that the function $a$ as defined on the integers has a generating Dirichlet series 
\begin{gather} \label{ojkond}
\sum_{k=1}^\infty a(k) k^{-s}= \prod_{p<\log N_1} \p{\frac{1-\chi(p) p^{-s}}{1-a(p) p^{-s}}} L(s,\chi), 
\end{gather}
and furthermore
\begin{gather} \label{ojkond2}
  a(k+\alpha)=a \p{k+\frac c d}=\frac{a(kd+c)}{a(d)}.
\end{gather}
We will prove the following result in Section \ref{secLemma1} which will be an important ingredient in the proof our  fundamental Lemma, Lemma \ref{LE4} which will allow us to prove our main theorem.
\begin{lem} \label{LE3}
Assuming that  $\alpha>0$ is rational or transcendental there exist some $A \geq 0$  such that given any $\varepsilon>0$, $N_0 \in \N$,  and compact set 
$\M$ defined by \eqref{Kdef}, function $a:\N+\alpha \to \C$ of type $(N_0,\chi)$
and any polynomial $p$ with $\min_{s \in \M} |p(s)| \geq A$  there exist some $N_1>N_0$  such that the function $a(k+\alpha)$ can be redefined  for $k \geq N_0$  to a function of type $(N_1,\chi)$
 such that for this redefinition and some $N^* \geq N_1$ we have that
 \begin{gather}  %Changed to be similar to Lemma 7
   \label{ao}   \sup_{N \geq N^*} \max_{s \in \M}\abs{ \sum_{k=0}^{N} \frac{a(k+\alpha)}{(k+\alpha)^{s}}-p(s)}<\varepsilon.
 \end{gather}
\end{lem}
We remark that the conditions on $A$ is to allow  Lemma \ref{LE3} to be true for any positive $\alpha=n+\halv$ and $\alpha=n$ for integers $n$. 
For all other rational and transcendental parameters $\alpha$ we may choose $A=0$.

\section{Sketch of the proof of Theorem \ref{TH1} - The simplest case}

We are going to sketch the  proof of the simplest case of Theorem \ref{TH1}, when $n=2$ and both parameters are transcendental. Substantial new difficulties occur when the parameters are rational, but having a rough understanding of the special case first should hopefully be useful when studying the general case.  By the Oka-Weil  theorem it follows that it is sufficient to consider $f$ a polynomial, and instead of the set $K$ we may choose the set $\M^2$ where $\M$ is defined by \eqref{Kdef}.

 \subsection{A Dirichlet polynomial approximation lemma}
 The first step in proving Voronin universality is proving that the function can be approximated by some Dirichlet polynomial. In one variable the following lemma is crucial\footnote{When we also allow rational parameters we need to replace this by Lemma \ref{LE3}.}.
\begin{lem} \label{onevar}
  Assume that $\alpha,N_0>0$, $p$ is a polynomial and $\varepsilon>0$. Then there exist an $N_1 > N_0$ and unimodular numbers $|a(k+\alpha)|=1$ for $N_0 \leq k \leq N_1$ such that
\[
  \max_{s \in \M} \abs{\sum_{k=N_0+1}^{N_1} \frac{a(k+\alpha)}{ (k+\alpha)^{s}}-p(s)}<\varepsilon.
\] 
\end{lem}
For $N_0=0$ this is exactly the Dirichlet polynomial approximation lemma we need to prove the Voronin universality theorem for the Hurwitz zeta-function with a transcendental parameter. It is a consequence of the celebrated Pechersky rearrangement theorem \cite{Pech} and the proof is the same for a general $N_0$ as for $N_0=0$ since throwing away a finite number of terms do not matter. Let us define
\begin{gather}
 \zeta_{\va}^{[N]}(\vs;\val)=\sum_{0 \leq k_1<k_2 \leq  N} \prod_{j=1}^2 \frac{a_j(k_j+\alpha_j)}{(k_j+\alpha_j)^{s_j}}
\end{gather}
 In two variables,  instead of Lemma \ref{onevar}, we need the following lemma.
\begin{lem} \label{prelem}
 Suppose that $p$ is a polynomial in two complex variables. Then given any $\varepsilon>0$ there exist some integer $N$ and coefficients $|a_j(k+\alpha_j)|=1$ for $0 \leq k \leq N$ and $j=1,2$  such that
 \[ 
  \max_{\vs \in \M^2} \abs {  \zeta_{\va}^{[N]}(\vs;\val)  -p(\vs)}<\varepsilon.
\] 
\end{lem}
More generally, for rational parameters we need the corresponding result when the coefficients $a_j: \N +\alpha_j \to \C$ are  completely multiplicative unimodular functions of type $(N_0,\chi)$ for some $N_0>0$, see our fundamental Lemma (Lemma \ref{LE4}, p. \pageref{LE4}) . 
The case of algebraic irrational parameters are currently far out of reach.

Before we sketch the proof of Lemma \ref{prelem} we may ask the question if we can find some corresponding result when the Dirichlet polynomial is replaced by a polynomial, and indeed we can and it is an important ingredient for the proof of the Dirichlet polynomial analogue
\begin{lem} \label{pla}
 Let $q(s_1,s_2)$ be a polynomial in two complex variables. Then there exists an integer $M \geq 0$, a sequence $\{j_m\}_{m=1}^M$ with $j_m \in \{1,2\}$ and monomials $q_{j,m}(s_j)$ for $j=1,2$ and $1 \leq m \leq M$  such that
 \begin{gather}
  \sum_{1 \leq m_1 < m_2  \leq M}  q_{1,m_1}(s_1) q_{2,m_2}(s_2)=q(s_1,s_2), \label{cd1}
 \\ \intertext{and such that $q_{j,m}(s_j)=0$ if $j \neq j_m$, and the following condition holds}
 \sum_{1 \leq m_2 \leq M} q_{2,m_2}(s_2)=0. \label{cd2}
\end{gather}
\end{lem}

While we will not sketch the proof of this result which is proved in a more general version in Lemma \ref{LE13} we  illustrate it by the  example given in Table \ref{exmpl} 
\begin{table}[h]
\begin{center}
\begin{tabular}{|c|c|c|c|c|c|c|c|c|c|c|c|}
\hline
$m$              & $1$     & 2  & 3     & 4        & 5 & 6 & 7 & 8 & 9  & $10$ & $11$     \\ \hline 
$q_{1,m}(s_1)$ & 0 & $1$ & 0 & $-1$     &  $0$  & $s_1$ &    $0$      & $-s_1$ & $0$  & $s_1^2$ & $0$ \\[2pt]  \hline
$q_{2,m}(s_2)$ & $-s_2$ & $0$ & $s_2$  & $0$ & $ - 1$ & $0$ & $ 1 $  & $0$ & $ -s_2^2$ & $0$ & $s_2^2$ \\[2pt]  \hline
\end{tabular} 
\caption{$q(s_1,s_2)=s_2+s_1+s_1^2 s_2^2$}
\label{exmpl}
\end{center}
\end{table}
and we leave it as an exercise to the reader to find a proof of the general result\footnote{or (s)he might look up Lemma \ref{LE13}.} by studying this example. We now proceed to sketch the proof of Lemma \ref{prelem}. First choose $\NC_0$ sufficiently large. For small values of $0 \leq k \leq \NC_0$ we may define $a_{j}(k+\alpha_j)=(-1)^{k}$ for $j=1,2$.  Then it follows that
\begin{gather}
     \sum_{0 \leq k_1<k_2 \leq \NC_0}  \prod_{j=1}^2 \frac{a_j(k_j+\alpha_j)}   {(k_j+\alpha_j)^{s_j}}     \approx p_2(s_1,s_2), \qquad (s_1,s_2) \in \M^2,
\end{gather}
for some polynomial $p_2$. To prove Lemma \ref{prelem} we need to find  $a_j(k+\alpha_j)$ for $\NC_0<k \leq N$ such that 
\begin{gather}
   \left \{ \sum_{k_1 \leq \NC_0}  \sum_{ \NC_0<k_2 \leq N}  + \sum_{\NC_0 < k_1<k_2 \leq N} \right \}  \prod_{j=1}^2 \frac{a_j(k_j+\alpha_j)}{ (k_j+\alpha_j)^{s_j}}   \approx q(s_1,s_2), 
\end{gather}
where $q(s_1,s_2)=p(s_1,s_2)-p_2(s_1,s_2)$. For this it is sufficient to find $a_j(k+\alpha_j)$ for $\NC_0 < k \leq N$ so that 
\begin{gather} \label{rer}
 \sum_{\NC_0 < k_1<k_2 \leq N}   \prod_{j=1}^2 \frac{a_j(k_j+\alpha_j)} {(k_j+\alpha_j)^{s_j}}   \approx q(s_1,s_2),
\end{gather}
and
\begin{gather} \label{rer2}
 \sum_{\NC_0 <k_2 \leq N} \frac {a_2(k_2+\alpha_2)} {(k_2+\alpha_2)^{s_2}} \approx 0, \qquad \abs{\sum_{0 \leq k_1\leq  \NC_0}  \frac{a_j(k_1+\alpha_j)}{(k_1+\alpha_j)^{s_1}}} \ll 1.
\end{gather}
To prove \eqref{rer} and \eqref{rer2} we first use Lemma \ref{pla} on the polynomial $q(s_1,s_2)$ and after that recursively for $m=1,\ldots,M$ use  Lemma \ref{onevar}   with $N_0=\NC_{m-1}$, $\alpha=\alpha_{j_m}$ and $p(s)=q_{j_m,m}(s)$ and define $\NC_m$ by the $N_1$ that comes from that application of Lemma \ref{onevar}  so that 
\begin{gather}  \label{at3a}
 \sum_{\NC_{m-1}<k \leq \NC_{m} } \frac{a_{j_m}(k+\alpha_{j_m})} {(k+\alpha_{j_m})^{s_j}} \approx q_{{j_m},m}(s_j).  
\end{gather}
If we furthermore define 
 \begin{gather} 
 \label{ttt}
  a_j(k+\alpha_j) =(-1)^k, \qquad (\NC_{m-1} < k \leq \NC_{m}, \, j \neq j_m),
\end{gather}
it  follows that \eqref{at3a} holds also for $j \neq j_m$. Let us finally define $N=\NC_{M}$. The approximation \eqref{rer2} follows from \eqref{cd2} together with \eqref{at3a}. We now rewrite the sum in \eqref{rer} as
  \begin{multline} \label{yu2}
  \sum_{\NC_0 < k_1<k_2 \leq N} * \,  = 
   \sum_{1 \leq m_1 <m_2 \leq M} \sum_{\NC_{m_1-1} < k_1 \leq \NC_{m_1}}  \sum_{\NC_{m_2-1} < k_2 \leq \NC_{m_2}} * \, \,\,   + \, \sum_{m=1}^{M} \sum_{\NC_{m-1} < k_1<k_2 \leq \NC_{m}} *
\end{multline} 
It follows that the first sum on the right hand side is approximately equal to $q(s_1,s_2)$ by \eqref{cd1} and \eqref{at3a}. Equation \eqref{ttt} is sufficient to show that  the second sum of the right hand side of \eqref{yu2} is approximately zero. 
If $q$ is given in Table \ref{exmpl} we can illustrate the summation by the following table:
\begin{center}
\begin{tabular}{|c|c|c|c|c|c|c|c|c|c|c|c|c|}
\hline
$m$              &  $1$  & $2$     & $3$        & $4$ & $5$  & $6$     & $7$  & $8$     & $9$        & $10$ & $11$      \\ \hline 
$ \sum\limits_{\NC_{m-1} < k \leq \NC_{m}} \frac{a_1(k+\alpha_1)}{(k+\alpha_2)^{s_1} } \approx $ &  0 & 1  & 0   &  $-1$  & 0 & $s_1$ &    0    &$-s_1$ &0  & $s_1^2$ & 0 \\ \hline
$\sum\limits_{\NC_{m-1} <k  \leq \NC_{m}} \frac{a_2(k+\alpha_2)}{(k+\alpha_2)^{s_2}} \approx $  & $-s_2$ & 0 & $s_2$ & 0 & $-1$ & 0 & $1$ & 0     & $-s_2^2$ & 0 & $s_2^2$  \\ \hline
\end{tabular} 
\end{center}

\subsection{Proceeding to Voronin universality}
We define the truncation of the double Hurwitz zeta-function series as
$$ \zeta^{[N]}(s_1,s_2;\alpha_1,\alpha_2) =\sum_{0 \leq k_1<k_2 \leq N} (k_1+\alpha_1)^{-s}(k_2+\alpha_2)^{-s}$$
A consequence of
Kronecker's approximation theorem and an equidistribution theorem of Weyl; following classical one-dimensional universality ideas; is the following result:
\begin{lem} \label{LE6simple}
 Let $\val=(\alpha_1,\alpha_2)$ where $\alpha_1,\alpha_2>0$ are transcendental and let $p$ be a polynomial in two complex variables. Then
 \begin{gather*}
     \liminf_{N \to \infty} \liminf_{T \to \infty}  \frac 1 {T^2} \operatorname{meas} \left \{\vt \in [0,T]^2: \max_{\vs \in \M^2} \abs{\zeta^{[N]}(\vs+i\vt;\val)-\zeta_\va^{[N]}(\vs;\val)} < \varepsilon  \right \} >0.
  \end{gather*}
	\end{lem}
For the full proof of the general case, Lemma \ref{LE6} see Section \ref{secstar}. A simple fact is that we can approximate the one variable Hurwitz zeta-function by some Dirichlet polynomial truncation
$$
  \zeta(\sigma+it;\alpha) \approx \sum_{n=1}^{\lfloor T \rfloor} (n+\alpha)^{-\sigma-it} \qquad (T^{1-\ddddd}<t<T, \, \, \sigma \geq 1/2)
$$
for some $\ddddd>0$.  Similarly as in the one variable case where  we will be able to prove some corresponding $n$-variable approximation which in two variables is roughly
\begin{gather}
  \zeta(s_1,s_2;\alpha_1,\alpha_2) \approx \zeta^{[T]} (s_1,s_2;\alpha_1,\alpha_2) \qquad (T^{1-\ddddd}<\Im(s_j)<T, \, \Re(s_j) \geq 1/2) \qquad \label{approxsimple} 
\end{gather}
The precise version of this result uses a smooth truncation, rather than a sharp one and can be found as Lemma \ref{LE7}.

By classical methods, squaring and treating the non-diagonal part in an elementary way utilizing the fact that the region lies strictly to the right of the half line in each variable   we can  obtain (for the general version, see Lemma \ref{LE8})
\begin{lem} \label{LE8simple}  
 Let $\ddd,R$ and $\M$ be defined by \eqref{Kdef}  and let $0<\ddddd<1$.  Then there exist some $C>0$ such that if $1 \leq N \leq T$ then
\begin{gather}
    \frac 1 {T^2}\int_{\M^2} \int_{[T^{\ddddd},T]^2} \abs{\zeta^{[N]}(\vs +i\vt;\val)-\zeta^{[T]}(\vs+i\vt;\val)}^2 d\vt d \vs \leq C  N^{-\ddd}. 
\end{gather}
\end{lem}
The proof of Theorem 1 in the simplest case now follows from Lemma \ref{prelem}, Lemma \ref{LE6simple}, Lemma \ref{LE8simple} and \eqref{approxsimple}. For details see the proof at the end of the next section.

\section{The proof of Theorem \ref{TH1} - The general case}

\definecolor{zzttqq}{rgb}{0.27,0.27,0.27}
\definecolor{cqcqcq}{rgb}{0.75,0.75,0.75}

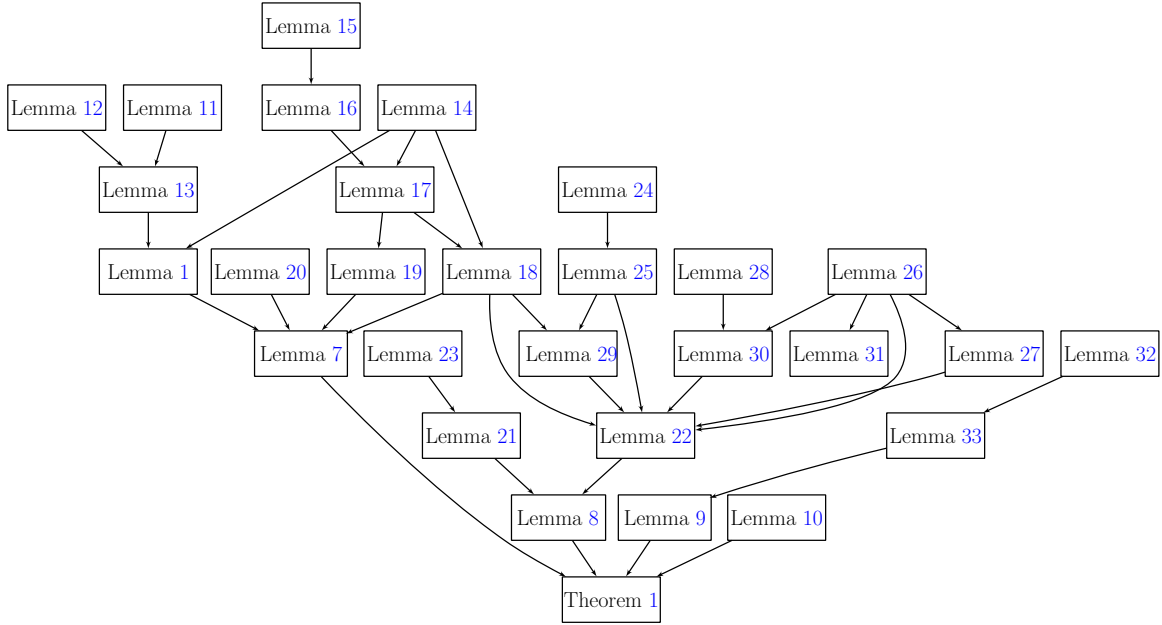
\begin{figure}[h]
\resizebox{\linewidth}{!}{\begin{tikzpicture}[>=latex',line join=bevel,]
  \pgfsetlinewidth{1bp}
\pgfsetcolor{black}
  % Edge: Lemma 9 -> Theorem 1
  \draw [->] (505.02bp,64.785bp) .. controls (500.86bp,58.535bp) and (496.05bp,51.332bp)  .. (485.96bp,36.194bp);
  % Edge: Lemma 13 -> Lemma 1
  \draw [->] (110.5bp,324.78bp) .. controls (110.5bp,319.09bp) and (110.5bp,312.61bp)  .. (110.5bp,296.19bp);
  % Edge: Lemma 25 -> Lemma 22
  \draw [->] (477.24bp,259.64bp) .. controls (479.96bp,251.08bp) and (483.12bp,240.56bp)  .. (485.5bp,231.0bp) .. controls (490.0bp,212.93bp) and (494.01bp,192.36bp)  .. (498.7bp,166.07bp);
  % Edge: Lemma 27 -> Lemma 22
  \draw [->] (736.89bp,198.34bp) .. controls (733.06bp,197.14bp) and (729.22bp,196.01bp)  .. (725.5bp,195.0bp) .. controls (665.84bp,178.8bp) and (595.88bp,165.29bp)  .. (540.29bp,155.44bp);
  % Edge: Lemma 18 -> Lemma 29
  \draw [->] (396.9bp,259.78bp) .. controls (403.11bp,253.26bp) and (410.32bp,245.69bp)  .. (424.12bp,231.19bp);
  % Edge: Lemma 16 -> Lemma 17
  \draw [->] (254.35bp,389.78bp) .. controls (260.36bp,383.26bp) and (267.32bp,375.69bp)  .. (280.67bp,361.19bp);
  % Edge: Lemma 11 -> Lemma 13
  \draw [->] (124.31bp,389.78bp) .. controls (122.55bp,383.95bp) and (120.54bp,377.29bp)  .. (115.69bp,361.19bp);
  % Edge: Lemma 12 -> Lemma 13
  \draw [->] (58.174bp,389.78bp) .. controls (65.873bp,383.05bp) and (74.836bp,375.21bp)  .. (90.849bp,361.19bp);
  % Edge: Lemma 33 -> Lemma 9
  \draw [->] (690.73bp,138.0bp) .. controls (657.99bp,130.09bp) and (609.74bp,117.68bp)  .. (551.67bp,98.599bp);
  % Edge: Lemma 20 -> Lemma 7
  \draw [->] (207.24bp,259.78bp) .. controls (210.28bp,253.81bp) and (213.75bp,246.97bp)  .. (221.77bp,231.19bp);
  % Edge: Lemma 18 -> Lemma 7
  \draw [->] (341.87bp,260.78bp) .. controls (321.47bp,252.21bp) and (296.33bp,241.65bp)  .. (265.66bp,228.77bp);
  % Edge: Lemma 26 -> Lemma 30
  \draw [->] (650.76bp,259.95bp) .. controls (636.5bp,252.53bp) and (619.61bp,243.74bp)  .. (595.27bp,231.06bp);
  % Edge: Lemma 22 -> Lemma 8
  \draw [->] (482.65bp,129.78bp) .. controls (475.34bp,123.12bp) and (466.86bp,115.37bp)  .. (451.33bp,101.19bp);
  % Edge: Lemma 28 -> Lemma 30
  \draw [->] (562.5bp,259.78bp) .. controls (562.5bp,254.09bp) and (562.5bp,247.61bp)  .. (562.5bp,231.19bp);
  % Edge: Lemma 23 -> Lemma 21
  \draw [->] (331.07bp,194.78bp) .. controls (335.63bp,188.53bp) and (340.89bp,181.33bp)  .. (351.95bp,166.19bp);
  % Edge: Lemma 25 -> Lemma 29
  \draw [->] (463.03bp,259.78bp) .. controls (460.09bp,253.81bp) and (456.72bp,246.97bp)  .. (448.96bp,231.19bp);
  % Edge: Lemma 1 -> Lemma 7
  \draw [->] (142.97bp,259.95bp) .. controls (157.11bp,252.53bp) and (173.86bp,243.74bp)  .. (198.0bp,231.06bp);
  % Edge: Lemma 8 -> Theorem 1
  \draw [->] (443.98bp,64.785bp) .. controls (448.14bp,58.535bp) and (452.95bp,51.332bp)  .. (463.04bp,36.194bp);
  % Edge: Lemma 17 -> Lemma 19
  \draw [->] (294.59bp,324.78bp) .. controls (293.95bp,319.09bp) and (293.23bp,312.61bp)  .. (291.41bp,296.19bp);
  % Edge: Lemma 26 -> Lemma 31
  \draw [->] (675.3bp,259.78bp) .. controls (672.46bp,253.81bp) and (669.2bp,246.97bp)  .. (661.69bp,231.19bp);
  % Edge: Lemma 10 -> Theorem 1
  \draw [->] (569.32bp,64.951bp) .. controls (553.86bp,57.458bp) and (535.52bp,48.571bp)  .. (509.71bp,36.062bp);
  % Edge: Lemma 15 -> Lemma 16
  \draw [->] (238.5bp,454.78bp) .. controls (238.5bp,449.09bp) and (238.5bp,442.61bp)  .. (238.5bp,426.19bp);
  % Edge: Lemma 24 -> Lemma 25
  \draw [->] (471.5bp,324.78bp) .. controls (471.5bp,319.09bp) and (471.5bp,312.61bp)  .. (471.5bp,296.19bp);
  % Edge: Lemma 30 -> Lemma 22
  \draw [->] (545.83bp,194.78bp) .. controls (539.51bp,188.26bp) and (532.18bp,180.69bp)  .. (518.15bp,166.19bp);
  % Edge: Lemma 32 -> Lemma 33
  \draw [->] (829.43bp,194.95bp) .. controls (812.98bp,187.39bp) and (793.44bp,178.4bp)  .. (766.6bp,166.06bp);
  % Edge: Lemma 14 -> Lemma 18
  \draw [->] (336.38bp,389.74bp) .. controls (344.97bp,368.17bp) and (359.77bp,331.03bp)  .. (373.57bp,296.4bp);
  % Edge: Lemma 18 -> Lemma 22
  \draw [->] (378.94bp,259.72bp) .. controls (378.24bp,241.68bp) and (379.94bp,213.5bp)  .. (394.5bp,195.0bp) .. controls (408.92bp,176.69bp) and (432.04bp,165.54bp)  .. (462.92bp,155.95bp);
  % Edge: Lemma 21 -> Lemma 8
  \draw [->] (383.08bp,129.78bp) .. controls (390.28bp,123.12bp) and (398.64bp,115.37bp)  .. (413.94bp,101.19bp);
  % Edge: Lemma 17 -> Lemma 18
  \draw [->] (319.45bp,324.78bp) .. controls (328.71bp,317.84bp) and (339.54bp,309.72bp)  .. (357.57bp,296.19bp);
  % Edge: Lemma 7 -> Theorem 1
  \draw [->] (246.57bp,194.62bp) .. controls (274.13bp,165.46bp) and (332.99bp,105.85bp)  .. (390.5bp,65.0bp) .. controls (402.83bp,56.237bp) and (417.01bp,47.892bp)  .. (439.08bp,36.053bp);
  % Edge: Lemma 14 -> Lemma 17
  \draw [->] (320.48bp,389.78bp) .. controls (317.35bp,383.81bp) and (313.77bp,376.97bp)  .. (305.51bp,361.19bp);
  % Edge: Lemma 29 -> Lemma 22
  \draw [->] (457.17bp,194.78bp) .. controls (463.49bp,188.26bp) and (470.82bp,180.69bp)  .. (484.85bp,166.19bp);
  % Edge: Lemma 26 -> Lemma 27
  \draw [->] (708.64bp,259.78bp) .. controls (718.88bp,252.77bp) and (730.88bp,244.56bp)  .. (750.39bp,231.19bp);
  % Edge: Lemma 14 -> Lemma 1
  \draw [->] (300.23bp,389.98bp) .. controls (285.33bp,381.31bp) and (266.95bp,370.6bp)  .. (250.5bp,361.0bp) .. controls (215.93bp,340.81bp) and (176.6bp,317.77bp)  .. (139.71bp,296.14bp);
  % Edge: Lemma 19 -> Lemma 7
  \draw [->] (273.38bp,259.78bp) .. controls (267.26bp,253.26bp) and (260.18bp,245.69bp)  .. (246.6bp,231.19bp);
  % Edge: Lemma 26 -> Lemma 22
  \draw [->] (693.81bp,259.71bp) .. controls (703.07bp,241.67bp) and (713.16bp,213.48bp)  .. (698.5bp,195.0bp) .. controls (680.34bp,172.11bp) and (602.55bp,159.54bp)  .. (540.25bp,152.44bp);
  % Node: Lemma 28
\begin{scope}
  \definecolor{strokecol}{rgb}{0.0,0.0,0.0};
  \pgfsetstrokecolor{strokecol}
  \draw  (601.0bp,296.0bp) -- (524.0bp,296.0bp) -- (524.0bp,260.0bp) -- (601.0bp,260.0bp) -- cycle;
  %\fill[pattern=north west lines, pattern color=red, opacity=0.2] (601.0bp,296.0bp) -- (524.0bp,296.0bp) -- (524.0bp,260.0bp) -- (601.0bp,260.0bp) -- cycle;
%  \fill[pattern=north east lines, pattern color=blue, opacity=0.2] (601.0bp,296.0bp) -- (524.0bp,296.0bp) -- (524.0bp,260.0bp) -- (601.0bp,260.0bp) -- cycle;
  \draw  (562.5bp,278.0bp) node {\Large{Lemma \ref{RHV2}}};
% \fill[pattern=north west lines, pattern color=grey]  (601.0bp,296.0bp) -- (524.0bp,296.0bp) -- (524.0bp,260.0bp) -- (601.0bp,260.0bp) -- cycle;
\end{scope}
  % Node: Lemma 29
\begin{scope}
  \definecolor{strokecol}{rgb}{0.0,0.0,0.0};
  \pgfsetstrokecolor{strokecol}
  \draw  (479.0bp,231.0bp) -- (402.0bp,231.0bp) -- (402.0bp,195.0bp) -- (479.0bp,195.0bp) -- cycle;
  %\fill[color=zzttqq,fill=zzttqq,fill opacity=0.1]  (479.0bp,231.0bp) -- (402.0bp,231.0bp) -- (402.0bp,195.0bp) -- (479.0bp,195.0bp) -- cycle;
  \draw (440.5bp,213.0bp) node {\Large{ Lemma \ref{TTV}}};
\end{scope}
  % Node: Theorem 1
\begin{scope}
  \definecolor{strokecol}{rgb}{0.0,0.0,0.0};
  \pgfsetstrokecolor{strokecol}
  \draw (513.0bp,36.0bp) -- (436.0bp,36.0bp) -- (436.0bp,0.0bp) -- (513.0bp,0.0bp) -- cycle;
  \draw (474.5bp,18.0bp) node {\Large{Theorem \ref{TH1}}};
\end{scope}
  % Node: Lemma 23
\begin{scope}
  \definecolor{strokecol}{rgb}{0.0,0.0,0.0};
  \pgfsetstrokecolor{strokecol}
  \draw (357.0bp,231.0bp) -- (280.0bp,231.0bp) -- (280.0bp,195.0bp) -- (357.0bp,195.0bp) -- cycle;
  \draw (318.5bp,213.0bp) node {\Large{Lemma \ref{lele11prelem}}};
\end{scope}
  % Node: Lemma 20
\begin{scope}
  \definecolor{strokecol}{rgb}{0.0,0.0,0.0};
  \pgfsetstrokecolor{strokecol}
  \draw (237.0bp,296.0bp) -- (160.0bp,296.0bp) -- (160.0bp,260.0bp) -- (237.0bp,260.0bp) -- cycle;
  \draw (198.5bp,278.0bp) node {\Large{Lemma \ref{LE13}}};
\end{scope}
  % Node: Lemma 21
\begin{scope}
  \definecolor{strokecol}{rgb}{0.0,0.0,0.0};
  \pgfsetstrokecolor{strokecol}
  \draw (403.0bp,166.0bp) -- (326.0bp,166.0bp) -- (326.0bp,130.0bp) -- (403.0bp,130.0bp) -- cycle;
  \draw (364.5bp,148.0bp) node {\Large{Lemma \ref{lele11}}};
\end{scope}
  % Node: Lemma 26
\begin{scope}
  \definecolor{strokecol}{rgb}{0.0,0.0,0.0};
  \pgfsetstrokecolor{strokecol}
  \draw (722.0bp,296.0bp) -- (645.0bp,296.0bp) -- (645.0bp,260.0bp) -- (722.0bp,260.0bp) -- cycle;
  \draw (683.5bp,278.0bp) node {\Large{Lemma \ref{gy2}}};
\end{scope}
  % Node: Lemma 27
\begin{scope}
  \definecolor{strokecol}{rgb}{0.0,0.0,0.0};
  \pgfsetstrokecolor{strokecol}
  \draw (814.0bp,231.0bp) -- (737.0bp,231.0bp) -- (737.0bp,195.0bp) -- (814.0bp,195.0bp) -- cycle;
  \draw (775.5bp,213.0bp) node {\Large{Lemma \ref{RHV}}};
\end{scope}
  % Node: Lemma 24
\begin{scope}
  \definecolor{strokecol}{rgb}{0.0,0.0,0.0};
  \pgfsetstrokecolor{strokecol}
  \draw (510.0bp,361.0bp) -- (433.0bp,361.0bp) -- (433.0bp,325.0bp) -- (510.0bp,325.0bp) -- cycle;
  \draw (471.5bp,343.0bp) node {\Large{Lemma \ref{BNDDIR}}};
\end{scope}
  % Node: Lemma 25
\begin{scope}
  \definecolor{strokecol}{rgb}{0.0,0.0,0.0};
  \pgfsetstrokecolor{strokecol}
  \draw (510.0bp,296.0bp) -- (433.0bp,296.0bp) -- (433.0bp,260.0bp) -- (510.0bp,260.0bp) -- cycle;
  \draw (471.5bp,278.0bp) node {\Large{Lemma \ref{gy}}};
\end{scope}
  % Node: Lemma 22
\begin{scope}
  \definecolor{strokecol}{rgb}{0.0,0.0,0.0};
  \pgfsetstrokecolor{strokecol}
  \draw (540.0bp,166.0bp) -- (463.0bp,166.0bp) -- (463.0bp,130.0bp) -- (540.0bp,130.0bp) -- cycle;
  \draw (501.5bp,148.0bp) node {\Large{Lemma \ref{lele121}}};
\end{scope}
  % Node: Lemma 31
\begin{scope}
  \definecolor{strokecol}{rgb}{0.0,0.0,0.0};
  \pgfsetstrokecolor{strokecol}
  \draw (692.0bp,231.0bp) -- (615.0bp,231.0bp) -- (615.0bp,195.0bp) -- (692.0bp,195.0bp) -- cycle;
  \draw (653.5bp,213.0bp) node {\Large{Lemma \ref{EEV3}}};
\end{scope}
  % Node: Lemma 30
\begin{scope}
  \definecolor{strokecol}{rgb}{0.0,0.0,0.0};
  \pgfsetstrokecolor{strokecol}
  \draw (601.0bp,231.0bp) -- (524.0bp,231.0bp) -- (524.0bp,195.0bp) -- (601.0bp,195.0bp) -- cycle;
  \draw (562.5bp,213.0bp) node {\Large{Lemma \ref{EEV2}}};
\end{scope}
  % Node: Lemma 33
\begin{scope}
  \definecolor{strokecol}{rgb}{0.0,0.0,0.0};
  \pgfsetstrokecolor{strokecol}
  \draw (768.0bp,166.0bp) -- (691.0bp,166.0bp) -- (691.0bp,130.0bp) -- (768.0bp,130.0bp) -- cycle;
  \draw (729.5bp,148.0bp) node {\Large{Lemma \ref{LE9999}}};
\end{scope}
  % Node: Lemma 32
\begin{scope}
  \definecolor{strokecol}{rgb}{0.0,0.0,0.0};
  \pgfsetstrokecolor{strokecol}
  \draw (905.0bp,231.0bp) -- (828.0bp,231.0bp) -- (828.0bp,195.0bp) -- (905.0bp,195.0bp) -- cycle;
  \draw (866.5bp,213.0bp) node {\Large{Lemma \ref{appfeq2}}};
\end{scope}
  % Node: Lemma 17
\begin{scope}
  \definecolor{strokecol}{rgb}{0.0,0.0,0.0};
  \pgfsetstrokecolor{strokecol}
  \draw (335.0bp,361.0bp) -- (258.0bp,361.0bp) -- (258.0bp,325.0bp) -- (335.0bp,325.0bp) -- cycle;
  \draw (296.5bp,343.0bp) node {\Large{Lemma \ref{LE10pre}}};
\end{scope}
  % Node: Lemma 16
\begin{scope}
  \definecolor{strokecol}{rgb}{0.0,0.0,0.0};
  \pgfsetstrokecolor{strokecol}
  \draw (277.0bp,426.0bp) -- (200.0bp,426.0bp) -- (200.0bp,390.0bp) -- (277.0bp,390.0bp) -- cycle;
  \draw (238.5bp,408.0bp) node {\Large{Lemma \ref{LE10preDim2}}};
\end{scope}
  % Node: Lemma 15
\begin{scope}
  \definecolor{strokecol}{rgb}{0.0,0.0,0.0};
  \pgfsetstrokecolor{strokecol}
  \draw (277.0bp,491.0bp) -- (200.0bp,491.0bp) -- (200.0bp,455.0bp) -- (277.0bp,455.0bp) -- cycle;
  \draw (238.5bp,473.0bp) node {\Large{Lemma \ref{additive}}};
\end{scope}
  % Node: Lemma 14
\begin{scope}
  \definecolor{strokecol}{rgb}{0.0,0.0,0.0};
  \pgfsetstrokecolor{strokecol}
  \draw (368.0bp,426.0bp) -- (291.0bp,426.0bp) -- (291.0bp,390.0bp) -- (368.0bp,390.0bp) -- cycle;
  \draw (329.5bp,408.0bp) node {\Large{Lemma \ref{trew}}};
\end{scope}
  % Node: Lemma 13
\begin{scope}
  \definecolor{strokecol}{rgb}{0.0,0.0,0.0};
  \pgfsetstrokecolor{strokecol}
  \draw (149.0bp,361.0bp) -- (72.0bp,361.0bp) -- (72.0bp,325.0bp) -- (149.0bp,325.0bp) -- cycle;
  \draw (110.5bp,343.0bp) node {\Large{Lemma \ref{rational2}}};
\end{scope}
  % Node: Lemma 12
\begin{scope}
  \definecolor{strokecol}{rgb}{0.0,0.0,0.0};
  \pgfsetstrokecolor{strokecol}
  \draw (77.0bp,426.0bp) -- (0.0bp,426.0bp) -- (0.0bp,390.0bp) -- (77.0bp,390.0bp) -- cycle;
  \draw (38.5bp,408.0bp) node {\Large{Lemma \ref{subprime}}};
\end{scope}
  % Node: Lemma 11
\begin{scope}
  \definecolor{strokecol}{rgb}{0.0,0.0,0.0};
  \pgfsetstrokecolor{strokecol}
  \draw (168.0bp,426.0bp) -- (91.0bp,426.0bp) -- (91.0bp,390.0bp) -- (168.0bp,390.0bp) -- cycle;
  \draw (129.5bp,408.0bp) node {\Large{Lemma \ref{rational}}};
\end{scope}
  % Node: Lemma 10
\begin{scope}
  \definecolor{strokecol}{rgb}{0.0,0.0,0.0};
  \pgfsetstrokecolor{strokecol}
  \draw (643.0bp,101.0bp) -- (566.0bp,101.0bp) -- (566.0bp,65.0bp) -- (643.0bp,65.0bp) -- cycle;
  \draw (604.5bp,83.0bp) node {\Large{Lemma \ref{LE8}}};
\end{scope}
  % Node: Lemma 19
\begin{scope}
  \definecolor{strokecol}{rgb}{0.0,0.0,0.0};
  \pgfsetstrokecolor{strokecol}
  \draw (328.0bp,296.0bp) -- (251.0bp,296.0bp) -- (251.0bp,260.0bp) -- (328.0bp,260.0bp) -- cycle;
  \draw (289.5bp,278.0bp) node {\Large{Lemma \ref{LE10}}};
\end{scope}
  % Node: Lemma 18
\begin{scope}
  \definecolor{strokecol}{rgb}{0.0,0.0,0.0};
  \pgfsetstrokecolor{strokecol}
  \draw (419.0bp,296.0bp) -- (342.0bp,296.0bp) -- (342.0bp,260.0bp) -- (419.0bp,260.0bp) -- cycle;
  \draw (380.5bp,278.0bp) node {\Large{Lemma  \ref{EEEV}}};
\end{scope}
  % Node: Lemma 1
\begin{scope}
  \definecolor{strokecol}{rgb}{0.0,0.0,0.0};
  \pgfsetstrokecolor{strokecol}
   \draw (149bp,296.0bp) -- (72.0bp,296.0bp) -- (72.0bp,260.0bp) -- (149.0bp,260.0bp) -- cycle;
  \draw (110.5bp,278.0bp) node {\Large{Lemma \ref{LE3}}};
\end{scope}
  % Node: Lemma 7
\begin{scope}
  \definecolor{strokecol}{rgb}{0.0,0.0,0.0};
  \pgfsetstrokecolor{strokecol}
  \draw (267.0bp,231.0bp) -- (194.0bp,231.0bp) -- (194.0bp,195.0bp) -- (267.0bp,195.0bp) -- cycle;
  \draw (230.5bp,213.0bp) node {\Large{Lemma \ref{LE4}}};
\end{scope}
  % Node: Lemma 9
\begin{scope}
  \definecolor{strokecol}{rgb}{0.0,0.0,0.0};
  \pgfsetstrokecolor{strokecol}
 \draw (553.0bp,101.0bp) -- (480.0bp,101.0bp) -- (480.0bp,65.0bp) -- (553.0bp,65.0bp) -- cycle;
  \draw (516.5bp,83.0bp) node {\Large{Lemma \ref{LE7}}};
\end{scope}
  % Node: Lemma 8
\begin{scope}
  \definecolor{strokecol}{rgb}{0.0,0.0,0.0};
  \pgfsetstrokecolor{strokecol}
 \draw (470.0bp,101.0bp) -- (396.0bp,101.0bp) -- (396.0bp,65.0bp) -- (470.0bp,65.0bp) -- cycle;
  \draw (432.5bp,83.0bp) node {\Large{Lemma \ref{LE6}}};
\end{scope}
\end{tikzpicture}}
\caption{Structure of proof of Theorem 1}
\end{figure} 
  
Let $\val=(\alpha_1,\ldots,\alpha_n)$ where $\alpha_j>0$ and let $\newsym{$(a_1,\ldots,a_j)$ where $a_j$ is typically of type $(N,\chi)$.}{\va}=(a_1,\ldots,a_j)$, where $a_j$ are unimodular completely multiplicative functions on $\N+\alpha_j$ for $j=1,\ldots,n$. Then define
\begin{gather} \label{flv}
    \newsym{A twisted truncated multiple Hurwitz zeta-function}{\zeta_{\va}^{[N]}(\vs;\val)} =\sum_{0 \leq k_1 < k_2<\cdots < k_n \leq N} \prod_{j=1}^n \frac{a_j(k_j+\alpha_j)}{ (k_j+\alpha_j)^{s_j}},
\end{gather}
 We now assume that $a_j:\N+ \alpha_j \to \C$ are of type $(N_0,\chi_j)$, where $\chi_j$ is a character of modulus $q_j$ and where $q_j$ and $q_{j+1}$ are coprime.  It is a consequence of Lemma \ref{LE10pre} (see section \ref{sc6}) that the limit
\begin{gather} \label{flvl}
     \newsym{A  twisted multiple Hurwitz zeta-function}{\zeta_{\va}(\vs;\val)} = \lim_{N \to \infty} \zeta_{\va}^{[N]}(\vs;\val).
\end{gather}
converges to an analytic function on $\{s \in \C: \Re(s)>1/2\}^n$  and that the convergence is uniform on $\M^n$.
\begin{lem} \label{LE4}
(Fundamental lemma) Let $n\geq 2$ and $\M$ be defined by \eqref{Kdef} and let $\val=(\alpha_1,\ldots,\alpha_n)$ where $\alpha_j>0$ are rational or transcendental.   Then for any $\varepsilon>0$ and polynomial $p$ in $n$ complex variables, and functions $\chi_j$ of modulus $q_j$ permissible with respect to $\alpha_j$ and where $q_j$ and $q_{j+1}$ are coprime for $1 \leq j \leq n-1$ there exists some  $N^* \in \Z^+$ and $\va=(a_1,\ldots,a_n)$ where the functions $a_j:\N+\alpha_j \to \C$ are of type $(N^*,\chi_j)$  such that 
\begin{gather*} 
\sup_{N \geq N^*} \max_{\vs \in \M^n} \abs{\zeta_{\va}^{[N]}(\vs;\val)-p(\vs)} < \varepsilon.
		\end{gather*}
\end{lem}
We will prove the fundamental lemma in section \ref{secstar}. While we can prove this result unconditionally, assuming the Riemann hypothesis for the  Dirichlet $L$-functions simplifies the proof, as we may use $\Pri$ as the set of all primes in Lemma \ref{subprime} (see p. \pageref{subprime}).  In the next lemma however, the Riemann hypothesis would not only simplify the proof, for example Lemma \ref{BNDDIR} would not be needed, but is also necessary if we wish  to obtain the strongest possible region of universality when the last two parameters $\alpha_{n-1}$ and $\alpha_n$ are rational. We would also like to remark that the case where $\alpha_n$ is transcendental is especially simple\footnote{The proof of Lemma \ref{lele121} in section \ref{lele121ref} would be trivial.} and assuming this would allow us to prove the next lemma in 5 pages rather than 22 pages.
\begin{lem} \label{LE6}
 Let  $n\geq 2$ and the conditions on $\val$ and $K$ of Theorem \ref{TH1} be satisfied.  Then for any $\varepsilon>0$ and $\va=(a_1,\ldots,a_n)$ satisfying the conditions of Lemma \ref{LE4}, where   furthermore the order of  $\chi_n$ is divisible by 4,  we have that 
\begin{gather*}
     \liminf_{N \to \infty} \liminf_{T \to \infty}  \frac 1 {T^n} \operatorname{meas} \left \{ \vt \in [0,T]^n: \max_{\vs \in K} \abs{\zeta_{\ett}^{[N]}(\vs+i\vt;\val)-\zeta_{\va}^{[N]}(\vs;\val)} < \varepsilon  \right \} >0.
  \end{gather*}   
\end{lem}
We will prove this result in section \ref{secstar}. Next we are going to state the fact that the multiple Hurwitz zeta-function can be approximated by a sufficiently smoothed truncated Dirichlet polynomial. Assume that
\begin{gather} \label{phidef} \phi \in C_0^\infty(\R),  \qquad \text{and} \qquad \phi(x)=1, \qquad \text{ if } \qquad 0 \leq x \leq 2,
\end{gather}
 and define
\begin{gather}\label{zetaphidef}
 \newsym{A smoothly truncated multiple Hurwitz zeta-function}{\zeta_n^{[\phi,T]}(\vs;\val)}=\sum_{0=k_0 \leq k_1<k_2<\cdots<k_n} 
\prod_{j=1}^n (k_j+\alpha_j)^{-s_j}  \phi\p{\frac{k_{j}-k_{j-1}} {T}}.
\end{gather}
Then we can prove the following weak approximate functional equation (for proof, see section \ref{sec5}, Lemma \ref{LE9999} for a slightly different version that by modifying $T,\ddddd$ slightly depending on $R$ gives the version below)
\begin{lem} \label{LE7}
  For any $A,\ddddd,R>0$ and $\phi$ that fulfills \eqref{phidef} we have that there exists some $B$ so that for $T \geq R+1$, $T^\ddddd-R \leq \Im(s_i) \leq T+R$, $-A \leq \Re(s_i) \leq A$ for all $i=1,\ldots, n$ the following inequality holds 
	   $$\abs{\zeta_n(\vs;\val)-\zeta_n^{[\phi,T]}(\vs;\val)} <B T^{-A}.$$
\end{lem}

By classical methods, squaring and treating the non-diagonal part in an elementary way utilizing the fact that the region lies strictly to the right of the half line in each variable   we obtain
\begin{lem} \label{LE8}
 Let $\ddd,R>0$, $\M$ be defined by \eqref{Kdef} and let $\phi$ fulfill \eqref{phidef} and let $0<\ddddd<1$.  Then there exist some $C>0$ such that if $1 \leq N \leq T$ then
\begin{gather}
    \frac 1 {T^n}\int_{\M^n} \int_{[T^{\ddddd},T]^n} \abs{\zeta_{\ett}^{[N]}(\vs +i\vt;\val)-\zeta_n^{[\phi,T]}(\vs+i\vt;\val)}^2 d\vt d \vs \leq C  N^{-\ddd}. 
\end{gather}
\end{lem}
We are now ready by means of Lemma \ref{LE4}, Lemma \ref{LE6}, Lemma \ref{LE7} and Lemma \ref{LE8} to prove our main result, Theorem \ref{TH1}.

\noindent {\em Proof of Theorem \ref{TH1}.} \label{here}
We first choose $R$ and $\ddd$ in \eqref{Kdef} so that $K \subset \M^n$ and the distance between $K$ and $(\M^n)^\complement$ is strictly positive.
It follows by the conditions in the theorem  and the Oka-Weil theorem (Oka \cite{Oka}, Weil \cite{Weil}, see also \cite{Passare}, \cite{Levenberg}), 
that the function $f$ can be uniformly approximated by a polynomial $p$ on the set  $K$, such that 
\begin{gather} \label{ett}
 \max_{\vs \in K} |p(\vs)-f(\vs)|<\frac {\varepsilon} 4.
\end{gather} 
By standard zero-density estimates \cite[(1.17)]{bombieri} for Dirichlet $L$-functions we may find  $q_j \in \Z^+$ for $j=1,\ldots,n$, and characters $\chi_j$ mod $q_j$ such that $q_{j}$ and $q_{j+1}$ are coprime, the order of $\chi_n$ is divisible by 4, and such that if $\alpha_j=\frac {c_j}{d_j}$ is rational with coprime $c_j,d_j$, then $L(s,\chi_j \chi^*)$ is zero-free on $\M$ for each character $\chi^*$ mod $d_j$. 
By Lemma \ref{LE4} there exists some $N_0$ and functions $a_j:\N+\alpha_j \to \C$ of type  $(N_0,\chi)$ such that 
\begin{gather} \label{ettb1}
   \sup_{N \geq N_0} \max_{\vs \in K} \abs{\zeta_{\va}^{[N]}(\vs;\val)-p(s) }< \frac {\varepsilon} 4.
\end{gather} 
By Lemma \ref{LE6} we have that there exists some $\ddddd>0$ and $N_1 \geq N_0$ such that for any $N \geq N_1$ we have that
\begin{gather} \label{ettb}
 \liminf_{T \to \infty} \frac 1 T \meas \left \{ \max_{\vs \in K} \abs{\zeta_{\ett}^{[N]}(\vs+i\vt;\val)- \zeta_{\va}(\vs;\val) }< \frac {\varepsilon} 4 \right \} \geq  \ddddd>0.
\end{gather}
 By the triangle inequality, Lemma 
\ref{LE7} and Lemma \ref{LE8} it follows
 that 
$\zeta_{\ett}^{[N]}(\vs+i\vt;\val)$ approximates $\zeta(\vs+i\vt;\val)$ with distance less than $\varepsilon_1^2$ on $L^2(\M^n)$ for $0 \leq t \leq T$  on a set with measure greater then $T(1-CN^{-\ddd} \varepsilon_1^{-2})$ for some $C>0$. By choosing $N \geq N_1$ sufficiently large we can make this measure greater than $T(1-\ddddd/2)$.  Since we have some room to spare, standard methods from complex analysis implies that arbitrarily close estimation in $L^2$-norm on the slightly larger set (distance between $K$ and $(\M^n)^\complement$ is positive) 
gives us arbitrarily close estimation in sup-norm on $K$ so that by choosing $\varepsilon_1$ sufficiently small we get
\begin{gather} \label{ettc}
 \max_{\vs \in K} \abs{\zeta_{\ett}^{[N]}(\vs+i\vt;\val)-\zeta_n(\vs+i\vt;\val) } <\frac {\varepsilon} 4,
\end{gather}
for $t$  in a subset of $[0,T]$  of positive measure greater than $T(1-\ddddd/2)$. By  the triangle inequality and the inequalities \eqref{ett}, \eqref{ettb1}, \eqref{ettb}, \eqref{ettc}  we conclude that 
\begin{gather} %\label{tva}
 \liminf_{T \to \infty} \frac 1 {T^n}  \operatorname{meas} \left \{t \in [0,T]^n: \max_{\vs \in K} \abs{\zeta_n(\vs+i\vt;\val)-f(\vs)} < \varepsilon \right \} \geq \frac {\ddddd} 2>0. 
\end{gather}  \qed

\section{Proof of Lemma \ref{LE3}} \label{secLemma1}
Lemma \ref{LE3} is an easy consequence of Lemma \ref{onevar} when $\alpha$ is transcendental. When $\alpha$ is rational we  use a variant of the following joint universality result for finite Euler-products
\begin{lem} \label{rational}
  Let $\chi_1$,\ldots,$\chi_m$ be distinct non-principal Dirichlet characters mod $q$. Then given any  subset $\newsym{A set of primes}{\mathcal P}$ of the primes that has positive density with respect to each residue class mod $q$
\begin{gather} \label{rrr}
  \lim_{N \to \infty} \frac {1} {\pi(N)} \{p \in \mathcal P: p \equiv a \pmod q, \, p \leq N \}=\frac{c_{a}}{\phi(q)}>0, \qquad  \GCD(a,q)=1, \qquad
\end{gather}
and any continuous $f_1,\ldots,f_m$ that are either identically zero\footnote{Allowing identically zero-functions in Lemma \ref{rational} easily follows from the Lemma without explicitly allowing such function since  as functions we may choose the functions $f_j(s)=\varepsilon$ for arbitrarily small $\varepsilon>0$. However since we will apply the Lemma in exactly this situation we choose to make this explicit in the formulation of the Lemma.} or are zero-free functions on $\M$ that are analytic in the interior of $\M$ and $\varepsilon,N_0>0$, then there exist some $N_1>N_0$ and coefficients $|a(p)|=1$ such that
\begin{gather*}
 \max_{1 \leq j \leq m} \max_{s \in \M}\abs{\prod_{\substack{N_0<p<N_1 \\ p \in \mathcal P}} (1-a(p) \chi_j(p)p^{-s})^{-1}-f_j(s)   }<\varepsilon.
\end{gather*}
\end{lem}
\begin{proof}
  When $\Pri$ is the set of all primes this is a standard result in the field which is usually used to prove joint universality for Dirichlet $L$-functions, see e.g \cite[Lemma 4.9]{Bagchi}. Its proof depends on the Pechersky rearrangement theorem. The version above follows by the same proof method since positive density of the primes in each residue class is sufficient for using the Pechersky rearrangement theorem in a similar manner. 
\end{proof}
We also need that there exist some subset $\Pri$ of the primes with the desired properties such that certain Euler-products over the primes are  convergent.
\begin{lem} \label{subprime}
 Let $q \geq 3$. Then there exists a subset $\mathcal P$ of the primes containing asymptotically all primes
\begin{gather} \label{priineq}
   \sharp \{p \in \Pri: p \leq N \} = \pi(N) \p{1+  O \p{\frac 1 {(\log N)^{A}}}}  \qquad (A>0),
\end{gather}
such that the product
  \begin{gather}
    \lim_{N \to \infty} \prod_{\substack{p \in \Pri \\ p \leq N }} (1-\chi(p)p^{-s})^{-1}=    \newsym{A restricted Dirichlet $L$-function with Euler-product over primes in $\Pri$}{L_{\Pri}(s,\chi)},
   \end{gather}
   is  convergent to a zero-free analytic function $L_{\Pri}(s,\chi)$  on the half plane $\Re(s)>1/2$ for each non-principal character $\chi \mod q$. 
\end{lem}
\begin{proof}
It is sufficient if we when  $N \to \infty$  can choose  asymptotically all  of the primes $p \equiv a \pmod q$ in each residue class $(a,q)=1$ from an interval $[N,N+\delta_N N]$, where $\delta_N=\frac 1 {(\log N)^A}$ to be in the set $\Pri$.
 The following is  a variant of a result of Selberg \cite{Selberg} in arithmetic progressions which is a consequence of zero-density estimates for the Dirichlet $L$-functions,
\begin{gather} \label{tyt}
   \max_{(a,q)=1} \int_x^{2x}
  \abs{\sum_{\substack{t <p<t+h \\ p \equiv a \pmod q}} \log p-\frac h {\phi(q)}} dt \ll_{q,B,\varepsilon} \frac {hx} {(\log x)^B},  \qquad (h \geq x^{1/6+\varepsilon}), 
 \end{gather}
see e.g. \cite[Theorem 1.1]{Dim}. We are going to consider primes in the intervals $I_n=[n^2,(n+1)^2].$ The expected number of primes in the interval $I_n$ congruent to $a \pmod q$ will be
\begin{gather}
   E_n=\frac 1 {\phi(q)} \left(\operatorname{Li}((n+1)^2)-\operatorname{Li}(n^2)\right).  
\end{gather}
From \eqref{tyt} it follows that there exist some $N_0>0$ such that for $N \geq N_0$ we have that
\begin{gather}
   \# \left \{n^2 \in \left[N,N+\delta_N N \right] : \min_{(a,q)=1} \# \{ p \in I_n , \, p \equiv a \pmod  q \} < (1-\delta_N) E_{n} \right \}   \leq  \sqrt{N}\delta_N^2,   
\end{gather}
so that for a proportion of at least $(1-\delta_N)$ of the intervals $I_n \subset [N,N+\delta_N N]$  there are at least $(1-\delta_N)$ of the expected number of primes in each residue class. 
For each such value of $n$ we choose $\lfloor (1-\delta_N)E_n \rfloor$ primes from each residue class $a \pmod q$ with $(a,q)=1$ for the set $\Pri$.  By its construction it is clear that the set $\Pri$ has   more than $(1-\delta_N)^2$  of the primes in each such residue class which gives us \eqref{priineq}. 
By taking the logarithm of the product in Lemma \ref{subprime} and by noticing that the prime power part is absolutely convergent for $\Re(s)>1/2$ it follows that
it is sufficient to show that the series
\begin{gather} \label{jj}
 \sum_{p \in \Pri} \chi(p)p^{-s}
\end{gather}
is  convergent for $\Re(s)>1/2$. If $p \in I_n$ we have that $p=n^2+x$ for $0<x \leq 2n$ and 
\begin{gather*} p^{-s}=(n^2+x)^{-s}=n^{-2s}\p{1+\frac x {n^2}}^{-s}=n^{-2x}+O \p{sn^{-2s-1}}, \end{gather*}
from which it follows that the contribution to the sum \eqref{jj} from each prime in the interval $I_n$ is
\begin{gather} \label{jj2}
  \chi(p) n^{-2s}+O(s n^{-2s-1}).
\end{gather}
Since we have an equal number of primes in each residue class $p \equiv a \pmod q$ with $(a,q)=1$ in the interval $I_n$ and the character is non-principal the first part will vanish when summing over all the primes in the interval. Since if $n \geq 2$ we have at most $n$ primes in the interval $I_n$ the second part is absolutely convergent when we sum over $n$ when $\Re(s)>1/2$.
\end{proof}
It is clear that if we assume the generalized Riemann hypothesis, then $L_\Pri(s,\chi)=L(s,\chi)$ when we choose $\Pri$ to be the set of all primes in Lemma \ref{subprime} since the convergence of the product over all the primes for  $\Re(s)>1/2$ is  equivalent to the Riemann hypothesis for the Dirichlet $L$-function $L(s,\chi)$. Thus lemma \ref{subprime} may be viewed in the following manner: There exist a subset of the primes containing almost all primes such that the Riemann hypothesis is true for all\footnote{By considering the moduli $q_n=n!$ and a limit argument it is not neccesary to restrict to a single modulus in Lemma \ref{subprime}. However a single modulus  is sufficient for our applications.}  restricted Dirichlet $L$-functions, where ``restricted Dirichlet $L$-functions'' means that rather than taking the Euler product over all primes we take the Euler product restricted over the set $\Pri$.

\begin{lem} \label{rational2}
  Let $\chi_1$,\ldots,$\chi_m$ be distinct non-principal Dirichlet characters mod $q$. Then given any continuous   $f_1,\ldots,f_m$ that are either identically zero or are zero-free functions on $\M$ that are analytic in the interior of $\M$ and $\varepsilon,N_0>0$ then there exist some $N_1>N_0$ and $|a(p)|=1$ for $N_0<p<N_1$ such that
\begin{gather*}
 \max_{1 \leq j \leq m} \max_{s \in \M}\abs{\prod_{N_0<p<N_1} \p{ \frac{1- \chi_j(p)p^{-s}} {1-a(p) \chi_j(p)p^{-s}}}-f_j(s)   }<\varepsilon.
\end{gather*}
\end{lem}
\begin{proof} 
 By Lemma \ref{subprime} it follows that there exists some subset $\Pri$ of the primes such that each residue class mod $q$ contains almost all primes, i.e. \eqref{rrr} holds with $c_a=1$ if $(a,q)=1$ and such that
\begin{gather} \label{iiiii}
  \newsym{A restricted Dirichlet $L$-function with Euler-product over primes in $\Pri$ greater than $N_0$}{L_{\Pri,N_0}(s,\chi)}=\lim_{N_1 \to \infty} \prod_{\substack{N_0<p<N_1 \\ p \in \Pri}} (1-  \chi(p)p^{-s})^{-1}, 
\end{gather}
for each non-principal $\chi$ mod $q$ is convergent to a zero-free analytic function for $\Re(s)>1/2$. Define
\begin{align} \label{irt}
 \lambda&= \min_{1 \leq j \leq m} \min_{s \in \M} \abs{L_{\Pri,N_0}(s,\chi_j)}, \\ \intertext{and}
 \Lambda&=  \max_{1\leq j \leq m} \max_{s \in \M} \abs{f_j(s)}.   \label{irt2}  
\end{align}
 Let $0<\varepsilon_1<\lambda/2$. By equation \eqref{iiiii} we get  some $M>N_0$ such that for any $N_1>M$ then
\begin{gather} \label{uy1} 
 \max_{s \in \M} \abs{L_{\Pri,N_0}(s,\chi)- \prod_{\substack{N_0<p<N_1 \\ p \in \Pri}} (1-  \chi(p)p^{-s})^{-1}}<\varepsilon_1<\frac \lambda 2. 
\end{gather}
By Lemma \ref{rational} we find an $N_1>M$ and $|a(p)|=1$ for $p \in \Pri$ and $N_0<p<N_1$ so that 
\begin{gather} \label{uy2}
 \max_{1 \leq j \leq m} \max_{s \in \M}\abs{\prod_{\substack{N_0<p<N_1 \\ p \in \Pri}}(1-a(p) \chi_j(p)p^{-s})^{-1}-f_j(s) L_{\Pri,N_0}(s,\chi_j)   } <\varepsilon_2.
\end{gather}
 By the triangle inequality and equations \eqref{irt2}, \eqref{uy1} and \eqref{uy2} it follows that
\begin{gather} \label{uy3}
\abs{\prod_{\substack{N_0<p<N_1\\ p \in \Pri}}(1-a(p) \chi_j(p)p^{-s})^{-1}-f_j(s) \prod_{\substack{N_0<p<N_1 \\ p \in \Pri}}(1- \chi_j(p)p^{-s})^{-1}   } <\varepsilon_2+\Lambda \varepsilon_1. \qquad
\end{gather}
for each $1 \leq j \leq m$ and  $s \in \M$,
and furthermore by the triangle inequality and equations \eqref{irt} and \eqref{uy1} that
\begin{gather} \label{uy4} \min_{s \in \M} \abs{\prod_{\substack{N_0<p<N_1 \\ p \in \Pri}}(1- \chi_j(p)p^{-s})^{-1}} \geq \frac \lambda 2.\end{gather}
By the inequalities \eqref{uy3} and \eqref{uy4} it  follows that
\begin{gather} \label{iii4}
\max_{1 \leq j \leq m} \max_{s \in \M}\abs{\prod_{\substack{N_0<p<N_1 \\ p \in \Pri}} \p{ \frac {1- \chi_j(p)p^{-s}} { 1-a(p) \chi_j(p)p^{-s}}}-f_j(s)   }<\frac 2 {\lambda} (\varepsilon_2+\Lambda \varepsilon_1)
\end{gather}
We notice that if we define $a(p)=1$ if $p \not \in \Pri$ for $N_0<p<N_1$ then the product in \eqref{iii4} can be written over all the primes in the interval $N_0<p<N_1$ and Lemma \ref{rational2}  follows by choosing $\varepsilon_1$ and $\varepsilon_2$ so that
$$
 \frac 2 \lambda (\varepsilon_2+\Lambda \varepsilon_1) <\varepsilon. 
$$
\end{proof}

\begin{lem} \label{trew}
 Let $\alpha>0$ be transcendental or rational, $\varepsilon>0$ and $q$ be given. Then there exists some  $C=C(\varepsilon,\alpha,q)$ such that if $a:\N+\alpha \to \C$ is of type $(N_0,\chi)$ and $\chi$ is of modulus $q$ then for any $ N_0 \leq N_1 \leq N_2$  and $s \in \C$ with $\Re(s) \geq \varepsilon$ we have that 
 \begin{gather}
    \abs{\sum_{k=N_1}^{N_2} \frac{ a(k+\alpha)}{ (k+\alpha)^{s}}} \leq C (1+|s|) N_1^{\varepsilon-\Re(s)}.
  \end{gather}
\end{lem}
\begin{proof}
 By partial summation it is sufficient to prove that  
  \begin{gather} \label{Bdeff}
   \abs{\sum_{n=0}^{N} a(n+\alpha)}   \leq \tfrac 1 2 C \varepsilon N^{\varepsilon/2}, \qquad (N \geq N_0).
   \end{gather}
for any $a:\N+\alpha \to \C$ of type $(N_0,\chi)$ where $\chi$ is of modulus $q$.   The inequality is trivial if $\alpha$ is transcendental so we assume that
 $\alpha=r+\frac c d$ is rational where $r \in \N$,  $c,d \in \Z^+$ and $0 \leq c <d$.   Since adding an integer $r$ to $\alpha$ just corresponds to shifting the summation indices by $r$ we may without loss of generality assume that $r=0$. As customary in the theory of numbers without large prime factors \cite{HilTen}  we let  \newsym{The number of positive integers less than $x$ with prime factors less than $y$.}{$\Psi(x,y)$} denote the number of integers less than $x$ with prime factors less than $y$. The inequality \cite[Eq 1.1.14]{HilTen} 
$$
 \Psi (x,(\log x)^\alpha)=x^{1-1/\alpha+o(1)},  \qquad (\alpha>1),
$$
with $\alpha=(1-\varepsilon/4)^{-1}$ gives us that for some $D =D(\varepsilon)>0$ and any $N_0 \in \Z^+$ then
\begin{gather}  \label{nwlpf}
  \Psi(N,\log N_0) \leq D N^{\varepsilon/2}, \qquad (N \geq N_0).
\end{gather}
Let  \newsym{The largest prime factor of $k$}{$P^+(k)$} denote the largest prime factor of $n$. 
By the construction of $a(n)$ we note that
\begin{gather} \label{rraj}
   \sum_{k=1}^\infty a(k) k^{-s}= B(s) L(s,\chi), \\ \intertext{where}
   B(s)=\sum_{P^+(k) \leq \log N_0} b(k) k^{-s}= \prod_{p \leq \log N_0} (1-a(p) p^{-s})^{-1}  (1-\chi(p) p^{-s}), \label{rraj2}
 \end{gather}
and $|b(k)| \leq 2$. From this we obtain 
\begin{gather} \label{rraj3}  
   \sum_{n=0}^{N_1} a(n+\alpha)   = \overline{a(d)}  \sum_{n=0}^{N_1} a(nd+c)  =
	 \overline{a(d)} \sum_{P^+(k) \leq \log N_0} b(k) \sum_{\substack{k l \equiv c \pmod d \\ 0  \leq l <N_1d/k}} \chi(l).
\end{gather}
 The inner sum is bounded by $q$ and we get that 
\begin{gather*}
 \abs{\sum_{n=0}^{N_1} a(n+\alpha)} \leq
2 q  \Psi(N_1, \log N_0)   
\end{gather*}
which  by \eqref{nwlpf} implies the inequality \eqref{Bdeff} with $C=4Dq/\varepsilon$. 
\end{proof}

\noindent {\em Proof of Lemma \ref{LE3}.}
 We first prove the simpler case when $\alpha$ is transcendental. Since $\alpha$ is transcendental we have that the set $\{\log(k+\alpha): k \in \N \}$ is linearly independent over $\Q$ and any unimodular function $a$ on $\N+ \alpha$  is a completely multiplicative function. 
By Lemma \ref{onevar} we have that there there exists unimodular $a(k+\alpha)$ for $N_0<k \leq N_1$ such that
\begin{gather*}
 \max_{s \in \M} \abs{ \sum_{k=N_0+1}^{N_1} a(k+\alpha) (k+\alpha)^{-s} - \left(p(s)-\sum_{k=0}^{N_0} a(k+\alpha) (k+\alpha)^{-s} \right)} <  \varepsilon.
\end{gather*}
The case of $\alpha$ rational needs some more care. We want to use similar reasoning as when we prove that the Hurwitz zeta-function for a rational parameter is universal. First we assume that 
\begin{gather} 
  \alpha= r+\frac c d, \qquad (c,d \in \Z^+, \,  c \leq d, \,  r  \in \N, \, (c,d)=1), \label{rdef} 
\end{gather}
since we see no reason to require that $\alpha_j<1$ in the definition of the multiple Hurwitz zeta-function. 
In a similar manner to how the Hurwitz zeta-function with a rational parameter $0<\frac c d<1$ can be written as a linear combination of Dirichlet $L$-functions it follows by \eqref{ojkond} and \eqref{ojkond2} that  
\begin{align} \notag 
    \sum_{k=0}^\infty a(k+&\alpha) (k  +\alpha)^{-s} + \sum_{k=0}^{r-1} a\p{k+\frac c d} \p{k+\frac c d}^{-s}, \\ 
&= \frac {d^s \overline{a(d)}} {\phi(d)} \sum_{\chi^* \pmod d}  \overline{\chi^*(c)} \prod_{p } (1-\chi^*(p) a(p) p^{-s})^{-1}, \notag  \\
 &= \frac {d^s \overline{a(d)}} {\phi(d)} \sum_{\chi^* \pmod d}  \overline{\chi^*(c)} L(s,\chi\chi^*) \prod_{p< N_1 } \p{\frac{1- \chi^*(p) \chi(p) p^{-s}   }{1-\chi^*(p) a(p) p^{-s}} },  \label{fer}
\end{align}
for any $N_1>\log N_0$. Let us define \begin{gather} \label{fdef}
  f(s)=p(s) + \sum_{k=0}^{r-1} a\p{k+\frac c d} \p{k+\frac c d}^{-s}, 
\end{gather}
and note that by choosing $A\geq r+\alpha^{-1}$ it follows that $\min_{s \in \M}  |f(s)| \geq 1$ and in particular $f$ is zero-free on $\M$. Now let $\chi_0$ be the principal character mod $d$.  Since $\chi$ was assumed to be permissible with respect to $\alpha$ then $\chi$ is a non-principal Dirichlet character mod $q$ where $q$ is coprime to $d$ and also $L(s,\chi \chi_0)$  is zero-free on $\M$. This allow us to apply Lemma \ref{rational2} for the $\phi(d)$ distinct non-principal Dirichlet characters $\chi \chi^*$ mod $dq$ and we can find $N_1>N_0$, and unimodular $a(p)$ for $N_0<p \leq N_1$ such that
\begin{gather} \label{trt1}
  \abs{\prod_{N_0<p \leq N_1}  \p{\frac{ 1-\chi_0(p)\chi(p) p^{-s}} {1-\chi_0(p) a(p) p^{-s}}} - \frac {f(s)   \phi(d) a(d) }{  L(s,\chi \chi_0)d^s}    \prod_{p \leq	N_0}  \p{\frac{ 1-\chi_0(p) a(p) p^{-s}}{1-\chi_0(p)\chi(p) p^{-s}}}} < \varepsilon_1, \hskip 20pt 
\\ \intertext{for each $s \in \M$  and} \label{trt2}
  \max_{s \in \M} \abs{\prod_{N_0<p \leq N_1} \p{\frac{ 1-\chi^*(p) \chi(p) p^{-s}}{1-\chi^*(p) a(p) p^{-s}}}} < \varepsilon_1,
\end{gather}
for all other characters $\chi^* \neq \chi_0$ mod $d$, where  \begin{gather} \label{trt3} \varepsilon_1 = \frac {\varepsilon} 2 \, \min_{\chi^* \mod d} \min_{s \in \M}  \abs{\frac {  d^{-s} }{ L(s,\chi \chi^*)}    \prod_{p \leq	N_0}  \p{\frac{ 1-\chi^*(p) a(p) p^{-s}}{1-\chi^*(p) \chi(p)p^{-s}}}}. \end{gather}  It follows from equation \eqref{trt1}, \eqref{trt2} and \eqref{trt3} that
\begin{gather*} 
  \max_{s \in \M} \abs{\frac {d^s L(s,\chi \chi^*)}{\phi(d)  a(d) \chi^*(c)}  \prod_{p \leq N_1}   \p{\frac{1-\chi^*(p)\chi(p)p^{-s}} {1-\chi^*(p) a(p) p^{-s}}} - f(s) \begin{cases} 1 &  \chi^*  = \chi_0  \\ 0  & \text{otherwise} \end{cases} }< \frac{\varepsilon}{2 \phi(d)},
\end{gather*}
which in view of equations \eqref{fer}, \eqref{fdef} and the triangle inequality gives us the inequality 
\begin{gather} \label{ewwe} \max_{s \in \M} \abs{ \sum_{k=0}^\infty a(k+\alpha) (k+\alpha)^{-s}  - p(s)}<\frac{\varepsilon} 2.
\end{gather}
From Lemma \ref{trew} it follows that
\begin{gather}  \label{ewwe2}
 \max_{s \in \M}  \abs{\sum_{k=N_1+1}^\infty  a(k+\alpha) (k+\alpha)^{-s}}< \frac{\varepsilon} 2,
\end{gather}
provided $N_1$ has been chosen sufficiently large. Our final result follows from the inequalities \eqref{ewwe} and \eqref{ewwe2} and the triangle inequality.
\qed

\section{Other lemmas} \label{sc6}

We would like to obtain  an $n$-dimensional version of Lemma \ref{trew}, but before proving this we prove some preliminary variants. First an averaged form of an  ``additive version''. 
\begin{lem} \label{additive}
 Suppose that $a_j: \N+\alpha_j \to \C$ is a function of type $(N,\chi_j)$ where $\chi_j$ is a character of modulus $q_j$ for $j=1,2$. Then for any $N \leq N_1 \leq N_2$ and $0 \leq r_0$ we have that
\begin{gather}
    \max_{(s_1,s_2) \in \M^2} \abs{\sum_{r=r_0}^{r_0+ q_1-1}  \sum_{n=N_1}^{N_2} \frac{a_1(n+\alpha_1) a_2(n+r+\alpha_2)}{(n+\alpha_1)^{s_1} (n+r+\alpha_2)^{s_2}}} \ll_{\varepsilon} N_1^{\varepsilon-2 \ddd-1}.
\end{gather}
\end{lem}

\begin{proof}
 We  use a similar method as we used to prove Lemma \ref{trew} to treat the right hand side of Eq. \eqref{ytreraj}. By the change of summation $k=n+r$ the sum in our Lemma may be written as 
\begin{align} \label{ppp1}
   (*)&= \sum_{r=r_0}^{r_0+ q_1-1}  \sum_{k=N_1+r}^{N_2+r} \frac{a_1(k-r+\alpha_1) a_2(k+\alpha_2)}{(k-r+\alpha_1)^{s_1} (k+\alpha_2)^{s_2}}. \\
\intertext{Changing the interval of integration in $k$ from $N_1+r \leq k \leq N_2+r$ to $N_1+r_0 \leq k \leq N_2 + r_0$ gives us an error term} 
    (*)&=\sum_{r=r_0}^{r_0+ q_1-1}  \sum_{k=N_1+r_0}^{N_2+r_0} \frac{a_1(k-r+\alpha_1) a_2(k+\alpha_2)}{(k-r+\alpha_1)^{s_1} (k+\alpha_2)^{s_2}} + O(q_1^2 N_1^{-2 \ddd-1}). \label{ppp2}
\end{align}
 Since assuming otherwise just corresponds to a shift of summation intervals we may without loss of generality assume that $0<\alpha_j<1$. Let  \begin{gather}\alpha_j=\frac {c_j} {d_j}, \qquad \operatorname{GCD}(c_j,d_j)=1,
\end{gather}
 where  $d_j \in \Z^+$ and $1 \leq c_j \leq d_j$ if $\alpha_j$ if rational and $d_j=1$,  $c_j=\alpha_j$ if $\alpha_j$ is transcendental.
We consider
\begin{gather} \label{str0}
 \Delta_{r_0}(X)=\sum_{r=0}^{q_2} I_{r}(X), \\ \intertext{where}
   I_r(X)= \sum_{n=N_1+r_0}^{X+r_0} a_1(n-r+\alpha_1) a_2(n+\alpha_2).
\end{gather}
By  Eq \eqref{rraj} and \eqref{rraj2} we may rewrite the sum 
 \begin{align} \label{str}
I_r(X) &= 
\sum_{k_j \in \mathcal A_j}    b_1(k_1) b_2(k_2) a_1(k_1) a_2(k_2) \overline{\chi_1(d_1k_1) \chi_2(d_2 k_2)}  \psi_{k_1,k_2}(r;X), \\ \intertext{where} \psi_{k_1,k_2}(r;X)&=  \sum_{\substack{N_1+r_0 \leq k \leq X+r_0 \\  (n-r)d_1+c_1 \equiv 0 \pmod{k_1} \\ nd_2+c_2 \equiv 0 \pmod{k_2}}  } \chi_1\p{(n-r)d_1+c_1} \chi_2 \p{nd_2+c_2}, \label{psidef3}
  \end{align}
and where $\mathcal A_j=\{k \in \Z: \GCD(k,q_j)=1: P^+(k) \leq \log N_0\}$, where $P^+(k)$ denote the largest prime of $k$ if $\alpha_j$ is rational, and $\mathcal A_j=\{1\}$ if  $\alpha_j$ is transcendental. Let 
 \begin{gather} 
 \Lambda= {\operatorname{LCM}(k_1,k_2)}.
\end{gather}
Since both $c_j$ and $d_j$ as well as $d_j$ and $q_j$ are coprime for $j=1$ and $j=2$ it follows that either there are unique solutions to the congruences
\begin{gather}
\label{congruences} (n-r)d_1+c_1 \equiv 0 \pmod{k_1} \qquad  n d_2+c_2 \equiv 0 \pmod{k_2}
\end{gather}
 which gives us a unique solution \begin{gather} 
   n \equiv b_r \pmod  \Lambda  
\end{gather}
 or there is no solution at all (which  happen if $d_j$ and $k_j$ is not coprime for $j=1$ or $j=2$). In case there is no solution the sum is empty so that 
\begin{gather} \label{prisest}
  \Psi_{k_1,k_2}(r;X)=0
\end{gather}
In case there is a solution we may furthermore find a solution to the congruences
\begin{gather}
  (n-r)d_1+c_1 \equiv m_1 k_1 \pmod {k_1 q_2}  \qquad \text{and} \qquad 
 nd_2+c_2 \equiv m_2 k_2 \pmod {k_2 q_1}
\\ \intertext{with} 0 \leq B_r < \Lambda q_1q_2 \\ \intertext{where}
b_r \equiv B_r \pmod \Lambda, \\ \intertext{and where} \label{trtt2}
 m_j k_j \equiv c_j \pmod {d_j}, \qquad (j=1,2)
\end{gather}
Because $k_j$ and $d_j$ is now assumed to be coprime there is a solution $m_j$ for the congruence \eqref{trtt2}. It follows that
\begin{gather} 
  B_r \equiv \frac{m_jk_j-c_j} {d_j} \pmod {q_j}
\end{gather}
for $j=1,2$ so that 
\begin{gather}
  B_r \equiv B \pmod {q_1q_2} 
\end{gather}  
is a constant with respect to $r$.
The sum may now be written as\begin{gather} 
\psi_{k_1,k_2}(r;X)=   \sum_{k = \lceil (N_1+r_0-B_r)/\Lambda \rceil}^{\lfloor (X+r_0-B_r)/\Lambda \rfloor  } 
\chi_1((\Lambda k+B-r)d_1+c_1) \chi_2 ((\Lambda k+B)d_2+c_2), 
\end{gather}   
When we  replace the summation interval from $(N_1+r_0-B_r)/\Lambda<k<(X+r_0-B_r)/\Lambda$ to $(N_1+r_0)/\Lambda<  k <(X+r_0)/\Lambda$,  the resulting sum will differ by at most $2 q_1q_2$ terms so that if
\begin{gather} 
\psi_{k_1,k_2}^*(r;X)=   \sum_{k=\lceil (N+r_0)/\Lambda \rceil}^{\lfloor (X+r_0)/\Lambda \rfloor  } 
\chi_1((\Lambda k+B)d_1+c_1) \chi_2 ((\Lambda k+B+r)d_2+c_2), 
\end{gather}   
then by the triangle inequality it follows that
\begin{gather} \label{est11}
 \abs{\psi_{k_1,k_2}(r;X)- \psi_{k_1,k_2}^*(r;X)} \leq 2q_1q_2.
\end{gather}
Let us now take the average with respect to  $r$.  We get that
\begin{gather} 
\sum_{r=0}^{q_1-1} \psi_{k_1,k_2}^*(r;X)=   \sum_{r=0}^{q_1-1} \sum_{k=\lceil (N+r_0)/\Lambda \rceil}^{ \lfloor (X+r_0)/\Lambda \rfloor  } 
\chi_1((\Lambda k+B-r)d_1+c_1)) \chi_2 ((\Lambda k+B)d_2+c_2), \\ \intertext{which by changing the summation order equals} 
  \label{rtrt}  \sum_{k=\lceil (N+r_0)/\Lambda \rceil}^{ \lfloor (X+r_0)/\Lambda \rfloor  }    
\chi_2 ((\Lambda k+B)d_2+c_2) 
\sum_{r=r_0}^{q_1-1}
\chi_1(-rd_1 +(\Lambda k+B)d_1+c_1)=0,  \qquad 
\end{gather}
where we have used that $\chi_1$ is a nonprincipal character mod $q_1$, and that $d_1$ and $q_1$ are coprime, so that the inner summation is over all residues mod $q_1$. From \eqref{rtrt}, \eqref{est11} and by using \eqref{prisest} if $d_j$ and $k_j$ is not coprime for $j=1$ or $j=2$ it follows that
\begin{gather} \label{est111}
  \abs{\sum_{r=0}^{q_1-1} \psi_{k_1,k_2}(r;X)} \leq 2 q_1^2 q_2. 
\end{gather}
It follows by \eqref{str0}, \eqref{str}, \eqref{est111}, the triangle inequality, the fact that the multiplicative function $|b_j(k)| \leq 2$  and the fact that the number of $k_j \in \mathcal A_j$ less than $X$ can be bounded by $\abs{\mathcal A_j \cap [0,X]} \leq \Psi(X,\log N_0)$ that
\begin{gather} \label{popp}
   \abs{\Delta_{r_0}(X)}   \leq 8 q_1^2q_2(\Psi(X,\log N_0))^2 \leq 8q_1^2 q_2 X^{\varepsilon}, 
\end{gather}  
where the last inequality follows by  \eqref{nwlpf}, the estimate of number without large prime factors. From \eqref{popp} it follows by partial summation in a similar manner
as in the proof of Lemma \ref{trew} that
\begin{multline}
 \max_{(s_1,s_2) \in \M^2}  \abs{\sum_{r=r_0}^{r_0+q_1-1} \sum_{k=N_1+r_0}^{N_2+r_0} \frac{a_1(k+\alpha_1-r) a_2(k+\alpha_2)}{(k+\alpha_1-r)^{s_1}(k+\alpha_2)^{s_2}}}  \\ \leq C(\varepsilon) (1+|s_1|)(1+|s_2|) (N_1-r_0)^{\varepsilon-\Re(s_1+s_2)},
\end{multline}
which concludes the proof of Lemma \ref{additive}.
\end{proof}

\begin{lem}
   \label{LE10preDim2}
  Let $\ddd$ and $\M$ be defined by \eqref{Kdef}. Suppose that $\alpha_1,\alpha_2$ are transcendental and rational numbers. Suppose that  $\chi_j$ is permissible with respect to $\alpha_j$ and of modulus $q_j$ for $1\leq j \leq m$ and such that $q_1$ and $q_2$ are coprime. Then, for any $\varepsilon >0$ there exist an $N_0$ such that if    $N_0 \leq N_1 \leq N_2$ and $a_j$ is of type $(N_1,\chi_j)$  then
	\begin{gather*} 
  \max_{s \in \M^2} \abs{\sum_{N_1 \leq k_1< k_2  \leq  N_2 }  \prod_{j=1}^2 \frac{a_j(k_j+\alpha_j)}{(k_j+\alpha_j)^{s_{j}}}}  \leq N_1^{-1/2-\ddd+\varepsilon}.
\end{gather*}
\end{lem}
We note that the condition  that $q_1$ and $q_2$ are coprime is essential since the Lemma does not hold in case $a_1(n)=\chi(n)$ and $a_2(n)=\overline{\chi(n)}$ for any character $\chi$ mod $q$. 

\begin{proof}
 We first use a dyadic decomposition with $I_n=[2^n,2^{n+1}] \cap \Z \cap [N_1,N_2]$. It is clear that for $N_1 \geq 1$ we get that
\begin{gather} 
\begin{split}
\sum_{N_1 \leq k_1< k_2  \leq  N_2 }  \prod_{j=1}^2 &\frac{a_j(k_j+\alpha_j)}{(k_j+\alpha_j)^{s_{j}}} = \sum_{0 \leq n_1 < n_2}  \,  \sum_{k_1 \in I_{n_1}}  \,  \sum_{k_2 \in I_{n_2}}  \,  \prod_{j=1}^2 \frac{a_j(k_j+\alpha_j)}{(k_j+\alpha_j)^{s_{j}}} \\
 &+ \sum_{n=0}^\infty \, \, \sum_{\max(2^{n},N_1) \leq k_1< k_2  \leq  \min(2^{n+1},N_2) } \,   \prod_{j=1}^2 \frac{a_j(k_j+\alpha_j)}{(k_j+\alpha_j)^{s_{j}}}. \end{split} \label{orejg}
\end{gather} 
The absolute value of the first term on the right hand side of \eqref{orejg} can be  bounded by $C_\varepsilon N_1^{2 \varepsilon-1-2\ddd}$ by Lemma \ref{trew} for any $(s_1,s_2) \in \M^2$. It is sufficient to prove that the second sum in \eqref{orejg} is bounded by  $N_1^{-1/2-\ddd+\varepsilon}$  for $(s_1,s_2) \in \M^2$ if $N_1$ is sufficiently large. But these sums are of the same type as in the current lemma, with the additional assumption that the upper bound is less than twice the lower bound in the summation. Hence without loss of generality we may assume $N_2 \leq 2 N_1$ since the general case of Lemma \ref{LE10preDim2} will follow from this special case. From now on we will thus assume that  $N_2 \leq 2 N_1$. Let us define 
\begin{gather} \label{lambdajdef}
\Lambda_j(x;s_j)=\sum_{N_1  \leq k<x} \frac{a_j(k+\alpha_j)}{ (k+\alpha_j)^{s_j}},  \\  \intertext{for $j=1,2$ and}   \label{lambdadef} \Lambda(X;\vs)=  \sum_{N_1 \leq k_1<k_2 \leq \min(X,N_2) }  \prod_{j=1}^2 \frac{a_j(k_j+\alpha_j)}{(k_j+\alpha_j)^{s_{j}}}.
\end{gather} With this terminology the statement of the Lemma can be written as 
\begin{gather} \label{prff}
\max_{s \in \M^2} \abs{\Lambda(N_2;\vs)}  \leq N_1^{-1/2-\ddd+\varepsilon},
\end{gather}
provided $N_0=N_0(\varepsilon)$ has been chosen sufficiently large. We can write the following formula
\begin{gather} 
 \Lambda(X+\Delta;\vs)-\Lambda(X;\vs) =  \Lambda_1(X;\vs) (\Lambda_2(X+\Delta;s_1)-\Lambda_2(X;s_2)) + \sum_{d=1}^\infty  \sum_{j=1}^{2^{d-1}} T_{d,j}(X;\vs) \qquad  \label{irw} 
\end{gather} where
\begin{multline} T_{d,j}(X;\vs)=   \p{\Lambda_{1} \p{X+ \frac{(2j-1) \Delta} {2^d};s_1 } - \Lambda_{1} \p{X+ \frac{ (2j-2) \Delta}{2^d}};s_1} \times   \\ \times
\p{\Lambda_2\p{X+ \frac{2j \Delta}{2^d};s_2 } - \Lambda_2\p{X + \frac{(2j-1) \Delta}{2^d}};s_2}, 
\end{multline} 
 and where the first term in \eqref{irw} comes from the case where $N_1 \leq k_1<X$ and where the second term from the dyadic division of $X \leq k_1<k_2 \leq X+\Delta$ illustrated by the figure  
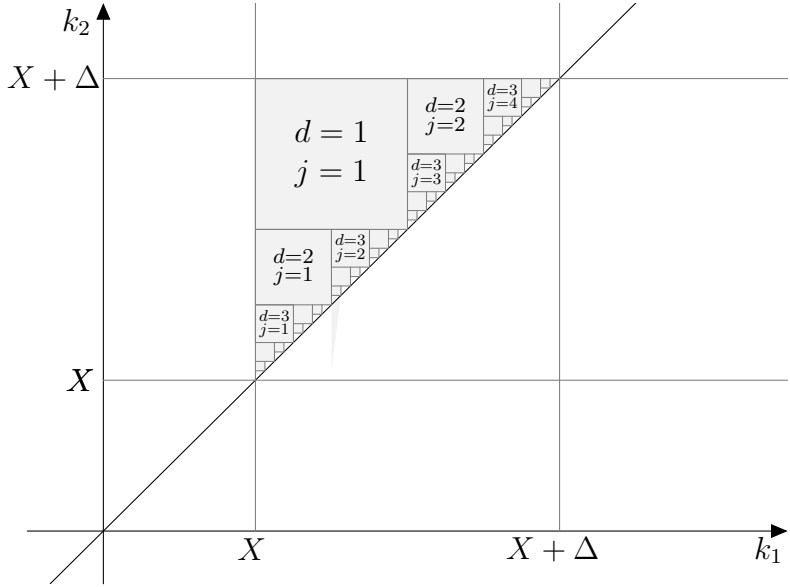
\definecolor{zzttqq}{rgb}{0.5,0.5,0.5}
\definecolor{xdxdff}{rgb}{0.49,0.49,1}
\begin{center}
\begin{tikzpicture}[line cap=round,line join=round,>=triangle 45,x=1.0cm,y=1.0cm]
\draw[->,color=black] (-1,0) -- (9,0);
\draw[->,color=black] (0,-0.7) -- (0,7);
\clip(-1.3,-0.7) rectangle (9,7);
\fill[color=zzttqq,fill=zzttqq,fill opacity=0.1] (2,6) -- (2,4) -- (4,4) -- (4,6) -- cycle;
\fill[color=zzttqq,fill=zzttqq,fill opacity=0.1] (4,6) -- (5,6) -- (5,5) -- (4,5) -- cycle;
\fill[color=zzttqq,fill=zzttqq,fill opacity=0.1] (2,4) -- (3,4) -- (3,3) -- (2,3) -- cycle;
\fill[color=zzttqq,fill=zzttqq,fill opacity=0.1] (2,3) -- (2.5,3) -- (2.5,2.5) -- (2,2.5) -- cycle;
\fill[color=zzttqq,fill=zzttqq,fill opacity=0.1] (5,6) -- (5.5,6) -- (5.5,5.5) -- (5,5.5) -- cycle;
\fill[color=zzttqq,fill=zzttqq,fill opacity=0.1] (3,4) -- (3.5,4) -- (3.5,3.5) -- (3,3.5) -- cycle;
\fill[color=zzttqq,fill=zzttqq,fill opacity=0.1] (4,5) -- (4.5,5) -- (4.5,4.5) -- (4,4.5) -- cycle;
\fill[color=zzttqq,fill=zzttqq,fill opacity=0.1] (2,2.5) -- (2,2.25) -- (2.25,2.25) -- (2.25,2.5) -- cycle;
\fill[color=zzttqq,fill=zzttqq,fill opacity=0.1] (2.5,3.0) -- (2.5,2.75) -- (2.75,2.75) -- (2.75,3.0) -- cycle;
\fill[color=zzttqq,fill=zzttqq,fill opacity=0.1] (3,3.5) -- (3,3.25) -- (3.25,3.25) -- (3.25,3.5) -- cycle;
\fill[color=zzttqq,fill=zzttqq,fill opacity=0.1] (3.5,4.0) -- (3.5,3.75) -- (3.75,3.75) -- (3.75,4.0) -- cycle;
\fill[color=zzttqq,fill=zzttqq,fill opacity=0.1] (4,4.5) -- (4,4.25) -- (4.25,4.25) -- (4.25,4.5) -- cycle;
\fill[color=zzttqq,fill=zzttqq,fill opacity=0.1] (4.5,5.0) -- (4.5,4.75) -- (4.75,4.75) -- (4.75,5.0) -- cycle;
\fill[color=zzttqq,fill=zzttqq,fill opacity=0.1] (5,5.5) -- (5,5.25) -- (5.25,5.25) -- (5.25,5.5) -- cycle;
\fill[color=zzttqq,fill=zzttqq,fill opacity=0.1] (5.5,6.0) -- (5.5,5.75) -- (5.75,5.75) -- (5.75,6.0) -- cycle;
\fill[color=zzttqq,fill=zzttqq,fill opacity=0.1] (2,2.25) -- (2,2.125) -- (2.125,2.125) -- (2.125,2.25) -- cycle;
\fill[color=zzttqq,fill=zzttqq,fill opacity=0.1] (3,3.25) -- (3,2.125) -- (3.125,3.125) -- (3.125,3.25) -- cycle;
\fill[color=zzttqq,fill=zzttqq,fill opacity=0.1] (4,4.25) -- (4,4.125) -- (4.125,4.125) -- (4.125,4.25) -- cycle;
\fill[color=zzttqq,fill=zzttqq,fill opacity=0.1] (5,5.25) -- (5,5.125) -- (5.125,5.125) -- (5.125,5.25) -- cycle;
\fill[color=zzttqq,fill=zzttqq,fill opacity=0.1] (2.5,2.75) -- (2.5,2.625) -- (2.625,2.625) -- (2.625,2.75) -- cycle;
\fill[color=zzttqq,fill=zzttqq,fill opacity=0.1] (3.5,3.75) -- (3.5,3.625) -- (3.625,3.625) -- (3.625,3.75) -- cycle;
\fill[color=zzttqq,fill=zzttqq,fill opacity=0.1] (4.5,4.75) -- (4.5,4.625) -- (4.625,4.625) -- (4.625,4.75) -- cycle;
\fill[color=zzttqq,fill=zzttqq,fill opacity=0.1] (5.5,5.75) -- (5.5,5.625) -- (5.625,5.625) -- (5.625,5.75) -- cycle;
\fill[color=zzttqq,fill=zzttqq,fill opacity=0.1] (2.25,2.5) -- (2.25,2.375) -- (2.375,2.375) -- (2.375,2.5) -- cycle;
\fill[color=zzttqq,fill=zzttqq,fill opacity=0.1] (3.25,3.5) -- (3.25,3.375) -- (3.375,3.375) -- (3.375,3.5) -- cycle;
\fill[color=zzttqq,fill=zzttqq,fill opacity=0.1] (4.25,4.5) -- (4.25,4.375) -- (4.375,4.375) -- (4.375,4.5) -- cycle;
\fill[color=zzttqq,fill=zzttqq,fill opacity=0.1] (5.25,5.5) -- (5.25,5.375) -- (5.375,5.375) -- (5.375,5.5) -- cycle;
\fill[color=zzttqq,fill=zzttqq,fill opacity=0.1] (2.75,3.0) -- (2.75,2.875) -- (2.875,2.875) -- (2.875,3.0) -- cycle;
\fill[color=zzttqq,fill=zzttqq,fill opacity=0.1] (3.75,4.0) -- (3.75,3.875) -- (3.875,3.875) -- (3.875,4.0) -- cycle;
\fill[color=zzttqq,fill=zzttqq,fill opacity=0.1] (4.75,5.0) -- (4.75,4.875) -- (4.875,4.875) -- (4.875,5.0) -- cycle;
\fill[color=zzttqq,fill=zzttqq,fill opacity=0.1] (5.75,6.0) -- (5.75,5.875) -- (5.875,5.875) -- (5.875,6.0) -- cycle;
\draw [domain=-0.7:9.0] plot(\x,{\x});
\draw [color=zzttqq] (2,2.125)-- (2.125,2.125);
\draw [color=zzttqq] (2.125,2.125)-- (2.125,2.25);
\draw [color=zzttqq] (3,3.125)-- (3.125,3.125);
\draw [color=zzttqq] (3.125,3.125)-- (3.125,3.25);
\draw [color=zzttqq] (4,4.125)-- (4.125,4.125);
\draw [color=zzttqq] (4.125,4.125)-- (4.125,4.25);
\draw [color=zzttqq] (5,5.125)-- (5.125,5.125);
\draw [color=zzttqq] (5.125,5.125)-- (5.125,5.25);
\draw [color=zzttqq] (2.5,  2.625)-- (2.625,2.625);
\draw [color=zzttqq] (2.625,2.625)-- (2.625,2.75);
\draw [color=zzttqq] (3.5,  3.625)-- (3.625,3.625);
\draw [color=zzttqq] (3.625,3.625)-- (3.625,3.75);
\draw [color=zzttqq] (4.5,  4.625)-- (4.625,4.625);
\draw [color=zzttqq] (4.625,4.625)-- (4.625,4.75);
\draw [color=zzttqq] (5.5,  5.625)-- (5.625,5.625);
\draw [color=zzttqq] (5.625,5.625)-- (5.625,5.75);
\draw [color=zzttqq] (2.75,  2.875)--  (2.875, 2.875);
\draw [color=zzttqq] (2.875, 2.875)--  (2.875, 3.0);
\draw [color=zzttqq] (3.75,  3.875)--  (3.875, 3.875);
\draw [color=zzttqq] (3.875, 3.875)--  (3.875, 4.0);
\draw [color=zzttqq] (4.75,  4.875)--  (4.875, 4.875);
\draw [color=zzttqq] (4.875, 4.875)--  (4.875, 5.0);
\draw [color=zzttqq] (5.75,  5.875)--  (5.875, 5.875);
\draw [color=zzttqq] (5.875, 5.875)--  (5.875, 6.0);
\draw [color=zzttqq] (2.25, 2.375)-- (2.375,2.375);
\draw [color=zzttqq] (2.375,2.375)-- (2.375,2.5);
\draw [color=zzttqq] (3.25, 3.375)-- (3.375,3.375);
\draw [color=zzttqq] (3.375,3.375)-- (3.375,3.5);
\draw [color=zzttqq] (4.25, 4.375)-- (4.375,4.375);
\draw [color=zzttqq] (4.375,4.375)-- (4.375,4.5);
\draw [color=zzttqq] (5.25, 5.375)-- (5.375,5.375);
\draw [color=zzttqq] (5.375,5.375)-- (5.375,5.5);
\draw [color=zzttqq] (2.25,2.25)-- (2.25,2.5);
\draw [color=zzttqq] (2,2.25)-- (2.25,2.25);
\draw [color=zzttqq] (2.75,2.75)-- (2.75,3.0);
\draw [color=zzttqq] (2.5,2.75)-- (2.75,2.75);
\draw [color=zzttqq] (3.25,3.25)-- (3.25,3.5);
\draw [color=zzttqq] (3,3.25)-- (3.25,3.25);
\draw [color=zzttqq] (3.75,3.75)-- (3.75,4.0);
\draw [color=zzttqq] (3.5,3.75)-- (3.75,3.75);
\draw [color=zzttqq] (4.25,4.25)-- (4.25,4.5);
\draw [color=zzttqq] (4,4.25)-- (4.25,4.25);
\draw [color=zzttqq] (4.75,4.75)-- (4.75,5.0);
\draw [color=zzttqq] (4.5,4.75)-- (4.75,4.75);
\draw [color=zzttqq] (5.25,5.25)-- (5.25,5.5);
\draw [color=zzttqq] (5,5.25)-- (5.25,5.25);
\draw [color=zzttqq] (5.75,5.75)-- (5.75,6.0);
\draw [color=zzttqq] (5.5,5.75)-- (5.75,5.75);
\draw [color=zzttqq] (2,6)-- (2,4);
\draw [color=zzttqq] (2,4)-- (4,4);
\draw [color=zzttqq] (4,4)-- (4,6);
\draw [color=zzttqq] (4,6)-- (2,6);
\draw [color=zzttqq] (4,6)-- (5,6);
\draw [color=zzttqq] (5,6)-- (5,5);
\draw [color=zzttqq] (5,5)-- (4,5);
\draw [color=zzttqq] (4,5)-- (4,6);
\draw [color=zzttqq] (2,4)-- (3,4);
\draw [color=zzttqq] (3,4)-- (3,3);
\draw [color=zzttqq] (3,3)-- (2,3);
\draw [color=zzttqq] (2,3)-- (2,4);
\draw [color=zzttqq] (2,3)-- (2.5,3);
\draw [color=zzttqq] (2.5,3)-- (2.5,2.5);
\draw [color=zzttqq] (2.5,2.5)-- (2,2.5);
\draw [color=zzttqq] (2,2.5)-- (2,3);
\draw [color=zzttqq] (5,6)-- (5.5,6);
\draw [color=zzttqq] (5.5,6)-- (5.5,5.5);
\draw [color=zzttqq] (5.5,5.5)-- (5,5.5);
\draw [color=zzttqq] (5,5.5)-- (5,6);
\draw [color=zzttqq] (3,4)-- (3.5,4);
\draw [color=zzttqq] (3.5,4)-- (3.5,3.5);
\draw [color=zzttqq] (3.5,3.5)-- (3,3.5);
\draw [color=zzttqq] (3,3.5)-- (3,4);
\draw [color=zzttqq] (4,5)-- (4.5,5);
\draw [color=zzttqq] (4.5,5)-- (4.5,4.5);
\draw [color=zzttqq] (4,4.5)-- (4,5);
\draw [color=zzttqq] (0,2) -- (9,2);
\draw [color=zzttqq] (4.5,4.5)-- (4,4.5);
\draw  [color=zzttqq] (0,6) -- (9,6);
\draw [color=zzttqq] (2,0) -- (2,7);
\draw  [color=zzttqq] (6,0) -- (6,7);
\draw[color=black] (5.91,-0.24) node {$X+\Delta$};
\draw[color=black] (-0.65,6) node {$X+\Delta$};
\draw[color=black] (1.95,-0.24) node {$X$};
\draw[color=black] (-0.3,2) node {$X$};
\draw[color=black] (-0.3,2) node {$X$};
\draw[color=black] (8.75,-0.25) node {$k_1$};
\draw[color=black] (-0.35,6.75) node {$k_2$};
\draw[color=black] (3,5) node {$\displaystyle \begin{matrix} d=1 \\ j=1 \end{matrix}$};  
\draw[color=black] (4.5,5.5) node {$\displaystyle \substack{ d=2 \\ j=2}$};
\draw[color=black] (2.5,3.5) node {$\displaystyle \substack{ d=2 \\ j=1}$}; 
\draw[color=black] (2.25,2.75) node {\scalebox{0.75}{$\displaystyle \substack{ d=3 \\ j=1}$}}; 
\draw[color=black] (3.25,3.75) node  {\scalebox{0.75}{$\displaystyle\substack{ d=3 \\ j=2}$}}; 
\draw[color=black] (4.25,4.75) node {\scalebox{0.75}{$\displaystyle\substack{ d=3 \\ j=3}$}}; 
\draw[color=black] (5.25,5.75) node {\scalebox{0.75}{$\displaystyle\substack{ d=3 \\ j=4}$}}; 
\end{tikzpicture}
\end{center}
By Lemma \ref{trew} it follows that
 \begin{gather} \label{Lambdaestimate} \max_{s_j \in \M} \abs{\Lambda_j(X+\Delta;\vs)-\Lambda_j(X;\vs)} \ll_\varepsilon X^{\varepsilon-\ddd-1/2}, \qquad (j=1,2). \end{gather}
It also follows that \begin{gather} \max_{\vs \in \M^2} \abs{T_{d,j}(X;\vs)} \ll_{\varepsilon}  X^{2\varepsilon-2\ddd-1}. \end{gather} 
This  allow us to bound the sum for small $d$  (this corresponds to pairs $(k_1,k_2)$ in the original sum that are far from  the diagonal)   in \eqref{irw}  and by the triangle inequality we get that
\begin{gather}
 \max_{\vs \in \M^2} \abs{\Lambda(X+\Delta;\vs)-\Lambda(X;\vs)} \ll_{\varepsilon} 2^M N_1^{\varepsilon-2\ddd-1}  +\abs{ \sum_{d=M}^\infty  \sum_{j=1}^{2^{d-1}} T_{d,j}(X;\vs)},  \qquad  \p{\frac{N_1} 2 \leq X}  \qquad   \label{irw2} \end{gather}
We will treat the case of large $d$ (this corresponds to pairs $(k_1,k_2)$ in the original sum that are close to the diagonal) with Lemma \ref{additive}. Before proceeding to treat this case we notice that by the definitions \eqref{lambdajdef},\eqref{lambdadef} we have that
\begin{gather}
   \Lambda(N_2;\vs)  =\sum_{N_1 \leq k_1< k_2  \leq  N_2 }\prod_{j=1}^2 \frac{a_j(k_j+\alpha_j)}{(k_j+\alpha_j)^{s_{j}}} =\Lambda(M_2;\vs)-\Lambda(M_1;\vs)  
\end{gather}
whenever $M_1\leq N_1$ and $N_2\leq M_2$. Thus
\begin{gather} \label{stardef}
  \Lambda(N_2;\vs) = \frac 1 H \int_{N_1-H}^{N_1} \left(\Lambda(X+\Delta;\vs)-\Lambda(X;\vs) \right) dX , \\ \intertext{where } H=\frac{N_2-N_1}{2^M-1}, \qquad \Delta=2^MH. \label{Hdef}
\end{gather}
since $X+\Delta \geq N_2$ and $X \leq N_1$ whenever $X$ is in the interval of integration.  It follows that 
\begin{align} \label{ytreraj}
 J(\vs) &= \frac 1 H  \int_{N_1-H}^{N_1}  \p{ \sum_{d=M}^\infty  \sum_{j=1}^{2^{d-1}} T_{d,j}(X;\vs)} dX = \sum_{r=1}^{\lfloor H \rfloor } \left(1-\frac r H \right) \ddddd_r(\vs), \\ \intertext{where} \ddddd_r(\vs) &=  \sum_{n=N_1}^{N_2-H} \frac{a_1(n+\alpha_1)a_2(n+\alpha_2+r)}{ (n+\alpha_1)^{s_1}(n+\alpha_2+r)^{s_2}}. 
\end{align}
We now write the sum as follows 
\begin{gather} \label{trt7}
 J(\vs) =  \sum_{r=1}^{\lfloor H \rfloor } \p{1-   \frac {q_1} H \left \lfloor \frac r {q_1} \right \rfloor } \ddddd_r(\vs) +   
 \frac 1 H \sum_{r=1}^{\lfloor H \rfloor} \p{q_1 \left \lfloor \frac r {q_1}  \right \rfloor -r }\ddddd_r(\vs).  \end{gather}
By Lemma \ref{additive}  the first sum in \eqref{trt7} can be estimated by
\begin{gather}
 \label{frstsum}
  \max_{\vs \in \M^2} \abs{\sum_{r=1}^{\lfloor H \rfloor } \p{1-   \frac {q_1} H \left \lfloor \frac r {q_1} \right \rfloor } \ddddd_r(\vs)} \ll_{\varepsilon} H N_1^{\varepsilon-2\ddd-1}.
\end{gather}
 By dividing the second sum in \eqref{trt7} into residue classes mod $q_1$ and using
\begin{gather}
 \max_{s \in \M}  \sum_{j=1}^{q_1} \abs{\sum_{\substack{r \leq H \\ r \equiv j \pmod {q_1}}} \frac{a_2(n+\alpha_2+r)}{(n+\alpha_2+r)^{s_2}}} \ll_{\varepsilon}  n^{\varepsilon -1/2-\ddd}, 
\end{gather}
which since $q_1$ and $q_2$ is coprime is a consequence of applying Lemma \ref{trew} with $\alpha=(j+\alpha_2)/q_1$,  we obtain after summing over $n$ that
\begin{gather}
 \label{scndsum}
  \max_{\vs \in \M^2} \abs{ \frac 1 H \sum_{r=1}^{\lfloor H \rfloor} \p{q_2 \left \lfloor \frac r {q_2}  \right \rfloor -r }\ddddd_r(\vs)}  \ll_{\varepsilon}  \frac{N_1^{\varepsilon-2 \ddd}} H. 
\end{gather}  
The equation \eqref{trt7}, the inequalities \eqref{frstsum},  \eqref{scndsum} and the triangle inequality implies that
\begin{gather} \label{estoo}
   \max_{\vs \in \M^2} \abs{J(\vs)} \ll_{\varepsilon} H N_1^{\varepsilon-2\ddd-1}+ \frac{N_1^{\varepsilon-2 \ddd}} H.
\end{gather}
Now \eqref{Hdef} and the initial assumption that $N_2 \leq 2 N_1$ implies that $H \leq N_1 2^{1-M}$. By combining \eqref{stardef} and the inequalities \eqref{irw2} and \eqref{estoo} it follows that 
\begin{gather} 
\max_{\vs \in \M^2} \abs{ \Lambda(N_2;\vs)} \ll_\varepsilon 2^M N_1^{\varepsilon-2\ddd-1}+ N_1^{\varepsilon-2 \ddd} 2^{-M}.
\end{gather}
This inequality together with the choice of
$$
 M=\frac 1 2 \lceil \log _2 N_1 \rceil,
$$
 where $\log_2 x$ denote the logarithm with base $2$ implies that 
\begin{gather}
 \max_{\vs \in \M^2} \abs{\Lambda(N_2;\vs)} \leq C_\varepsilon N_1^{\varepsilon-2\ddd-1/2},
\end{gather}
for some $C_\varepsilon>0$. The choice $N_0 = \lceil C_{\varepsilon}^{-1/\ddd} \rceil$ implies the inequality \eqref{prff}   which concludes the proof  of Lemma \ref{LE10preDim2}.
\end{proof}

\begin{lem}
   \label{LE10pre}
  Let $\ddd$ and $\M$ be defined by \eqref{Kdef}. Suppose that $\{\alpha_j\}_{j=1}^m$ are transcendental and rational numbers. Suppose that  $\chi_j$ is permissible with respect to $\alpha_j$ and of modulus $q_j$ for $1\leq j \leq m$ and such that $q_j$ and $q_{j+1}$ are coprime for $1 \leq j \leq m-1$. Then, for any $\varepsilon >0$ there exist an $N_0$ such that if    $N_0 \leq N_1 \leq N_2$ and $a_j$ is of type $(N_1,\chi_j)$  then
	\begin{gather*} 
  \max_{s \in \M^m} \abs{\sum_{N_1 \leq k_1<\cdots < k_m < N_2 }  \prod_{j=1}^m \frac{a_j(k_j+\alpha_j)}{(k_j+\alpha_j)^{s_{j}}}}  \leq N_1^{-1/2-\ddd+\varepsilon},
\end{gather*}
\end{lem}

 \begin{proof}
  We give a proof by strong induction where we need two base cases, $m=1$ and $m=2$
\begin{description}[leftmargin=18pt]
\item[Base case $m=1$:] This is a consequence of Lemma \ref{trew} and the definition of the set $\M$ ( Eq. \eqref{Kdef}). 
\item[Base case $m=2$:] This is Lemma \ref{LE10preDim2}. 
\item[Induction step:] As an induction hypothesis we assume that the Lemma holds for each positive integer strictly less than $m=p$.  Assume that $m=p+1 \geq 3$ and choose some $1<v<p+1$. Without loss of generality we may assume that $0<\varepsilon<2 \ddd$ and choose $N_0$ sufficiently large so that
\begin{gather} \label{pr}
 \sum_{k=N_0}^{\infty}(k+\alpha_v)^{-1/2- \ddd}   k^{-1/2-\ddd+\varepsilon} \leq \sum_{k=N_0}^{\infty}k^{-1-2 \ddd+\varepsilon}  \leq 1.
\end{gather}
Let us rewrite the summation (with $k=k_v$) as
\begin{multline}
 (*)= \abs{\sum_{N_1 \leq k_1<\cdots < k_{p+1} < N_2 }  \prod_{j=1}^{p+1} \frac{a_j(k_j+\alpha_j)} {(k_j+\alpha_j)^{s_{j}}}} = \\  
\abs{\sum_{k=N_1}^{N_2}   \frac{a_v(k+\alpha_v)}{(k+\alpha_v)^{s_v}} \p{\sum_{N_1 < k_1< \cdots k_{v-1}<k}   \frac{a_j(k_j+\alpha_j)}{(k_j+\alpha_j)^{s_{j}}}}  \p{\sum_{k < k_{v+1}< \cdots k_{p+1}<N_2}  \frac{a_j(k_j+\alpha_j)}{(k_j+\alpha_j)^{s_{j}}}}}.  
\end{multline} 
By the triangle inequality and by using the facts that  $s_v \in \M$ and $a_v$ is a unimodular function on the first factor in each term, and the induction hypothesis on the second two factors in each term  this can be bounded for each choice of $(s_1,\ldots,s_m) \in \M^m$ in absolute values by
\begin{gather} \begin{split}
  (*)&\leq   \sum_{k=N_1}^{N_2-1}   \frac{1}{(k+\alpha_v)^{1/2+\ddd}} \times N_1^{-1/2-\ddd+\varepsilon}   \times k^{-1/2-\ddd+\varepsilon} \\ &= N_1^{-1/2-\ddd+\varepsilon} \sum_{k=N_1}^{N_2} (k+\alpha_v)^{-1/2- \ddd}   k^{-1/2-\ddd+\varepsilon}  \leq N_1^{-1/2-\ddd+\varepsilon}, \end{split}
\end{gather}
where the last inequality follows from \eqref{pr} since $N_1 \geq N_0$.
\end{description}
\end{proof}
A consequence of Lemma \ref{LE10pre} which is useful in the proof of the fundamental Lemma and will also  find its use in the proof of Lemma \ref{LE6} is the following result.
\begin{lem} \label{EEEV} Let $\alpha>0$ be a transcendental or rational number and let $a:\N +\alpha \to \C$ be of type $(N_0,\chi)$. Then for any $\varepsilon>0$ there exist a $C(\varepsilon)>0$ such that
\begin{gather}
\sum_{0 \leq k_1 < \cdots <k_{n} \leq N} \prod_{j=1}^n \frac{a_j(k_j+\alpha_j)}{ (k_j+\alpha_j)^{s_j}} =\zeta_\va(\vs;\val)+ \newsym{An error term of order $O_\varepsilon(N^{\varepsilon-1/2-\ddd})$}{E_{\va,N}(\vs;\val)},
\end{gather} where   $\zeta_\va(\vs;\val)$ is an analytic function on $\M^n$ and
\begin{gather} 
\max_{\vs \in \M^n} \abs{E_{\va,N}(\vs;\val)} \leq C(\varepsilon) N^{\varepsilon -1/2-\ddd},  \qquad (N \in \Z^+). 
 \end{gather}
\end{lem}
\begin{proof} 
It is sufficient to prove that
\begin{gather} \label{ui} 
\max_{\vs \in \M^n} \abs{\zeta_{\va}^{[N_1]}(\vs;\val) - \zeta_{\va}^{[N]}(\vs;\val)}\leq C(\varepsilon)  N^{\varepsilon-1/2-\ddd}, \qquad (1 \leq N < N_1)  \\ \intertext{where}
\zeta_{\va}^{[N]}(\vs;\val)=
  \sum_{0 \leq k_1 < \cdots <k_{n} \leq N} \prod_{j=1}^n \frac{a_j(k_j+\alpha_j)}{ (k_j+\alpha_j)^{s_j}},
\end{gather}
since the condition \eqref{ui} gives us a Cauchy sequence and thus the sequence \begin{gather} \lim_{N \to \infty} \zeta_{\va}^{[N]}(\vs;\val)=\zeta_\va(\vs;\val) \end{gather} is uniformly convergent to an analytic function on $\M^n$ and the Lemma follows from \eqref{ui} when taking the limit $N_1 \to \infty$. We now use a proof by induction to prove the estimate \eqref{ui}. The base case $n=1$ is an immediate consequence of Lemma \ref{trew}. As the induction assumption we assume that \eqref{ui} is true for $n=p$. By dividing the sum occurring in \eqref{ui} for $n=p+1$ according to how many $k_j$ are greater than $N$ we get 
\begin{multline}
\zeta_{\va}^{[N_1]}(\vs;\val) - \zeta_{\va}^{[N]}(\vs;\val)  \\ = \sum_{v=1}^p \p{
\sum_{1 \leq k_1 < \cdots <k_{v} \leq N} \prod_{j=1}^{v} \frac{a_j(k_j+\alpha_j)}{ (k_j+\alpha_j)^{s_j}}} \p{ \sum_{N < k_{v+1} < \cdots <k_{p+1} \leq N_1} \prod_{j=v+1}^{p+1} \frac{a_j(k_j+\alpha_j)}{ (k_j+\alpha_j)^{s_j}} }.
\end{multline}
By the induction assumption the first factor in the sum is bounded  on $\M^v$ and by Lemma \ref{LE10pre} the second factor is bounded by $BN_1^{\varepsilon -\ddd-1/2}$ whenever $(s_{v+1},\ldots,s_{p+1}) \in \M^{p+1-v}$ for some constant $B>0$. 
 This concludes our proof.
\end{proof}

\begin{lem}   \label{LE10}
  Suppose that $\{\alpha_j\}_{j=1}^n$ are transcendental and rational numbers such that  $\chi_j$ is permissible with respect to $\alpha_j$ and of modulus $q_j$ such that $q_j$ and $q_{j+1}$ are coprime for $1 \leq j \leq n-1$ and let $1\leq j_0\leq n$.  Then, for any $\varepsilon^* >0$ there exist an $N_0$ such that if  $N_0<N_1 \leq N_2$ and   $a_j$ is of type $(N_2,\chi_j)$ if $j=j_0$ and of type $(N_1,\chi_j)$ otherwise, then
	\begin{gather*} 
  \max_{1 \leq v <w \leq n} \max_{s \in \M^m} \abs{\sum_{N_1 \leq k_v<\cdots < k_w < N_2}  \prod_{j=v}^w \frac{a_j(k_j+\alpha_j)}{(k_j+\alpha_j)^{s_{j}}}}  < (1+\Delta)  \varepsilon^*, \\ \intertext{where}
   \Delta= \max_{s \in \M^n} \abs{\sum_{k=N_1}^{N_2} \frac{a_{j_0}(k+\alpha_{j_0})}{(k+\alpha_{j_0})^{s}}}. \end{gather*}
\end{lem} 
\begin{proof}
First choose $N_0$ sufficiently large such that
 \begin{gather} \label{ajj}
  \sum_{k=N_0}^\infty (k+\alpha_{j_0})^{-1/2-\ddd} k^{-1/2} < \varepsilon^*
 \end{gather}
and such that Lemma \ref{LE10pre} holds for that choice of $N_0$ with $\varepsilon=\ddd$.
 Assume that $1 \leq v <w \leq n$ is given and consider the sum
\begin{gather*}
(*)=\sum_{N_1 \leq   k_v<\cdots < k_w < N_2 }  \prod_{v=1}^w \frac{a_j(k_j+\alpha_j)}{(k_j+\alpha_j)^{s_{j}}}
\end{gather*}
 We have the following possibilities
\begin{enumerate}[leftmargin=18pt]
\item  $j_0 \not \in [v,w]$:
 In this case the inequality follows trivially from  Lemma \ref{LE10pre}, in fact we get the stronger result $\abs{(*)} \leq N_0^{-1/2-\ddd+\varepsilon}<\varepsilon^*$, provided that $N_0$ has been chosen sufficiently large.
\item $j_0=v$. This case follows by rewriting the sum and applying the triangle inequality 
\begin{align} 
 \abs{(*)}&=\abs{\sum_{k=N_1}^{N_2}   \frac{a_{v}(k+\alpha_{v})}{(k+\alpha_{v})^{s_{v}}}    \p{\sum_{k < k_{v+1}< \cdots k_{w}<N_2}  \frac{a_j(k_j+\alpha_j)}{(k_j+\alpha_j)^{s_{j}}}}} \\
          &\leq \sum_{k=N_1}^{N_2} {\abs{k+\alpha_{v}}^{-1/2-\ddd}} k^{\varepsilon-\ddd-1/2} \leq \varepsilon^*,
\end{align}
where the inequality for the first factor follow from the fact that $a_j$ is unimodular and  $s_j \in \M$,   where the inequality for the second factor follows by Lemma \ref{LE10pre}, and the final inequality follows from \eqref{ajj}.
\item $v < j_0<w$: In the same way as in the induction step in the proof of Lemma \ref{LE10pre} the sum can be rewritten as $(*)=$
\begin{align} 
 \sum_{k=N_1}^{N_2}   \frac{a_{j_0}(k+\alpha_{j_0})}{(k+\alpha_{j_0})^{s_{j_0}}} \p{\sum_{N_1 < k_v< \cdots k_{j_0-1}<k}   \frac{a_j(k_j+\alpha_j)}{(k_j+\alpha_j)^{s_{j}}}}  \p{\sum_{k < k_{j_0+1}< \cdots k_{w}<N_2}  \frac{a_j(k_j+\alpha_j)}{(k_j+\alpha_j)^{s_{j}}}}.
\end{align}
We now apply the triangle inequality.  Since $a_j$ is unimodular and $s_j \in \M$ the first factor in each term can be bounded by $(k+\alpha_{j_0})^{-1/2-\ddd}$. The second factor in each term is bounded by $N_1^{-1/2-\ddd+\varepsilon} \leq 1$ if $v<j_0$ by Lemma \ref{LE10pre}.  The third factor in each term is bounded  by $k^{-1/2-\ddd+\varepsilon}$ by Lemma \ref{LE10pre}. Since $N_1 \geq N_0$ it follows by \eqref{ajj} that
$$
 \abs{(*)} \leq \sum_{k=N_1}^{N_2} (k+\alpha_{j_0})^{-1/2-\ddd}\times 1 \times k^{-1/2} \leq \varepsilon^*.
$$
\item $j_0=w$: This is the most difficult case to treat. We rewrite the sum in the last variable as
\begin{gather}
  \sum_{k_{w-1}<k_w} \frac{a_w(k_w+\alpha_w)}{(k_w+\alpha_w)^{s_w}}=\Delta(s_w)-\sum_{N_1 \leq k_w \leq  k_{w-1} }(k_w+\alpha_w)^{-s_w}, \\ \intertext{where}
\Delta(s)=\sum_{k=N_1}^{N_2} \frac{a_w(k+\alpha_w)}{(k+\alpha_w)^{s}}. 
\end{gather}
and we can now write the sum $(*)$ as
\begin{multline} \label{ytre1}
 (*)= \Delta(s) \sum_{N_1 \leq k_v< \cdots < k_{w-1} \leq N_2} \prod_{j=v}^{w-1} \frac{a_j(k_j+\alpha_j)}{(k_j+\alpha_j)^{s_j}}  \\ -\sum_{\substack{N_1 \leq k_v <  \cdots  < k_{w-1} \leq N_2 \\ N_1 \leq k_w \leq k_{w-1}}}  \prod_{j=1}^{w} \frac{a_j(k_j+\alpha_j)}{(k_j+\alpha_j)^{s_j}} (k_w+\alpha_w)^{-s_w}, 
\end{multline}
The first term on the right hand side of \eqref{ytre1} can be estimated  as
\begin{gather} \label{ytreojoj}
 \abs{ \Delta(s) \sum_{N_1 \leq k_v< \cdots < k_{w-1} \leq N_2} \prod_{j=v}^{w-1} \frac{a_j(k_j+\alpha_j)}{(k_j+\alpha_j)^{s_j}} } < \Delta \varepsilon^*
 \end{gather}
 by Lemma \ref{LE10pre} provided $N_0$ has been chosen sufficiently large. The second term in the right hand side of \eqref{ytre1} can be rewritten as
\begin{multline} \label{ytre2} 
 \sum_{\substack{N_1 \leq k_v< \cdots  < k_{w-1} \leq N_2 \\ N_1 \leq k_w \leq k_{w-1}}}\prod_{j=v}^{w-1} \frac{a_j(k_j+\alpha_j)}{(k_j+\alpha_j)^{s_j}}    \\ =   \sum_{N_1 \leq  k_v < \cdots < k_{w-2} <k_{w} < k_{w-1} \leq N_2} \prod_{j=v}^{w-1} \frac{a_j(k_j+\alpha_j)}{(k_j+\alpha_j)^{s_j}}  + (\star), 
\end{multline}
where
\begin{gather} (\star)= \sum_{\substack{N_1 \leq k_v< \ldots < k_{w-2} <k_{w-1}\\ N_1 \leq k_w \leq k_{w-2}}}\prod_{j=v}^{w-1} \frac{a_j(k_j+\alpha_j)}{(k_j+\alpha_j)^{s_j}}.
\end{gather}
The absolute value of the first term is less than $\varepsilon^*/2$  by step 2  ($j_0<w$) in the proof of this result since we have  rearranged the $k_j$ and $\varepsilon^*$ can be chosen arbitrarily small. In case $w-v=3$ then $|(\star)|$ can be estimated in a similar manner by  $\varepsilon^*/2$. In case $w-v \geq 4$ we rewrite the final term (with $k=k_{w-2}$) as
\begin{multline*}
(\star)=\sum_{\substack{N_1 \leq k_v< \ldots < k_{w-2} <k_{w-1}\\ N_1 \leq k_w \leq k_{w-2}}}\prod_{j=v}^{w-1} \frac{a_j(k_j+\alpha_j)}{(k_j+\alpha_j)^{s_j}}   =
\sum_{k = N_1}^{N_2} \frac{a_{w-2}(k+\alpha_{w-1})}{(k+\alpha_{w-2})^{s_{w-2}}} \times  \\ \times 
\left( \sum_{N_1 \leq k_v< \ldots < k_{w-3}<k} \prod_{j=v}^{w-3} \frac{a_j(k_j+\alpha_j)}{(k_j+\alpha_j)^{s_j}}  \right) \times \\  \times   \left( \sum_{k<k_{w-1} \leq N_1}  \frac{a_{w-1}(k_{w-1}+\alpha_{w-1})}{(k_{w-1}+\alpha_{w-1})^{s_{w-1}}} \right) \left( \sum_{N_1 \leq k_w \leq  k}  \frac{a_w(k_w+\alpha_w)}{(k_w+\alpha_w)^{s_w}} \right)
\end{multline*}
The absolute value of the first factor in the sum can be estimated by $k^{-1/2-\ddd}$, the second and the third term can be estimated by Lemma \ref{LE10pre} by $k^{-1/2-\ddd+\varepsilon}$, and the fourth term can be estimated by $2 k^{1/2-\ddd}$. It follows that
\begin{gather} \label{starr}
 |(\star)| \leq \sum_{k = N_1}^{N_2} k^{-1/2-\ddd} \times k^{-1/2-\ddd+\varepsilon} \times k^{-1/2-\ddd+\varepsilon} \times 2 k^{1/2-\ddd} = 2 \sum_{k = N_1}^{N_2} k^{-1-2\ddd+2 \varepsilon} \leq \frac{\varepsilon^*} 2  \qquad
\end{gather}
provided that $0<\varepsilon<\ddd$ and $N_0$ is sufficiently large.
The final result follows from  estimates \eqref{ytre1}, \eqref{ytreojoj}, \eqref{ytre2}, \eqref{starr} and the triangle inequality.
\end{enumerate}
\end{proof}

The following  combinatorial lemma will be crucial for our argument. 
\begin{lem} \label{LE13}
   Let $K  \subset \C \setminus\{0 \}$ be a compact set, $C>0$,  $n \geq 2$ and let $p(\vs)$  be a  polynomial in $n$ variables. Then there exists some monomials $q_{j,m_j}(s_j)$, integer $M$ and a sequence $\{j_m\}_{m=1}^{M}$ with $j_m \in \{1,\ldots,n\}$  so that
  \begin{gather} \label{firstcond}
     p(\vs)= \sum_{1 \leq  m_1 < \cdots < m_n \leq M} \prod_{j=1}^n  q_{j,m_j} (s_j), 
		\\ \intertext{and} \label{nexttofinalcond}
 \sum_{1 \leq m_v <\cdots < m_n \leq M} \prod_{j=v}^n q_{j,m_j} (s_j), \qquad (2 \leq v \leq n),
  \\ \intertext{and such that $q_{j,m}(s)=0$ if $j \neq j_m$. Furthermore we can choose the monomials such that the sums}  \sum_{m=1}^N q_{j,m}(s)
			   \\  \intertext{for $1 \leq N \leq M$ and $1 \leq j \leq n$ are also monomials and either}
			  \label{finalcond1}
			      \min_{s \in K} \abs{\sum_{m=1}^N q_{1,m}(s)}>C \qquad \text{or} \qquad \sum_{m=1}^N q_{1,m}(s) =0, \\   \intertext{for $1 \leq N \leq M$ and} \label{finalcond2} 
                      \max_{s \in K} \abs{\sum_{m=1}^N q_{j,m}(s)} \leq 1, 
         \end{gather}
          for  $1 \leq N \leq M$ and $2 \leq j \leq n$.
	\end{lem}

\begin{proof}
 First we will prove the result without the final  condition. Assume that 
\begin{gather} p(s)=\sum_{l=1}^L a_l \prod_{j=1}^n s_j^{c_{j,l}}. \end{gather}
The result follows by choosing $N=(2n-1)L-1$ and
\begin{align*}
 q_{j,(2n-1)l+j-1}(s_j)&=s_j^{c_{j,l}}, &  q_{j,(2n-1)l-j+1}(s_j)&=-s_j^{c_{j,l}},  & (2 \leq j \leq n, 1 \leq l \leq L),& \\
  q_{1,(2n-1)l}(s_1)&=a_l s_1^{c_{1,l}},  &  q_{1,(2n+1)l}(s_1)&=-a_l s_1^{c_{1,l}}, & (1\leq l \leq L).&
\end{align*}
 To prove the result with those conditions we notice that it is sufficient to multiply the monomials $q_{j,m_j}(s_j)$ for $1 \leq j \leq n-1$ and $1 \leq m_j \leq M$  by some sufficiently large constant $B$ and divide the monomials $q_{n,m_n}(s_n)$  for $1 \leq m_n \leq M$  by  $B^n$. This assures that \eqref{finalcond1} and \eqref{finalcond2} are true while not changing the sum given in \eqref{firstcond}.
\end{proof}
By the construction in the proof of Lemma \ref{LE13} 
for the polynomial $q(s_1,s_2)=s_2+s_1+s_1^2s_2^2$  we arrive at $N_1=11$ and  Table \ref{exmpl}. The final step in the lemma to assure that the final conditions are satisfied gives us
\begin{center}
\begin{tabular}{|c|c|c|c|c|c|c|c|c|c|c|c|}
\hline
$m$              & $1$     & 2  & 3     & 4        & 5 & 6 & 7 & 8 & 9  & $10$ & $11$     \\ \hline 
$q_{1,m}(s_1)$ & 0 & $B$ & 0 & $-B$     &  $0$  & $Bs_1$ &    $0$      & $-Bs_1$ & $0$  & $Bs_1^2$ & $0$ \\[1pt]  \hline
$q_{2,m}(s_2)$ & $\displaystyle -\frac{s_2} B$ & $0$ & $\displaystyle \frac{s_2} B$  & $0$ & $\displaystyle - \frac 1 B$ & $0$ & $\displaystyle \frac 1 B$  & $0$ & $\displaystyle -\frac{s_2^2}  B$ & $0$ & $\displaystyle\frac{s_2^2}B$ \\[6pt]  \hline
\end{tabular} 
\end{center}
for a sufficiently large $B$. It is clear that multiplying the first row by $B$ and the second by $1/B$  does not change the sum \eqref{firstcond}. Similarly for the example $p(s_1,s_2,s_3)=1+s_1s_2s_3$ the first construction (when $B=1$) in the proof of the lemma gives us
\begin{center}
\begin{tabular}{|c|c|c|c|c|c|c|c|c|c|c|c|}
\hline
$m$              & 1   & $2$ & $3$     & $4$        & $5$ & $6$ & $7$ & $8$ & $9$  & $10$ & $11$ % &  $12$  
 \\ \hline 
$q_{1,m}(s_1)$ & 0 & 0 & 1 & 0     &  $0$  & $-1$ &    $0$      & $0$ & $s_1$  & $0$ & $0$ % & $-s_1$
 \\  \hline
$q_{2,m}(s_2)$ & $0$ & $-1$ & $0$  & $1$ & $0$ & $0$ & $0$  & $-s_2$ & $0$ & $s_2$ & $0$ % & $0$
 \\  \hline
$q_{3,m}(s_3)$ & $-1$ & $0$ & $0$  & $0$ & $1$ & $0$ & $-s_3$  & $0$ & $0$ & $0$ & $s_3$ % & $0$ 
\\ \hline
\end{tabular} 
\end{center}

\section{A proof of the fundamental lemma, Lemma \ref{LE4}} \label{fundlem}

\begin{description}[leftmargin=18pt]
\item[Initial step:] To prove the fundamental lemma we need to show that there exist some functions $a_j:\N+\alpha_j \to \C$ of type $(\chi_j,N^*)$  such that for any $N\geq N^*$ then
\begin{gather} \label{fundlemsuff}
  \max_{\vs \in \M^n}   \abs{\zeta_{\va}^{[N]}(\vs;\val)-p(\vs)}<\varepsilon. 
\end{gather}
We first use Lemma \ref{LE3} to find functions   $a_j$ on $\N+\alpha_j$ of type $(L_0,\chi_j)$  for $j=1,\ldots,n$ and  numbers $A_0$, where $A_0$ can be chosen as the maximum of the $n$ different $A$'s that appears in Lemma \ref{LE3}, and $L_0$ can be chosen as the maximum of the $n$ different $N_1$'s that comes from the use of Lemma \ref{LE3} with $N_0=0$ and $\varepsilon=1$ such that
\begin{gather} \label{estaa}
  \max_{1 \leq j \leq n} \max_{s \in \M} \abs{\sum_{k=0}^{\infty} \frac{a_j(k+\alpha_j)}{(k+\alpha_j)^{s}}-A_0-3}<1,  
\end{gather}
By Lemma \ref{EEEV} it follows that
\begin{gather} \label{ggdef} \lim_{N \to \infty} \,
 \sum_{ 0 \leq k_1 <\cdots <k_n < N} \prod_{j=1}^n \frac{a_j(k_j+\alpha_j)}{ (k_j+\alpha_j)^{s_j}} = g(\vs), \qquad \p{\Re(s_j)>\frac{1+\ddd}2 },
\end{gather}
where $g$ is an analytic function on  $\{s \in \C: \Re(s)>\frac 1 2 \}^n$ where the convergence  is uniform  on $\M^n$ and  that for some constant $B>0$ we have that
\begin{gather} 
 \label{tt1}        
        \max_{j=1,\ldots,n} \sup_{m \geq 0} \max_{s \in \M}  \abs{\sum_{k=1}^m      \frac{a_j(k+\alpha_j)}{   (k+\alpha_j)^{s}}} \leq B, \\ 			\intertext{and}                                                            \label{tt2}    \max_{1 \leq l \leq n} \sup_{m \geq 0} \max_{\vs \in \M^l}  \abs{\sum_{ 0 \leq k_1 < \cdots < k_{l}<m} \prod_{j=1}^l \frac{a_j(k_j+\alpha_j)}{ (k_j+\alpha_j)^{s_j}}} \leq B.
\end{gather}
By the Oka-Weil theorem we can approximate the function $p-g$ by a polynomial $q$ so that
\begin{gather}\label{est2} 
  \max_{\vs \in \M^n} \abs{q(\vs)-(p(\vs)-g(\vs))} < \frac{\varepsilon} 3.
\end{gather}
We now apply Lemma \ref{LE13} on the polynomial $q$ with  $K=\M^n$ and \begin{gather} \label{Ckond} C \geq B+2A_0+1, \end{gather}  ensuring that the functions we need to estimate in a later step have sufficiently large minimum value to allow us to use Lemma \ref{LE3}.  This means that we have monomials   $q_{k_j,j}(s_j)$ and an integer $M$ so that
  \begin{gather} \label{psum}
     q(\vs)= \sum_{1 \leq  m_1 < \cdots < m_n \leq M} \prod_{j=1}^n q_{j,m_j} (s_j), \\ \intertext{and}
		 \label{psum2}  \sum_{1 \leq  m_v  < \cdots < m_n \leq M} \prod_{j=v}^n q_{j,m_j} (s_j)=0, \qquad (2 \leq v \leq n).
 \end{gather}
Define 
\begin{gather} 
  \label{Qdef} 
     \Delta=\max_{1 \leq m \leq M} \max_{1 \leq j \leq n}\max_{s \in \M}\abs{q_{j,m}(s)}.
  \end{gather}
where $N_0$ is the constant that comes from the application of Lemma \ref{LE10} with 
\begin{gather} \label{epstar}
 \varepsilon^*=\min \p{1,\frac{\varepsilon}{3n(1+Bn)(2(M+1)(\Delta+1))^n}}.
\end{gather}
By \eqref{ggdef} and Lemma \ref{EEEV}  there exist some  sufficiently large $\NC_0>\max(L_0,N_0)$ such that
\begin{gather} \label{est1l} 
   \max_{\vs \in \M^n} \abs{\sum_{ 
 0 \leq k_1 <\cdots < k_n < \NC_0} \prod_{j=1}^n \frac{a_j(k_j+\alpha_j)}{ (k_j+\alpha_j)^{s_j}}-g(\vs)} < \frac {\varepsilon} 3, \\ \intertext{and} 
 \label{est2l}
    \max_{1 \leq j \leq n} \max_{s \in \M} \abs{\sum_{k=\NC_0}^\infty \frac{a_j(k_j+\alpha_j)}{ (k_j+\alpha_j)^{s_j}}}<1. 
\end{gather}
 By the definition of  $\zeta_{\va}^{[N]}(\vs;\val)$  as a multiple sum; equation \eqref{flv}; it can be rewritten as 
	\begin{gather} \label{flv3c}
    \zeta_{\va}^{[N]}(\vs;\val) =  \sum_{v=1}^n A_v(\vs) B_v(\vs), \\ \intertext{where}  A_1(\vs)=1, \qquad  A_v(\vs)=   \sum_{0 \leq k_1 < \cdots < k_{v-1} <\NC_0} \prod_{j=1}^{v-1} \frac{a_j(k_j+\alpha_j)}{ (k_j+\alpha_j)^{s_j}},   \qquad  (2 \leq v \leq n), \qquad \label{Avdef} \\ \intertext{and}
      B_v(\vs)= \sum_{\NC_0 \leq k_v < \cdots < k_n < N } \prod_{j=v}^{n} \frac{a_j(k_j+\alpha_j)}{ (k_j+\alpha_j)^{s_j}}. \label{Bvdef}
\end{gather}
Let us furthermore define
\begin{gather}
  \label{QQdef} Q_{j,m}(s)= \sum_{0 \leq k \leq \NC_0} a_j(k+\alpha_j)(k+\alpha_j)^{-s}+\sum_{l=1}^m q_{j,l}(s)
\end{gather}
By the triangle inequality, the inequalities \eqref{tt2}, \eqref{estaa}, \eqref{est2l} and \eqref{finalcond1} it follows that
\begin{gather} \label{ytre}
 \min_{s \in \M} \abs{Q_{j,m}(s)} \geq  A_0. 
\end{gather}
The same inequality also follows from \eqref{Ckond}, \eqref{tt2}, \eqref{estaa}, \eqref{est2l} and \eqref{finalcond2}. Since either \eqref{finalcond1} or \eqref{finalcond2} holds we know that \eqref{ytre} holds for each $1 \leq j \leq n$ and $1 \leq m \leq M$.
By \eqref{QQdef} it is also clear that \begin{gather} \label{ajabaja}
   q_{j,m}(s)=Q_{j,m}(s) - Q_{j,m-1}(s).
\end{gather}
\item[Recursion step:]
We now define $\NC_m>\NC_{m-1}$ for $1 \leq m \leq M$ recursively; at each step redefining the  $a_{j_m}(k+\alpha_{j_m})$ for $k \geq \NC_{m-1}$    by use of Lemma \ref{LE3} on the functions $Q_{j_m,m}(s)$ which by \eqref{ytre} fulfills the necessary conditions and letting $\NC_m$ be the $N_1$ that appears in the application of   Lemma \ref{LE3} so that we have
 \begin{gather} \label{aer}
     \max_{s  \in \M} \abs{\sum_{0 \leq k<\NC_{m}} \frac{a_j(k+\alpha_{j_m})}
{(k+\alpha_{j_m})^{s}}-Q_{j_m,m}(s)} <\varepsilon^*
\\  \intertext{and}
		 \sup_{\NC_{m} \leq N_2<N_3}  \max_{s  \in \M} \abs{\sum_{n=N_2}^{N_3} \frac{a_{j_m}(n+\alpha_{j_m})}{ (n+\alpha_{j_m})^{s}}} <\varepsilon^*, 
\label{aer2}
  \end{gather}
	where $\varepsilon^*$ is given by \eqref{epstar}. After the final redefinition of the $a_j$ to $\N+\alpha_j$ the inequalities \eqref{aer} and \eqref{aer2} imply together with the triangle inequality and \eqref{ajabaja} that 
	\begin{gather} \label{iii}
        \max_{s  \in \M} \abs{\sum_{\NC_{m-1} \leq  k<\NC_{m}} \frac{a_j(k+\alpha_j)}{(k+\alpha_j)^{s}}- q_{j,m}(s)} <2\varepsilon^*,
		\end{gather}
	for  $1 \leq m \leq M$ and $1 \leq j \leq n$. 
\item[Conclusion step:]
We now choose $N^*=\NC_M$. It is clear that each function $a_j:\N+\alpha_j \to \C$ after its final redefinition can be viewed as a function of type $(N^*,\chi_j)$.  For notational convenience we let $\NC_{M+1}=N>\NC_M$. The terms $B_v(\vs)$ defined by \eqref{Bvdef} can be written as
	\begin{gather} \label{flv3}
    B_v(\vs)= \sum_{1 \leq m_v \leq \cdots \leq m_n \leq M+1} \, \sum_{\substack{ k_v < \cdots < k_n \\ \NC_{m_j-1} \leq k_j < \NC_{m_{j}}}} \prod_{j=v}^n \frac{a_j(k_j+\alpha_j)}{ (k_j+\alpha_j)^{s_j}}
\end{gather}
It follows by \eqref{iii}  and \eqref{Qdef} that for $1 \leq v \leq n$ that 
\begin{gather} \label{ppp} \abs{\sum_{\substack{0 \leq k_v < \cdots < k_n \\ \NC_{m_j-1} \leq k_j < \NC_{m_{j}}}} \prod_{j=v}^n \frac{a_j(k_j+\alpha_j)}{ (k_j+\alpha_j)^{s_j}} - \prod_{j=v}^n q_{j,m}(s_j)}< (n-v+1) \varepsilon^* (2(\Delta+\varepsilon^*))^{n-v}, \qquad \end{gather}
in case we are summing from different intervals, i.e. $m_v< \cdots <m_n$. In case we allow $m_j=m_{j+1}$ for some $v \leq j \leq n-1$ then the same result follows from Lemma \ref{LE10} and the fact that at least one of $q_{j,m_j}(s_1)$ and $q_{j+1,m_j}(s_2)$ is identically zero so the product is also zero. From \eqref{flv3} and \eqref{ppp} and using the fact that we have chosen $\varepsilon^* \leq 1$ in \eqref{epstar} we see that
\begin{gather} \label{trw}
  \abs{B_v(\vs)-  \sum_{1 \leq m_v <\cdots <m_n \leq M} \prod_{j=v}^n q_{j,m}(s_j)}< n \p{2(M+1)(\Delta+1)}^{n}  \varepsilon^*.
\end{gather}  
By \eqref{psum} and \eqref{trw} we get the inequality
\begin{gather}
  \max_{\vs \in \M^n} \abs{B_1(\vs)- q(\vs)} <  n \p{2(M+1)(\Delta+1)}^{n} \,\varepsilon^*, 
\\ \intertext{and by  \eqref{psum2} and \eqref{trw}  we have}
 \max_{\vs \in \M^n} \abs{B_v(\vs)} < n \p{2(M+1)(\Delta+1)}^{n} \,\varepsilon^*, \qquad (2 \leq v \leq n).
\end{gather}
By  \eqref{Avdef} and \eqref{tt1} we find that $|A_v(\vs)| \leq B$ when $\vs \in \M^n$.  It follows that  
\begin{gather} \label{est222}
  \max_{\vs \in \M^n} \abs{ \zeta_{\va}^{[N]}(\vs;\val) - p(\vs)} < 
\frac {\varepsilon} 3
\end{gather}
since by \eqref{flv3c} and \eqref{Avdef} and the triangle inequality the left hand side in of \eqref{est222} can be bounded by
\begin{gather}
  \p{n\p{2(M+1)(\Delta+1)}^{n} +  \sum_{v=2}^n B n \p{2(M+1)(\Delta+1)}^{n}} \varepsilon^* 
\end{gather}
which by the definition of $\varepsilon^*$ in \eqref{epstar} is bounded by $\varepsilon/3$.
Finally we notice that the estimates \eqref{est2}, \eqref{est1l}, 
\eqref{est222} together with the triangle inequality implies \eqref{fundlemsuff} which concludes the proof of Lemma \ref{LE4}. 
\end{description}
\qed

\section{Proof of Lemma \ref{LE6} \label{secstar}}

Let us assume that $\va=(a_1,\ldots,a_n)$ and that $a_j$ is a completely multiplicative unimodular function on $\N+\alpha_j$. If $\alpha_j$  is rational then $a_j$ is determined on its values on the primes. Assume that
$$
  a_j(p)=e^{2 \pi i \theta_{j,p}}
$$
 if $\alpha_j$ is rational and  that
$$
 a_j(k+\alpha_j)=e^{2 \pi i \theta_{j,k}}
$$
if $\alpha_j$ is transcendental where $\theta_{j,p}, \theta_{j,k} \in [0,1)$.
Let us assume that $0 <\dddd \leq 1$ and let $\norm{x}$ denote the distance between $x$ and its nearest integer and define
\begin{gather} \label{BCdef}
   \newsym{A certain subset of $[0,T]$}{\BC_{j,N,T}( \dddd)}= \bigcap_{k=0}^{\lfloor N-\alpha_j \rfloor} \left \{0 \leq t \leq T: 
\norm { \frac{t} { 2 \pi \log(k+\alpha_j)} -\theta_{j,k} }<\frac { \dddd} 2  \right \}, 
\end{gather}
for $1 \leq j \leq n$ if $\alpha_j$ is transcendental and
\begin{gather} \label{BCdef2}
  \BC_{j,N,T}( \dddd)= \bigcap_{p \in {\mathcal P}_j(N)} \left \{0 \leq t_j \leq T: 
\norm { \frac{t} { 2 \pi \log(p)} -\theta_{j,p} }<\frac { \dddd} 2  \right \}, 
\end{gather}
if $\alpha_j$ is rational,
where we choose \begin{gather} \newsym{a finite set of primes containing all primes less than $N$ }{{\mathcal P}_j(N)} = \{p \leq N :  p \text{ prime} \}, \\ \intertext{if $1 \leq j \leq n-1$, and}  {\mathcal P}_n(N) = \mathcal P, \\ \intertext{where $\mathcal P$ is a finite set of primes such that}   \label{Pdef}  \{p \leq N : p \text{ prime} \} \subseteq {\mathcal P}. \end{gather}
 Furthermore, if $\alpha_j$ is rational we  let
\begin{gather}
  \mathcal P_j'(N) = \{p :  p \text{ prime}\} \setminus  P_j(N),%
\end{gather} 
and let \newsym{If $\alpha_j$ is rational, numbers with only prime factors in $\Pri_j(N)$. If $\alpha_j$ is transcendental $\{n+ \alpha_j :  0 \leq n \leq N \}$}{$\cA_j(N)$} and \newsym{If $\alpha_j$ is rational, numbers without prime factors in $\Pri_j(N)$. If $\alpha_j$ is transcendental $\{1 \}$} {$\cB_j(N)$} denote the multiplicative semigroups generated by $\Pri_j(N)$ and $\Pri_j'(N)$ respectively, so that $\cA_j(N)$ consists of positive integers with all prime factors in the set $\Pri_j(N)$ and $\cB_j(N)$ consists of positive integers with all prime factors in the set $\Pri_j'(N)$.
If $\alpha_j$ is transcendental we will let
\begin{gather}
  \cA_j(N)= \{n+ \alpha_j :  0 \leq n \leq N \}, \qquad \text{and} \qquad   \cB_j(N)=\{1\}.
\end{gather}
Furthermore we let $\cC_j(N)$ denote the multiplicative group generated by $\cA_j(N)$, so that $b \in \cC_j(N)$ if and only if
\begin{gather}
  b=\prod_{k=1}^m {a_k}^{b_k}
\end{gather}
for some $a_k \in \cA_j(N)$ and integers $b_k$.  It is a  consequence of a theorem of Weyl \cite{Weyl} (for some related results, see  for example the discussion in \cite[pp 55-58]{Steuding3}) that
\begin{gather} \label{u77}
\lim_{T \to \infty} \frac{ \meas(\BC_{j,N,T}( \dddd)
)} {T} = \begin{cases}  \dddd^{\lfloor N -\alpha_j \rfloor +1}, & \alpha_j \text{ transcendental}, \\  \dddd^{\pi(N)}, & \alpha_j \text{ rational, and } 1 \leq j \leq n-1, \\  \dddd^{|\Pri|}, & \alpha_j \text{ rational, and } j=n, \\ \end{cases}  \qquad
\end{gather}
when  $0 < \dddd \leq 1$  and   that 
\begin{gather} \label{eq55}  
  \lim_{\dddd \to 0} \lim_{T \to \infty} \frac 1 { \meas \BC_{j,N,T}( \dddd)}  \int_{ \BC_{j,N,T}( \dddd)} b^{it} dt= \begin{cases} 1, &  b  \in \cC_j(N), \\ 0, & \text{otherwise}. \end{cases}
\end{gather}
Let ${\bf N}=(N_1,\ldots,N_n)$ and define
\begin{gather} \label{BCD}
 \newsym{A certain subset of $[0,T]^n$}{\BC_{{\bf N}, \Pri, T}( \dddd)}=\prod_{j=1}^n  \BC_{j,N_j,T}( \dddd). 
 \end{gather}
It follows from \eqref{u77} and \eqref{BCD} that 
 \begin{gather} 
  \label{eq56}
   \lim_{T \to \infty} \frac{\meas \BC_{{\bf N},\Pri,T}( \dddd)} {T^n}  =   \dddd^{A}, \qquad (0 < \dddd \leq 1), \\ \intertext{where} \label{eq56b}
                     A=\sum_{\alpha_j \text{ rational}} |\mathcal P_j(N_j)|+\sum_{\alpha_j \text{ transcendental}}  (\lfloor N_j -\alpha_j \rfloor +1).
\end{gather}
Furthermore we define $d_j$ as follows.
\begin{gather}
   d_j= \begin{cases} \min \{d \in \Z^+ : d \alpha_j \in \Z^+\}, & \alpha_j \text{ rational}, \\ 1,  & \alpha_j \text{ transcendental,} \end{cases}
\end{gather}
so that if $\alpha_j>0$ is rational then $\alpha_j=\frac{c_j}{d_j}$ for $c_j$ and $d_j$ coprime. With these notations we will prove (for its proof see subsection \ref{lele11ref}) that
\begin{lem} \label{lele11}  Let $\BC_{{\bf N},\Pri,T}( \dddd)$ be defined by \eqref{BCD} and let $K \subset D^n$ be a compact set. Let $\val=(\alpha_1,\ldots,\alpha_n)$ where $\alpha_j>0$ be rational or transcendental numbers and let $N_j \geq 1$ be integers for $1 \leq j \leq n$. Then
\begin{gather}
\begin{split}
  \lim_{ \dddd \to 0} \lim_{T \to \infty} \frac 1 {\meas  \BC_{{\bf N},\Pri,T}( \dddd) }\int_{\BC_{{\bf N},\Pri,T}( \dddd)} &  \int_{K} \abs{\zeta_{\va}^{[N]}(\vs; \val)- \zeta_{{\bf 1}}^{[N]}(\vs+i \vt;\val)}^2 d\vs d \vt  \\  =&\int_K \abs{\zeta_\va^{[N]}(\vs;\val)-\zeta_{\va,N}^{\bf{1},\bf{N},\Pri}(\vs;\val)}^2 d\vs +\Delta, \end{split}
 \\ \intertext{where} \label{ret} 
  \Delta= \newsym{A quantity appearing in  Lemma \ref{lele11} and Lemma \ref{lele121}}{\Delta_N(K,\va, {\bf N},\Pri)}=  \sum_{\substack{b_1 \cdots b_n \geq 2 \\ b_j \in \cB_j(N_j)}} \, \int_{K} \abs{\zeta_{\va,N}^{\vb,\bf{N},\Pri} (\vs;\val)}^2 d \vs,  \\ \intertext{and}
  \label{ret2} \newsym{A quantity appearing in  Lemma \ref{lele11} and Lemma \ref{lele121}}{\zeta_{\va,N}^{\vb,\bf{N},\Pri} (\vs;\val)}=  \sum_{\substack{0\leq k_1 < \cdots <k_n \leq N \\  d_j(k_j+\alpha_j) b_j^{-1} \in \cA_j(N_j) }}  \prod_{j=1}^n \frac{a_j(k_j+\alpha_j)}{ (k_j+\alpha_j)^{s_j} }
  \end{gather}
  \end{lem}
We also need the following Lemma which we will prove in subsection \ref{lele121ref}
\begin{lem}
  \label{lele121} Let the conditions on $\val$, and $K$ of Theorem  \ref{TH1} be satisfied and let 
$a_j:\N+\alpha_j \to \C$  be  functions of type $(N_0,\chi_j)$, where the order of $\chi_n$ is divisible by 4 if $\alpha_n$ is rational. Then for any $\varepsilon>0$ there exists ${\bf N} \in \N^n$ with $N_j \geq N_0$ for $j=1,\ldots,n$ and a finite set of primes $\Pri$
 containing all primes less than $N_n$ such that  
\begin{gather} 
\limsup_{N \to \infty} |\Delta_N(K,\va, {\bf N},\Pri)|<\varepsilon, 
\end{gather} where 
$\Delta_N(K,\va, {\bf N},\Pri)$ is defined in Lemma \ref{lele11}, and
\begin{gather}
  \max_{\vs \in K}  \abs{\zeta_\va^{[N]}(\vs;\val)-\zeta_{\va,N}^{\bf{1},\bf{N},\Pri}(\vs;\val)} <\varepsilon.
\end{gather}

\end{lem}  
We are now ready to prove Lemma \ref{LE6}.
 
\noindent {\em Proof of Lemma \ref{LE6}.}
 Let \begin{gather} \label{Kxidef}
        K^\ddddd= \{s \in \C: d(s,K) \leq \ddddd \}
      \end{gather} 
      be a neighborhood of $K$ where $\ddddd>0$ is chosen sufficiently small such that $K^\ddddd \subset \M^n$ and that if $K$ satisfies \eqref{Kkond} then also  $K^\ddddd$ satisfies \eqref{Kkond}.   In a similar way as in the proof of Theorem \ref{TH1} it is sufficient to prove that
\begin{gather} \label{yrt4}
 \int_{K^\ddddd} \abs{\zeta_{\va}^{[N]}(\vs; \val)- \zeta_{\bm{1}}^{[N]}(\vs+i \vt;\val)}^2 d\vs < \varepsilon_2 
\end{gather}
  for any $\varepsilon_2>0$ in order for 
\begin{gather} \label{yrt3}
  \max_{s \in K} \abs{\zeta_{\va}^{[N]}(\vs;\val)-\zeta_{\bm{1}}^{[N]}(\vs+i\vt;\val)}< \varepsilon
  \end{gather}
to be true for any $\varepsilon>0$ since by complex analysis an estimate  in $L^2$-norm on a slightly larger set gives approximation in sup-norm on the smaller set.  It now follows from applying Lemma \ref{lele11} and Lemma \ref{lele121} on the set $K^\ddddd$ that for some ${\bf N} \in (\Z^+)^n$ then \begin{gather}
 \lim_{ \dddd \to 0} \lim_{\min_j T_j \to \infty} \frac 1 {\meas  \BC_{{\bf N},\Pri,{\bf T}}( \dddd) }\int_{\BC_{{\bf N},{\bf T}}( \dddd)}  \int_{K^\ddddd} \abs{\zeta_{\va}^{[N]}(\vs; \val)- \zeta_{\bf{1}}^{[N]}(\vs+i \vt;\val)}^2 d\vs d \vt <   \frac{\varepsilon_2} 2. \label{yrt5} 
\end{gather} 
Together with equation \eqref{eq56} and with  equation \eqref{yrt3} this implies \eqref{yrt4} with positive lower measure greater than $\frac 1 2  \dddd^{A},$ where $A$ is given in \eqref{eq56b}, and thus for  suitable small $\varepsilon_2$ depending on $\varepsilon, K , \ddddd$ also \eqref{yrt3} follows with a positive lower measure. \qed

\subsection{Proof of lemma \ref{lele11}} \label{lele11ref}

Before we prove Lemma \ref{lele11} we prove a useful Lemma.
\begin{lem}
  \label{lele11prelem}  We have that
\begin{gather} 
  \lim_{ \dddd \to 0} \lim_{T \to \infty} \frac 1 {\meas  \BC_{{\bf N},\Pri,T}( \dddd) } \int_{\BC_{{\bf N},\Pri, T}( \dddd)}  \abs{\zeta_{\bf{1}}^{[N]}(\vs+i \vt;\val)}^2 d\vt = \sum_{\substack {b_1 \cdots b_n \geq 2 \\ b_j \in \cB(N_j)}}  \abs{\zeta_{\va,N}^{\vb,\bf{N},\Pri} (\vs;\val)}^2, \qquad \label{aaaa} \\  
\intertext{and}
  \lim_{ \dddd \to 0} \lim_{T \to \infty} \frac 1 {\meas  \BC_{{\bf N},\Pri,T}( \dddd) } \int_{\BC_{{\bf N}, \Pri,T}( \dddd)} \zeta_{\bf{1}}^{[N]}(\vs+i \vt;\val) d\vt= \zeta_{\va,N}^{\bf{1},{\bf N},\Pri} (\vs;\val). 
\end{gather} 
where  ${\bf 1}=(1,\ldots,1) \in \N^n$.
\end{lem}

\begin{proof} 
We divide the proof into two cases. We first prove the first part. Consider
\begin{gather}
(*)=   \lim_{ \dddd \to 0} \lim_{T \to \infty} \frac 1 {\meas  \BC_{{\bf N},\Pri,T}( \dddd) }      \int_{\BC_{{\bf N},\Pri, T}( \dddd)}  \abs{\zeta_{\bf{1}}^{[N]}(\vs+i \vt;\val)}^2 dt.
\end{gather}
By writing the truncated zeta-functions as finite sums and integrating explicitly, we get
\begin{multline}
 (*)  =  \lim_{ \dddd \to 0} \lim_{T \to \infty}  \frac 1 {\meas  \BC_{{\bf N},\Pri,T}( \dddd) }   
\sum_{\substack{0\leq m_1 < \cdots <m_n \leq N  \\0\leq k_1 < \cdots <k_n \leq N }}  \int_{\BC_{{\bf N},\Pri, T}( \dddd)} \times \\ \times  \prod_{j=1}^n \frac{a_j(k_j+\alpha_j)\overline{a_j(m_j+\alpha_j)} }{ (k_j+\alpha_j)^{s_j} (m_j+\alpha_j)^{\overline{s_j}}}  \p{\frac{m_j+\alpha_j}{k_j+\alpha_j}}^{it_j}  d \vt,
\end{multline}
which by the definition \eqref{BCD} of $\BC_{{\bf N}, \Pri,T}( \dddd)$ as a product set can be written as
\begin{multline}
 (*)  =  
\sum_{\substack{0\leq m_1 < \cdots <m_n \leq N  \\0\leq k_1 < \cdots <k_n \leq N }}   \prod_{j=1}^n \left(\frac{a_j(k_j+\alpha_j)\overline{a_j(m_j+\alpha_j)} }{ (k_j+\alpha_j)^{s_j} (m_j+\alpha_j)^{\overline{s_j}}}  \times \right. \\ \left. \times   \lim_{ \dddd \to 0} \lim_{T \to \infty}    \frac 1 {\meas  \BC_{j,N_j,T}( \dddd) } \int_{\BC_{j,N_j,T}( \dddd)}   \p{\frac{m_j+\alpha_j}{k_j+\alpha_j}}^{it_j} dt_j \right).
\end{multline}
By \eqref{BCdef}, \eqref{BCdef2} and  \eqref{eq55}   it follows that 
\begin{multline}
 \lim_{ \dddd \to 0}  \lim_{T \to \infty} \frac 1 {\meas  \BC_{j,N_j,T}( \dddd) }  \int_{\BC_{j,N,T}( \dddd)}   \p{\frac{m_j+\alpha_j}{k_j+\alpha_j}}^{it_j} dt_j  \\ = \begin{cases} a_j(k_j+\alpha_j) \overline{a_j(m_j+\alpha_j)}, & \frac{m_j+\alpha_j}{k_j+\alpha_j}=\frac a b, \text{ for some } a,b \in \cA(N_j), \\ 0, & \text{otherwise,} \end{cases}
\end{multline}
or in other words, we get the expected contribution if and only if  $d_j (m_j+\alpha_j) b_j^{-1} =a$ and $d_j (k_j+\alpha_j) b_j^{-1} =b$ for some $a,b \in \cA_j(N)$.  Thus the sum $(*)$ can be written as required. We now proceed to prove the second part. Similarly as above we get by expanding the sum that
\begin{multline}
  \frac 1 {\meas  \BC_{{\bf N},\Pri, T}( \dddd) } \int_{\BC_{{\bf N},\Pri, T}( \dddd)} \zeta_{\bf{1}}^{[N]}(\vs+i \vt;\val) d\vt   \\ = 
\sum_{0\leq k_1 < \cdots <k_n \leq N }    \prod_{j=1}^n  \p{ (k_j+\alpha_j)^{-s_j}   \lim_{ \dddd \to 0} \lim_{T \to \infty}   \frac 1 {\meas  \BC_{j,N_j,T}( \dddd)} \int_{\BC_{j,N_j,T}( \dddd)} (k_j+\alpha_j)^{-it_j} dt_j}.
\end{multline}
When $d_j(k_j+\alpha_j)b_j^{-1} \in \cA(N_j)$ it follows by \eqref{BCdef}, \eqref{BCdef2} and  \eqref{eq55}  that  the last  limit in the equation above  will equal $a_j(k_j+\alpha_j)$,  whereas otherwise it will be zero.   This concludes the proof of the second part. 
\end{proof}

\noindent {\em Proof of Lemma \ref{lele11}.}
We have that \begin{gather*}      
  I(\vs) =  \lim_{ \dddd \to 0}  \lim_{T \to \infty} \frac 1 {\meas  \BC_{{\bf N},\Pri,T}( \dddd) } 
  \int_{\BC_{{\bf N}, \Pri, T}( \dddd)}    \abs{\zeta_{\va}^{[N]}(\vs; \val)- \zeta_{\bf{1}}^{[N]}(\vs+i \vt;\val)}^2 d\vt  
\end{gather*}
 After  rewriting the integrand in the standard way as $ {\left| * \right|}^2 = * \overline{*}$ and expanding the integral
we obtain
\begin{align}
  I(\vs) &=  \lim_{ \dddd \to 0} 
 \lim_{T \to \infty} \frac 1 {\meas  \BC_{{\bf N},\Pri,T}( \dddd) } \times \\ &\times \left( \int_{\BC_{{\bf N}, \Pri,T}( \dddd)}    \abs{ \zeta_{\bf{1}}^{[N]}(\vs+i \vt;\val)}^2 d\vt +\meas  \BC_{{\bf N},\Pri,T}( \dddd) \abs{\zeta_{\va}^{[N]}(\vs; \val)}^2 \right. \\  &\left.  - \overline{\zeta_{\va}^{[N]}(\vs ; \val)}  \int_{\BC_{{\bf N},\Pri, T}( \dddd)} \zeta_{\bf{1}}^{[N]}(\vs+i \vt;\val) d\vt - \zeta_{\va}^{[N]}(\vs ; \val)  \int_{\BC_{{\bf N},\Pri, T}( \dddd)} \overline{\zeta_{\bf{1}}^{[N]}(\vs+i \vt;\val)} d\vt \right).
 \end{align} 
By moving the limits inside the parenthesis and applying Lemma \ref{lele11prelem} this can be simplified to
\begin{multline}
 I(\vs)= \p{\sum_{b_j \in \cB_j(N_j)}  \abs{\zeta_{\va,N}^{\vb,\bf{N},\Pri} (\vs;\val)}^2} + \abs{ \zeta_{\va}^{[N]}(\vs;\val)}^2 \\
- \overline{\zeta_{\va}^{[N]}(\vs ; \val)} \zeta_{\va,N}^{{\bf 1},\bf{N},\Pri}(\vs;\val) - \zeta_{\va}^{[N]}(\vs ; \val) \overline{\zeta_{\va,N}^{{\bf 1},\bf{N},\Pri}(\vs;\val)}. 
\end{multline}
This equality may be rewritten as
\begin{gather} \label{or77}
 I(\vs)= \p{\sum_{\substack{ b_1 \cdots b_n \geq 2 \\  b_j \in \cB_j(N_j)}}  \abs{\zeta_{\va,N}^{\vb,\bf{N},\Pri} (\vs;\val)}^2}  + \abs{\zeta_{\va}^{[N]}(\vs; \val)- \zeta_{\va,N}^{{\bf 1},\bf{N},\Pri}(\vs;\val)}^2.
\end{gather}
The lemma follows by taking the integral over the set $K$.
\qed

\subsection{Proof of Lemma \ref{lele121}} \label{lele121ref}

Before proving Lemma \ref{lele121}  we need some useful lemmas. The following lemma will be needed when we do not assume any Riemann hypothesis.
Since its proof is the same we choose to write the next joint approximation lemma in a slightly more general form than we need since it might have some independent interest.
\begin{lem} \label{BNDDIR}
  Let $\chi$ mod $q$ be a Dirichlet character of  order divisible by $4$ and let  $q$ and $d$ be coprime.  Then given $M \geq 0$, $\varepsilon>0$ and analytic zero-free functions $f_{\chi^*}$ on $\M$ where $\chi^*$ are the Dirichlet characters mod $d$ there exist some finite set of primes $\mathcal P$ such that $\mathcal P$ contains all primes less than $M$ and the finite Euler products
\begin{gather}
   L_{\mathcal P}(s,\chi \chi^*)= \prod_{p \in {\mathcal P}} (1-\chi(p) \chi^*(p) p^{-s})^{-1},
\end{gather}
 fulfill the inequality
\begin{gather}
  \max_{\chi^* \mod d} \max_{s \in \M}  
\abs{L_{\mathcal P}(s,\chi \chi^*)-f_{\chi^*}(s)} <\varepsilon.
\end{gather}
\end{lem}
\begin{proof} 
  Since
  \begin{gather}
     \log L_{\mathcal P}(s,\chi \chi^*) =   - \sum_{\substack{(j,d)=1 \\ 0< j <d}} \chi^*(j)  \sum_{\substack{p \equiv j \pmod d \\ p \in \Pri}}   \log(1-\chi(p) p^{-s}), \\ \intertext{and}
     \log  f_{\chi^*}(s) = \sum_{\substack{(j,d)=1 \\ 0< j <d}}  \chi^*(j) F_j(s), 
 \intertext{where} 
    F_j(s)=\frac 1 {\phi(d)} \sum_{\chi^* \mod d} \overline{\chi^* (j)} \log f_{\chi^*}(s),
 \end{gather}
  it is sufficient to prove that for any $\varepsilon_1>0$ 
we can find some finite set of primes $\Pri$ containing all primes less than $M$ such that
  \begin{gather} \label{M10} 
    \max_{s \in \M} \max_{(j,d)=1} \abs{\sum_{\substack{p \equiv j \pmod d \\ p \in \Pri}} \log(1-\chi(p) p^{-s}) + F_j(s)}<\varepsilon_1. 
   \end{gather}
Let us now assume that $M_1 \geq M$ is sufficiently large so that
\begin{gather}
   \label{M11}\max_{s \in \M} \sum_{p>M_1} \abs{\log(1-\chi(p) p^{-s}) + \chi(p) p^{-s} }<\frac{\varepsilon_1} {6}, \\ \intertext{and that}
    \label{M12} \max_{s \in \M} \sum_{p>M_1}   \abs{ (p-2\sqrt {p})^{-s} -  p^{-s} }<\frac{\varepsilon_1} {6}.  
\end{gather}
By Mergelyan's theorem we can find polynomials $Q_j(s)$ such that
\begin{gather} \label{RT11}
  \max_{s \in \M } \abs{Q_j(s) -F_j(s)+\sum_{\substack{p \equiv j \pmod d \\ p \leq M_1}} \log(1-\chi(p) p^{-s})}<\frac{\varepsilon_1}{6}.
\end{gather}
We let
\begin{gather} \label{Q12}
  Q_j(s)=P_{j,0}(s)+i P_{j,1}(s)
\end{gather}
where the polynomials $P_{j,l}(s)$ have real coefficients.  By the construction of Lemma \ref{subprime} we find that there exist a subset  $\mathcal P^*$ of the primes
 such that if $p_1^{j,l},p_2^{j,l},p_3^{j,l}, \ldots$ is an enumeration in increasing order of the primes greater than $M_1$ that are congruent to $j$ mod $d$ in $\Pri^*$  and such that $\chi(p_n^{j,l})=i^l$ for $l \in \{0,1,2,3\}$ where for convenience we denote $p_n=p_n^{1,0}$ such that each sequence contains approximately the same numbers of primes
\begin{gather} \label{uir}
  \abs{p_n^{j,l} -p_n} \leq 2 \sqrt{p_n}, \qquad (j,d)=1,
\end{gather} 
and furthermore $\Pri^*$  has full density amongst the primes, which in particular implies that
 \begin{gather}  \label{denseprime}
  p_n \sim  \phi(qd) n \log n. 
 \end{gather} 
Since the polynomials $P_{j,l}(s)$ defined in \eqref{Q12} have real coefficients it follows from the Pechersky rearrangement theorem for real Hilbert spaces due  to
 Mishou-Nagoshi\footnote{ Similarly as in Lemma \ref{rational} it is sufficient to have a subset of the primes fulfilling \eqref{denseprime}  
in order to prove the result of Mishou-Nagoshi.   It should be noted that Mishou-Nagoshi's result has typically; as in their original papers \cite{MisNag}, \cite{MiNag2}; found applications
 for proving universality for families of zeta- and  $L$-functions which have real coefficients, in the family aspect, see for example \cite{AndSod} and  \cite{ChoKim}. 
 } \cite[Proposition 2.5]{MisNag}    that whenever $\varepsilon_1>0$
then there exists some real numbers $c_{j,l}(n) \in \{-1,1\}$ and $N_1>0$  such that
\begin{gather} \label{ytree}
  \max_{\substack{(j,d)=1  \\ 0 < j < d \\  0 \leq l \leq 1}}   
 \max_{s \in \M} \abs{ \sum_{n=1}^{N_1}  c_{j,l}(n) p_n^{-s}-P_{j,l}(s)} <\frac{\varepsilon_1} {6}. 
\end{gather}
By the triangle inequality and  the inequalities \eqref{M12} and \eqref{uir} we have that
\begin{gather} \label{ytreee} 
 \max_{\substack{(j,d)=1  \\ 0 < j < d \\  0 \leq l \leq 1}}  \max_{s \in \M} \abs{ \sum_{n=1}^{N_1} c_{j,l}(n) \left(p_{n}^{j,l+1-c_{j,l}(n)}\right)^{-s}  - \sum_{n=1}^{N_1}  c_{j,l}(n) p_n^{-s}} <\frac{\varepsilon_1} {6}.
\end{gather}
 It follows by  \eqref{Q12}, \eqref{ytree}, \eqref{ytreee}   and the triangle inequality that
\begin{gather} \label{ytreeee}
  \max_{\substack{(j,d)=1 \\ 0 < j < d}}  \max_{s \in \M} \abs{ \sum_{n=1}^{N_1} \sum_{l=0}^1 i^l c_{j,l}(n) \left(p_{n}^{j,l+1-c_{j,l}(n)}\right)^{-s}  -Q_{j}(s)} <\frac{2\varepsilon_1} {3}.
\end{gather}
 We now choose the set $\Pri$ as follows
\begin{gather}
  \Pri=\{p \leq M_1: p \text{ prime}\} 
 \bigcup \bigcup_{(j,d)=1} \bigcup_{l=0}^1 \bigcup_{n=1}^{N_1} \left \{p_n^{j,l+1-c_{j,l}(n)} \right \}.
\end{gather}
With this choice of $\Pri$ it follows from  \eqref{M11},\eqref{ytreeee} and the triangle inequality that
\begin{gather} \label{ytr1e}
  \max_{\substack{(j,d)=1 \\ 0< j < d}}  \max_{s \in \M} \abs{ \sum_{\substack{p \equiv j \pmod d \\  p \in \Pri \\ p > M_1}} \log(1-\chi(p) p^{-s})^{-1} - Q_j(s)}<\frac {5\varepsilon_1} 6. 
\end{gather}
The inequality \eqref{M10} now follows from  \eqref{RT11}, \eqref{ytr1e} and the triangle inequality.
\end{proof}

\begin{lem} \label{gy} 
  Let $a:\N +\alpha \to \C $ be of type $(N_0,\chi)$, where the order of $\chi$  is  divisible by 4, and assume that $\alpha=\frac c d$ with $c,d \in \Z^+$  is a rational number. Assume that $\M$ is given by \eqref{Kdef}. Then there exists  some $C>0$  such that for any $M>0$ there exist a finite  set of primes $\Pri$ such that
if $\mathcal A$ denote the set of positive integers with all prime factors in $\Pri$ and $\cB$ denote the set of positive integers without prime factors in $\Pri$, then
\begin{gather} \label{OPP9}
      \abs{\sum_{\substack{d(k+\alpha)b^{-1}\in {\mathcal A} \\ k \geq 0}} \frac{a(k+\alpha)}{ (k+\alpha)^{s}}   } <
 C b^{-\Re(s)}, \qquad (b \in \cB, s \in \M) \\ \intertext{and} \label{OPP10}
   \abs{\sum_{\substack{d(k+\alpha)\in {\mathcal A} \\ k \geq M}} \frac{a(k+\alpha)}{ (k+\alpha)^{s}}   } < C M^{-1/2}. 
\end{gather}
\end{lem}
\begin{proof}
As in the proof of Lemma \ref{LE3} we will assume that
\begin{gather}
  \alpha=r+\{\alpha\}=r+ \frac c d,   \qquad (\GCD(c,d)=1, \, r \in {\mathbb N}, \,  0 \leq c <d).
\end{gather}
By Lemma \ref{trew} it follows  that the sum
\begin{gather}
  \zeta_a^{[N]}(s;\{\alpha\})= \sum_{k=-r }^N \frac{a(k+\alpha)}{(k+\alpha)^s}
\end{gather}
is uniformly convergent to an analytic function $f(s)$ on $\M$ as $N \to \infty$ and that
\begin{gather} \label{sssa}
  \max_{s \in \M}  \abs{    \sum_{k=-r}^N \frac{a(k+\alpha)}{(k+\alpha)^s} -f(s) }  \ll_{\varepsilon}  N^{-\varepsilon-\ddd-1/2}.
\end{gather}
By the same basic proof method as in the proof of Lemma \ref{LE3} where we prove the identity \eqref{fer} we have the corresponding identity 
\begin{align} \label{ferro299}
 \sum_{\substack{k \geq  - r\\  d(k+\alpha)  \in \cA}}  a(k+& \alpha) (k  +\alpha)^{-s} \\ 
 &= \frac {d^s \overline{a(d)}} {\phi(d)} \sum_{\chi^* \pmod d}  \overline{\chi^*(c)} L_{\Pri}(s,\chi\chi^*) \prod_{p< N_0 } \p{\frac{1- \chi^*(p) \chi(p) p^{-s}   }{1-\chi^*(p) a(p) p^{-s}} }.  
\end{align}
We may now let $\chi_0$ denote the principal character mod $d$ and apply Lemma \ref{BNDDIR} to find $\Pri$ so that
 \begin{gather} \label{fer11}
  \max_{s \in M}  \abs{L_{\Pri}(s,\chi\chi_0) -   \frac {\phi(d)}  {d^s \overline{a(d)}} \prod_{p< N_0 } \p{\frac{1- a(p) \chi(p) p^{-s}   }{1-\chi_0(p) \chi(p) p^{-s}} }  \p{f(s)- \sum_{k=-r}^{-1} \frac{a(k+\alpha)}{(k+\alpha)^s} } } < \varepsilon, \qquad  \\ \intertext{and}
    \max_{\substack{\chi^* \mod d  \\ \chi^* \neq \chi_0}} \max_{s \in M} \abs{L_{\Pri}(s,\chi \chi^*)}<\varepsilon, \qquad \label{fer11a}
\\ \intertext{where}
 \label{fer12}
  \varepsilon= M^{-1/2} \min_{s \in \M} \abs{d^{-s} \prod_{p< N_0 }    \p{\frac {1-\chi^*(p) a(p) p^{-s}} {1- \chi^*(p) \chi(p) p^{-s}  }}}. 
\end{gather} 
Equation \eqref{OPP10} now follows from the triangle inequality and \eqref{sssa}, \eqref{ferro299}, \eqref{fer11}, \eqref{fer11a} and \eqref{fer12}. We now proceed to prove that this choice of $\Pri$ also gives \eqref{OPP9}. We have
\begin{gather}  \label{ferro9bb}
 \sum_{\substack{d(n+\alpha)b^{-1} \in {\mathcal A} \\ n \geq 0}}  \frac{a(n+\alpha)}{(n+\alpha)^s} = b^{-s} \sum_{\substack{d(k+\beta) \in {\mathcal A} \\ k \geq 0}}  \frac{a(k+\beta)} {(k+\beta)^s},
\end{gather}
where $\beta=\frac c d$ for some $(c,d)=1$.  
In the same way as we obtained \eqref{ferro299} we have the corresponding identity 
\begin{align} \label{ferrobb}
  \sum_{\substack{d(k+\beta) \in {\mathcal A} \\ k \geq -r_2}}  \frac{a(k+\beta)}{(k  +\beta)^{s}} 
 = \frac {d^s \overline{a(d)}} {\phi(d)} \sum_{\chi^* \pmod d}  \overline{\chi^*(c)} L_{\Pri}(s,\chi\chi^*) \prod_{p< N_0 } \p{\frac{1- \chi^*(p) \chi(p) p^{-s}   }{1-\chi^*(p) a(p) p^{-s}} }  \qquad
\end{align}
for some $0 \leq r_2 \leq r$. 
Equation \eqref{OPP9} now follows by \eqref{ferro9bb}, \eqref{ferrobb}, \eqref{fer11a}, \eqref{fer11}, \eqref{fer12} and the triangle inequality.
\end{proof}

The following Lemma follows from sparsity of smooth numbers
\begin{lem} \label{gy2}   Assume that  $\alpha$                        is a rational number with denominator $d$.
 Suppose $\Pri$ is a finite set of primes with largest prime $P$ and that $\mathcal A$ denote the set of integers with all prime factors in $\Pri$ and $\cB$ denote the set of integers without prime factors in $\Pri$. Then 
\begin{gather} 
    \sum_{\substack{d(k+\alpha)b^{-1}\in {\mathcal A} \\ k \geq M}} \frac{1}{ \abs{k+\alpha}^{\Re(s)}}\ll_{\varepsilon} b^{\varepsilon- \Re(s)}, \qquad (M \geq \exp(P), b \in \cB, s \in \M)
\end{gather}
\end{lem}

\begin{proof}
  We have that 
\begin{align} 
(*)&=\sum_{\substack{d(k+\alpha)b^{-1}\in {\mathcal A} \\ k \geq M}} \frac{1}{ \abs{k+\alpha}^{s}}\\  &= d^{\Re(s)} b^{-\Re(s)} \sum_{\substack{b^{-1}(nd+c) \in \mathcal{A} \\ n \geq (M-c)/d}}  \abs{(nd+c) b^{-1}}^{-\Re(s)} \\ \intertext{which with $kb=nd+c$ gives us}  (*)&\leq d  b^{-\Re(s)} \sum_{\substack{k \in \cA \\ k > (M-c)/(bd)}} k^{-1/2-\ddd}, \qquad (s \in \M).  \label{rtyt}
\end{align}

 By the estimate \eqref{nwlpf} of the numbers of integers without small prime factors we find that if $b<\sqrt M$ then the sum in \eqref{rtyt} is bounded and thus $(*)  \ll  b^{-\Re(s)}$. If $b>\sqrt M$ we may estimate the  sum in \eqref{rtyt} by a sum over all integers in $\mathcal A$ 
\begin{align} 
(*)&\leq d  b^{-\Re(s)} \sum_{k \in \mathcal A} k^{-1/2-\ddd} 
=  d  b^{-\Re(s)} \prod_{p \in \Pri} (1-p^{-\Re(s)})^{-1}, \\
 &\leq d  b^{-\Re(s)} \prod_{p \leq P} (1-p^{-\Re(s)})^{-1} \ll d b^{-\Re(s)} e^{\sqrt P} \ll   b^{\varepsilon-\Re(s)},   \qquad \qquad (s \in \M), 
\end{align}
since $b^\varepsilon \ll e^{\sqrt P}$ if $b>e^{P/2}$.
\end{proof}

\begin{lem} \label{RHV} Let  $a: \N +\alpha \to \C$ be of type $(N_0,\chi)$ where $\alpha=\frac {c} {d}$ and suppose that $\varepsilon>0$. Then there exists some $C>0$ such that for any %given any $M>0$, 
positive integer $b \in \cB$ and finite set of primes $\Pri$ containing all primes less than $M$ with largest prime $P$ then
\begin{gather}
  \abs{\sum_{  \substack{d(n+\alpha)b^{-1} \in {\mathcal A} \\ N_1 \leq n \leq N_2}}  \frac{a(n+\alpha)} {(n+\alpha)^s}} \leq 
Cb^{-\Re(s)} \exp(\sqrt P). 
\end{gather}
holds for all $s \in \M$.
\end{lem}
\begin{proof}
The second case is Lemma \ref{gy2}. We now prove the final case which in fact holds for any finite set of primes $\Pri$ and for any $0\leq N_1 \leq N_2$. In this case we can estimate the sum from above by the triangle inequality as in the proof of Lemma \ref{gy2}  by
\begin{align}
  d^{\Re(s)} b^{-\Re(s)} \sum_{dn+l \in \mathcal A} (dn+l)^{-\Re(s)} &\leq d^{\Re(s)} b^{-\Re(s)} \prod_{p \leq P} (1-p^{-s})^{-1}, \\ &\ll b^{-\Re(s)} e^{\sqrt P}, \qquad (s \in \M). 
\end{align}
\end{proof}

\begin{lem} \label{RHV2} Suppose that $a: \N +\alpha \to \C$ is of type $(N_0,\chi)$,  where $\alpha=\frac {c} {d}$.  Furthermore assume that the Riemann hypothesis holds for $L(s,\chi \chi^*)$ for all Dirichlet characters $\chi^*$ mod $d$. Then given $\varepsilon>0$ there exists some $C>0$ such that given any $M>0$ and positive integer $b$ there exists some $M_1 \geq M$ such that if $\mathcal A$ is the set of positive integers with all prime factors less than $M_1$, then 
\begin{gather} \max_{s \in \M} \abs{\sum_{\substack{d(n+\alpha)b^{-1} \in {\mathcal A} \\ N_1 \leq n \leq N_2}}  \frac{a(n+\alpha)} {(n+\alpha)^s}} \leq C  ((b+N_1)^{\varepsilon-\ddd-1/2}+b^{\varepsilon-\ddd-1/2} M_1^{\varepsilon-\ddd}), \qquad (b \in \cB). 
\end{gather} 
\end{lem}

\begin{proof}
Without loss of generality we may assume that $0<\alpha<1$ since the general result easily follows from this special case.  We may write 
\begin{gather} \label{uyr3} 
 \sum_{\substack{d(n+\alpha)b^{-1} \in {\mathcal A} \\ N_1 \leq n \leq N_2}}  \frac{a(n+\alpha)} {(n+\alpha)^s} = b^{-s} \sum_{\substack{d(n+\beta) \in {\mathcal A} \\ N_1/b \leq n \leq N_2/b}}  \frac{a(n+\beta)} {(n+\beta)^s},
\end{gather}
where $\beta= \frac l d$ for some $1 \leq  l \leq d$. It follows that we have the generating Dirichlet series (this is similar to \eqref{rraj3})
\begin{gather}  
  \sum_{\substack{d(n+\beta) \in {\mathcal A} \\  n \geq 0}}  \frac{a(n+\beta)} {(n+\beta)^s} = \frac {d^{s}}{\phi(d)}\sum_{\chi^* \mod d} \,  \sum_{P^+(k) \leq \log N_0} b(k)  \overline{\chi^*(lk)} k^{-s} L_\Pri(s,\chi \chi^*),
\end{gather}
where the coefficients $b(k)$ are defined by \eqref{rraj2}, and that
\begin{multline}
  \sum_{\substack{d(n+\beta) \in {\mathcal A} \\  \substack N_1/b < n <N_2/b}}  \frac{a(n+\beta)} {(n+\beta)^s} \\ = \frac {d^{s}}{\phi(d)}\sum_{\chi^* \mod d}  \, \sum_{P^+(k) \leq \log N_0} b(k) \overline{\chi^*(lk)} k^{-s} \sum_{ \substack{ N_1/(bk) <  m < N_2/(bk) \\ m \in \mathcal A}} \frac{\chi(m) \chi^*(m)} {m^{s}}. \label{ire}
\end{multline}
It is a consequence of the Riemann hypothesis for $L(s,\chi \chi^*)$ that the innermost sum can be estimated by
\begin{gather} \label{ire2}
 \max_{s \in \M} \abs{\sum_{\substack{  N_1/(bk) <  m < N_2/(bk) \\ m \in \mathcal A}}  \frac{\chi(m) \chi^*(m)} {m^{s}} } \ll_{\varepsilon} \p{\frac {N_1}{bk} + 1}^{\varepsilon-\ddd-1/2}+M_1^{\varepsilon-\ddd}.
\end{gather}
From \eqref{uyr3}, \eqref{ire} and \eqref{ire2} and the triangle inequality it follows that
\begin{align}
\max_{s \in \M} \abs{\sum_{\substack{d(n+\alpha)b^{-1} \in {\mathcal A} \\ N_1 \leq n \leq N_2}}  \frac{a(n+\alpha)} {(n+\alpha)^s}
}  &\ll d b^{-\ddd-1/2}  \sum_{P^+(k) \leq \log N_0 } k^{-\ddd-1/2}\p{\p{\frac {N_1}{bk}+1 }^{\varepsilon-\ddd-1/2}+ M_1^{\varepsilon-\ddd}} \\ &\ll   ((b+N_1)^{\varepsilon-\ddd-1/2} + \, M_1^{\varepsilon-\ddd}),
\end{align}
where the final inequality follows from the sparsity of smooth numbers, or in other words the inequality \eqref{nwlpf} for the number of integers without large prime factors.
\end{proof}
  
\begin{lem} \label{TTV} Suppose that $\alpha_j$ are rational or transcendental numbers and $a_j: \N + \alpha_j \to \C$ are of type $(N_0,\chi_j)$. Then for $M \geq N_0$, $\varepsilon>0$ and any finite set of primes $\Pri$ containing all primes less than $M$ that satisfies \eqref{OPP9} we have that 
\begin{gather} 
\limsup_{N \to \infty} \max_{\vs \in \M^n} \abs{\sum_{0 \leq k_1 < \cdots <k_{n-1} \leq N} \sum_{\substack{\substack{d_n(k_n+\alpha_n)b_n^{-1} \in {\mathcal A}_n(N_n)}\\ k_{n-1} < k_n \leq N}} \prod_{j=1}^n \frac{a_j(k_j+\alpha_j)}{ (k_j+\alpha_j)^{s_j}}} \ll_{\varepsilon} b_n^{\varepsilon-1/2-\ddd}. 
\end{gather}
\end{lem} 

\begin{proof}
  By Lemma \ref{EEEV} there exists some $C>0$ and functions $A,B_k:\M^{n} \to \C$ such that 
 \begin{gather} 
\label{Mtineq}
   \sum_{1 \leq k_1 < \cdots <k_{n-1}<k}\prod_{j=1}^n \frac{a_j(k_j+\alpha_j)}{ (k_j+\alpha_j)^{s_j}} =A(\vs)+ B_k(\vs), \\ \intertext{where}
\max_{\vs \in \M^{n}} \abs{A(\vs)} \leq C,   \label{Errorineq0} \\ \intertext{and} \label{Errorineq}
\sup_{k \geq 0} \max_{\vs \in \M^{n}}   \abs{B_k(\vs)} \sqrt k \leq C.
\end{gather} 
We may write 
\begin{multline} 
 I_N(b_n,\vs)=\sum_{1 \leq k_1 < \cdots <k_{n-1} \leq N}  \sum_{\substack{\substack{d_n(k_n+\alpha_n)b_n^{-1} \in {\mathcal A}_n}\\ k_{n-1} < k_n \leq N}} \prod_{j=1}^n \frac{a_j(k_j+\alpha_j)}{ (k_j+\alpha_j)^{s_j} } \\ =  \label{urt5}
A(\vs)\sum_{\substack{d_n(k+\alpha_n)b_n^{-1}\in {\mathcal A}_n \\ 0 \leq k \leq N}} \frac{a_n(k+\alpha_n)}{ (k+\alpha_n)^{s_n}} \, \,  + 
\sum_{\substack{d_n(k+\alpha_n)b_n^{-1}\in {\mathcal A}_n  \\ 0 \leq k \leq N}}B_k(\vs)  \frac{a_n(k+\alpha_n)}{ (k+\alpha_n)^{s_n}}.
 \end{multline}
By Lemma \ref{gy} there exist some $N^*$ such that 
\begin{gather} \label{OPP2}
 \max_{\vs \in \M} \abs{\sum_{\substack{d_n(k+\alpha_n)b_n^{-1}\in {\mathcal A}_n \\ 0 \leq k \leq N}} \frac{a_n(k+\alpha_n)}{ (k+\alpha_n)^{s_n}}} \leq 2Cb_n^{-1/2-\ddd},  \qquad (N \geq N^*).
\end{gather}
By using the triangle inequality on the second term in \eqref{urt5} and \eqref{Errorineq} we get that 
 \begin{gather} \label{yr123}
    \abs{\sum_{\substack{ d_n(k+\alpha_n)b_n^{-1}\in {\mathcal A}_n   \\ 0 \leq k \leq N}} B_{k}(\vs) \frac{a_n(k+\alpha_n)}{ (k+\alpha_n)^{s_n}}} \leq  C \sum_{\substack{d_n(k+\alpha_n)b_n^{-1}\in {\mathcal A}_n  \\ 0 \leq k }} k^{-1-\ddd} \ll b_n^{-1-\ddd}. 
\end{gather}  
The lemma follows by using the triangle inequality  and the inequalities \eqref{Mtineq}, \eqref{OPP2}, \eqref{yr123} on \eqref{urt5}.
\end{proof}

\begin{lem} \label{EEV2}
  Assume that $a_j:\N+\alpha_j \to \C$ is of type $(N_0,\chi_j)$ where $\alpha_j$ are rational and transcendental, $\alpha_n=\frac c d$ with $(c,d)=1$ is rational.  Furthermore assume that $A_j(N_j)$ denote the set of integers with all prime factors less than $N_j$ for $1 \leq j \leq n$  so in particular we assume that $\Pri$ is the set of prime less than $N_n$.
Assume  that $k \geq \exp(N_j)$ for $j=v+1,\ldots,n-1$,
  and that the Riemann hypothesis is true for $L(s,\chi \chi^*)$  if $\chi^*$ is a character mod $d$. Then for $0 \leq v \leq n-1$ the following inequality holds
\begin{multline}
    \max_{(s_{v+1},\ldots,s_n) \in \M^{n-v}} \abs{\sum_{\substack{k \leq k_{v+1} < \cdots <k_n \leq N \\  d_j(k_j+\alpha_j) {b_j}^{-1} \in \cA_j( N_j) }} \prod_{j=v+1}^n \frac{a_j(k_j+\alpha_j)}{ (k_j+\alpha_j)^{s_j}}} \\ \ll_{\varepsilon}  \p{(k+b_n)^{\varepsilon-\ddd-1/2}+b_n^{\varepsilon-1/2-\ddd} N_n^{\varepsilon-\ddd} } \prod_{j=v+1}^{n-1}  b_j^{\varepsilon-1/2-\ddd}.
\end{multline} 
\end{lem} 

\begin{proof}
By the triangle inequality we may estimate the sum
\begin{gather} \label{ara} \begin{split}
     (*)&= \abs{\sum_{\substack{k \leq k_{v+1} < \cdots <k_n \leq N \\ d_j(k_j+\alpha_j) {b_j}^{-1} \in \cA_j( N_j) }} \prod_{j=v+1}^{n} \frac{a_j(k_j+\alpha_j)}{ (k_j+\alpha_j)^{s_j}}} \\ &\leq \sum_{\substack{k \leq k_{v+1} < \cdots <k_{n-1}  \\  d_j(k_j+\alpha_j) b_j^{-1} \in \cA_j( N_j)  }}   \p{\prod_{j=v+1}^{n-1} \abs{k_j+\alpha_j}^{-\Re(s_j)}}   S_{k_{n-1}}(s_n), \end{split} \\ \intertext{where}
  S_{k_{n-1}}(s_n) = \abs{ \sum_{\substack{k_{n-1}<k_n \leq N  \\   d_n(k_n+\alpha_n) {b_n}^{-1} \in \cA_n( N_n)    }}  \frac{a_n(k_n+\alpha_n)} {(k_n+\alpha_n)^{s_n}}}. 
\end{gather}
It follows  by Lemma \ref{RHV2} together with the assumption $k \leq k_{n-1}$   that \begin{gather} \label{yrtr}  S_{k_{n-1}}(s_n) \ll  (k+b_n)^{\varepsilon-1/2-\ddd}+b_n^{\varepsilon-1/2-\ddd} N_n^{\varepsilon-\ddd}, \qquad (s_n \in \M). \end{gather} 
It follows  by \eqref{ara} and \eqref{yrtr} that  for $s_j  \in \M$ we have that
\begin{align}
  (*) &\ll_{\varepsilon}  ((k+b_n)^{\varepsilon-\ddd-1/2}+N_n^{\varepsilon-\ddd} )  \sum_{\substack{k \leq k_{v+1} < \cdots <k_{n-1} \leq N \\  d_j(k_j+\alpha_j) b_j^{-1} \in \cA_j( N_j)  }}  \prod_{j=v+1}^{n-1}\abs{k_j+\alpha_j}^{-\Re(s_j)},  \\ &\leq  ((k+b_n)^{\varepsilon-\ddd-1/2}+N_n^{\varepsilon-\ddd} )  \prod_{j=v+1}^{n-1}  \p{\sum_{\substack{k_j \geq k \\ d_j(k_j+\alpha_j) b_j^{-1} \in \cA_j( N_j)}}  \abs{k_j+\alpha_j}^{-\Re(s_j)}}. 
\end{align}
The conclusion follows by using Lemma \ref{gy2} on the factors in the final product. 
\end{proof}

\begin{lem} \label{EEV3}
  Assume that $a_j:\N+\alpha_j \to \C$ is of type $(N_0,\chi_j)$ where $\alpha_j$ are rational and transcendental positive numbers, and where $\alpha_n=\frac c d$ with $(c,d)=1$ is rational, and where 
 $k \geq \exp(N_j)$ for $j=1,\ldots,n-1$ and $k \geq \exp(P)$, where $P$ is the largest prime in the set $\Pri$.  Then for any $\ddddd>0$ we have that 
\begin{gather}
     \abs{\sum_{\substack{k \leq k_{v+1} < \cdots <k_n \leq N \\  d_j(k_j+\alpha_j) {b_j}^{-1} \in \cA_j( N_j) }} \prod_{j=v+1}^n \frac{a_j(k_j+\alpha_j)}{ (k_j+\alpha_j)^{s_j}}} \ll_{\ddddd} k^{n\ddddd-A} \prod_{j=v+1}^n  b_j^{-1/2-\ddddd},  \\ \intertext{where}
 A=\sum_{j=v+1}^n \p{\Re(s_j)-\frac 1 2}, 
\end{gather}
for  $(s_{v+1},\ldots,s_n) \in \M^{n-v}$.
\end{lem}
\begin{proof}
We will use the inequality
\begin{gather}
\frac{1}{\abs{k_j+\alpha_j}^{\Re(s_j)}} \leq k^{1/2-\Re(s_j)+\varepsilon+\ddddd} \frac{1}{\abs{k_j+\alpha_j}^{1/2+\ddddd+\varepsilon}} 
\end{gather}
valid for $\varepsilon,\ddddd>0$ and $k_j \geq k$ and the triangle inequality 
\begin{align} 
\abs{\sum_{\substack{k \leq k_{v+1} < \cdots <k_n \leq N \\ d_j(k_j+\alpha_j) {b_j}^{-1} \in \cA_j( N_j) }} \prod_{j=v+1}^{n-1} \frac{a_j(k_j+\alpha_j)}{ (k_j+\alpha_j)^{s_j}}}   &\leq \prod_{j=v+1}^n \p{\sum_{k \leq k_{j}} \frac{1}{ {\abs{k_j+\alpha_j}^{\Re(s_j)}}}} \\  
   &\leq k^{(n-v)(\ddddd+\varepsilon) -A}  \prod_{j=v+1}^n \p{\sum_{k \leq k_{j}} \frac{1}{ {\abs{k_j+\alpha_j}^{1/2+\ddddd+\varepsilon}}}}.
\end{align} 
The result now follows from  the choice $\varepsilon=v\ddddd/n$ and Lemma \ref{gy2}.
\end{proof}

\noindent {{\em Proof of Lemma  \ref{lele121}.}} 
We assume that $R$, $\ddd$ in \eqref{Kdef} are chosen so that $K \subset \M^n$. It is sufficient to show that there exist some $\ddddd>0$ and $C>0$  such that for some  sufficiently large $\LL \in \Z^+$ there exist some  ${\bf N}=(N_1,\ldots,N_n)$ with $N_j \geq \LL$ for $j=1,\ldots,n$ and a finite set of primes $\Pri$  containing all primes less than $N_n$ such that
\begin{gather}
 \lim_{N \to \infty}  \max_{\vs \in K}  \abs{\zeta_\va^{[N]}(\vs;\val)-\zeta_{\va,N}^{\bf{1},\bf{N},\Pri}(\vs;\val)} <\varepsilon.
 \label{visa}
\end{gather}
and such that
\begin{gather}
  \label{tr33} 
 \limsup_{N \to \infty}
\max_{\vs \in K}  \abs{\zeta_{\va,N}^{\vb,{\bf N},\Pri} (\vs;\val)} \leq C \prod_{j=1}^n b_j^{-1/2-\ddddd}, \qquad (b_j \in \cB_j(N_j)),
\end{gather}
since then 
\begin{align}
\Delta_N(K,\va, {\bf N},\Pri)&=  
   \sum_{\substack{b_1 \cdots b_n \geq 2 \\ b_j \in \cB_j(N_j) }} \, \int_{K} \abs{\zeta_{\va,N}^{\vb,{\bf N},\Pri} (\vs;\val)}^2 d \vs, \\ 
&\leq C^2 \operatorname{meas}(K) \sum_{\substack{b_1 \cdots b_n \geq 2 \\ b_j =1 \text{ or } b_j \geq N_j}} \, \prod_{j=1}^n b_j^{-1-2\ddddd}, \\ &\ll \min(N_1,\ldots,N_j)^{-2\ddddd} \leq \LL^{-2\ddddd} <\varepsilon, 
\end{align}
provided  $\LL$ is sufficiently large.  
We now divide the proof according to the conditions on $\val$  of Theorem  \ref{TH1}. 
\begin{description}[leftmargin=18pt]
  \item[$\alpha_n$ is transcendental:] This is the easiest case to treat, since if we choose $N_n=\LL$ and $N_j=Q(\LL+\lceil \alpha_n \rceil)$ for $1 \leq j \leq n-1$, where $Q$ denote the least common denominator of the parameters $\alpha_j$ that are rational (where $Q=1$ if they are all transcendental) then the condition $b_1 \cdots b_n \geq 2$ in the sum in \eqref{ret} forces at least one $k_j> L$ for $1 \leq j \leq n-1$  which  implies that the sum in \eqref{ret2} is empty.  Thus $\zeta_{\va,N}^{\vb,{\bf N},\Pri} (\vs;\val)$ and also $\Delta_N(K,\va,{\bf N},\Pri)$ will in fact be identically zero. Similarly in this case  $\zeta_{\va}^{[N]} (\vs;\val)=\zeta_{\va,N}^{\bf{1},\bf{N},\Pri}(\vs;\val)$ so that \eqref{visa} is trivially true. Note that when $\alpha_n$ is transcendental, these quantities will in fact not depend on $\Pri$, so $\Pri$ may be chosen arbitrarily.
  \item[$\alpha_{n-1}$ is transcendental and $\alpha_n$ rational:]  Let $N_n=\LL$ and use Lemma \ref{gy} to choose a finite set of primes $\Pri$ containing all primes less than $N_n$ such that \eqref{OPP9} and \eqref{OPP10} holds for $M$ given with $\alpha=\alpha_n$.  We now choose $N_{j}= Q (P+ \lfloor \alpha_n \rfloor)$ for $1\leq j \leq n-1$ where $Q$ is defined as the case above when $\alpha_n$ is transcendental and $P$ is the largest prime in $\Pri$. This forces $b_1 = \cdots =  b_{n-1}=1$ if we want the sum in \eqref{ret2} to be non empty. In case the sum is empty then \eqref{tr33} is trivially true. Thus we assume that $b_1 = \cdots =  b_{n-1}=1$ and that 
\begin{gather}
  \zeta_{\va,N}^{\vb,\bf{N},\Pri} (\vs;\val) = \sum_{1 \leq k_1 < \cdots <k_{n-1} \leq N} \, \sum_{\substack{d_n(k_n+\alpha_n) b_n^{-1} \in \cA_n( N_n) \\ k_{n-1} \leq k_n \leq N}} \prod_{j=1}^n \frac{a_j(k_j+\alpha_j)}{ (k_j+\alpha_j)^{s_j}},
\end{gather}
 which by Lemma \ref{TTV}  and $K \subseteq \M^n$ implies \eqref{tr33}. We also have that
\begin{gather} \label{ktet}
\zeta_{\va}^{[N]} (\vs)- \zeta_{\va,N}^{\bf{1},\bf{N},\Pri}(\vs)  =
  \sum_{\substack{0 \leq k_n \leq N \\ d_n(k_n+\alpha_n) \not \in \cA_n( N_n)  }}  \frac{a_n(k_n+\alpha_n)}{ (k_n+\alpha_n)^{s_n}}   \sum_{1 \leq k_1<\cdots<k_{n-1} < k_n}    \prod_{j=1}^{n-1} \frac{a_j(k_j+\alpha_j)}{ (k_j+\alpha_j)^{s_j}}.  \\  \intertext{By letting $k=k_n$ and defining $A(\vs)$ and $B_k(\vs)$ as in Lemma \ref{TTV} by \eqref{Mtineq} this can be written as (compare with \eqref{urt5})} \label{krtret}
A(\vs)   \sum_{\substack{0 \leq k \leq N \\ d_n(k+\alpha_n) \not \in \cA_n( N_n)  }}  \frac{a_n(k+\alpha_n)}{ (k+\alpha_n)^{s_n}}   +   \sum_{\substack{0 \leq k \leq N \\ d_n(k+\alpha_n) \not \in \cA_n( N_n)  }}  B_{k}(\vs) \frac{ a_n(k+\alpha_n)}{ (k+\alpha_n)^{s_n}}.
\end{gather}
The first term in \eqref{krtret} gives the contribution $O(M^{-1/2})$  by \eqref{Errorineq0} and \eqref{OPP10} and the second term  in \eqref{krtret}  is bounded by $O(M^{\varepsilon-\ddd})$ by \eqref{Errorineq} and by estimating the sum by its absolute values and using the fact that $\cA_n(N_n)$ contains all natural numbers  less than $M$ so that the sum is only over integers greater than $M$. If $M$ is sufficiently large this implies \eqref{visa}.

\item[$\alpha_n$ and $\alpha_{n-1}$ both rational:]
 If we assume the Riemann hypothesis condition we let $\Pri$ be the set of primes less than $\LL$. If instead we assume the condition \eqref{Kkond} on the set $K$ we use Lemma \ref{gy} to find a finite set of primes $\Pri$ containing all primes less than $\LL$ such that 
\eqref{OPP9} and \eqref{OPP10} holds with $\alpha=\alpha_n$ and for $\LL$ given.  
 We let $P$ be the largest prime in the set $\Pri$  and define
\begin{gather}
  \begin{split}
 N_{n-1}&=\exp(\exp(\LL)), \\  
 N_{j}&=\exp(\exp(N_{j+1})) \, \, \, \, (j=n-2,\ldots,1)  \\
 N_n &= \exp(\exp(N_1))
  \end{split}
      \, \, \, \, \, \, \, \ \begin{pmatrix} \text{the Riemann hypothesis is} \\ \text{true for } L(s,\chi_n\chi^*), \chi^* \mod d, \end{pmatrix}, \\ 
  \begin{split}
 N_{n}&=\exp(\exp(P)), \\   N_{j}&=\exp(\exp(N_{j+1})) \qquad (j=n-1,\ldots,1) 
  \end{split}
    \, \, \, \, \, \, \, \, \, \, \, \, \begin{pmatrix} \text{No Riemann hypothesis  } \\ \text{is assumed true }\end{pmatrix}. 
 \end{gather}
Finally we assume that $N \geq N_j$ for all $j=1,\ldots,n$ and define the intervals
\begin{gather}
  I_v= \begin{cases} [N_1,N] & v=1 \\ [N_{v},N_{v-1}] & 2 \leq v \leq n-1 \text { or (RH condition and }v=n) \\ [N_n,N] & v=n \text{ and no RH condition}
           \end{cases}
\end{gather}   
 It is clear that we can write
\begin{align}
 \zeta_{\va,N}^{\vb,\bf{N},\Pri} (\vs;\val)  &= \left(\sum_{k_v \not \in I_v \text{ for }  1\leq v \leq n-1} +\sum_{v=1}^{n-1} \sum_{k_v \in I_v}  \right)  \sum_{\substack{0 \leq k_1 < \cdots <k_n \\  d_j(k_j+\alpha_j) b_j^{-1} \in \cA_j( N_j) }} \prod_{j=1}^n \frac{a_j(k_j+\alpha_j)}{ (k_j+\alpha_j)^{s_j}}  \\ &=\qquad A(\vs,\vb) \, \, \, + \, \, \sum_{v=1}^{n-1} A_v(\vs,\vb), \label{yrte3} 
\end{align}
by the inclusion-exclusion principle since if $k_v \in I_v$ and $k_w \in I_w$  and $1 \leq v  <  w \leq n-1$ then $k_w \geq k_v$ so that the sum is empty. The following pictures illustrates how we do the summation when $n=5$.

\definecolor{zzttqq}{rgb}{0.27,0.27,0.27}
\definecolor{cqcqcq}{rgb}{0.75,0.75,0.75}
\begin{tikzpicture}[line cap=round,line join=round,>=triangle 45,x=0.4cm,y=0.4cm]
\draw [color=cqcqcq,dash pattern=on 2pt off 2pt, xstep=0.8cm, ystep=0.8cm]   (2,8) grid (12,18);
\draw [color=cqcqcq,dash pattern=on 2pt off 2pt, xstep=1.6cm, ystep=0.8cm] (12,8) grid (16,18);
\draw [color=cqcqcq,dash pattern=on 2pt off 2pt, xstep=0.8cm, ystep=0.8cm] (20,8) grid (30,18);
\draw [color=cqcqcq,dash pattern=on 2pt off 2pt, xstep=13.6cm, ystep=0.8cm] (30,8) grid (34,18);
\draw [color=zzttqq] (4,18)-- (4,18.4);
\draw [color=zzttqq] (6,18)-- (6,18.4);
\draw [color=zzttqq] (8,18)-- (8,18.4);
\draw [color=zzttqq] (10,18)-- (10,18.4);
\draw [color=zzttqq] (12,18)-- (12,18.4);
\clip(1.5,5) rectangle (37.5,21);
\fill[color=zzttqq,fill=zzttqq,fill opacity=0.1] (4,10) -- (6,10) -- (6,8) -- (4,8) -- cycle;
\draw[color=black] (5,9) node {$I_5$}; \draw[color=black] (31,9) node {$I_5$};
\fill[color=zzttqq,fill=zzttqq,fill opacity=0.1] (6,12) -- (6,10) -- (8,10) -- (8,12) -- cycle;
\draw[color=black] (7,11) node {$I_4$};\draw[color=black] (23,11) node {$I_4$};
\fill[color=zzttqq,fill=zzttqq,fill opacity=0.1] (8,14) -- (8,12) -- (10,12) -- (10,14) -- cycle;
\draw[color=black] (9,13) node {$I_3$}; \draw[color=black] (25,13) node {$I_3$};
\fill[color=zzttqq,fill=zzttqq,fill opacity=0.1] (10,16) -- (10,14) -- (12,14) -- (12,16) -- cycle;
\draw[color=black] (11,15) node {$I_2$}; \draw[color=black] (27,15) node {$I_2$};
\fill[color=zzttqq,fill=zzttqq,fill opacity=0.1] (12,18) -- (12,16) -- (16,16) -- (16,18) -- cycle;
\draw[color=black] (13,17) node {$I_1$};\draw[color=black] (30,17) node {$I_1$};
\draw[color=black] (4,19) node {$N_5$};
\draw[color=black] (6,19) node {$N_4$};
\draw[color=black] (8,19) node {$N_3$};
\draw[color=black] (10,19) node {$N_2$};
\draw[color=black] (12,19) node {$N_1$};
\draw[color=black] (16,19) node {$N$};
\draw[color=black] (22,19) node {$N_4$};
\draw[color=black] (24,19) node {$N_3$};
\draw[color=black] (26,19) node {$N_2$};
\draw[color=black] (28,19) node {$N_1$};
\draw[color=black] (30,19) node {$N_5$};
\draw[color=black] (34,19) node {$N$};
\draw[color=black] (3,17) node {$k_1$};
\draw[color=black] (3,9) node {$k_5$}; 
\draw[color=black] (3,11) node {$k_4$};
\draw[color=black] (3,13) node {$k_3$};
\draw[color=black] (3,15) node {$k_2$};
\draw[color=black] (21,17) node {$k_1$};
\draw[color=black] (21,9) node {$k_5$};
\draw[color=black] (21,11) node {$k_4$};
\draw[color=black] (21,13) node {$k_3$};
\draw[color=black] (21,15) node {$k_2$};
\draw [color=zzttqq] (22,18)-- (22,18.4);
\draw [color=zzttqq] (24,18)-- (24,18.4);
\draw [color=zzttqq] (26,18)-- (26,18.4);
\draw [color=zzttqq] (28,18)-- (28,18.4);
\draw [color=zzttqq] (30,18)-- (30,18.4);
\fill[color=zzttqq,fill=zzttqq,fill opacity=0.1] (30,10) -- (30,8) -- (34,8) -- (34,10) -- cycle;
\fill[color=zzttqq,fill=zzttqq,fill opacity=0.1] (22,12) -- (24,12) -- (24,10) -- (22,10) -- cycle;
\fill[color=zzttqq,fill=zzttqq,fill opacity=0.1] (24,14) -- (24,12) -- (26,12) -- (26,14) -- cycle;
\fill[color=zzttqq,fill=zzttqq,fill opacity=0.1] (26,16) -- (26,14) -- (28,14) -- (28,16) -- cycle;
\fill[color=zzttqq,fill=zzttqq,fill opacity=0.1] (28,16) -- (34,16) -- (34,18) -- (28,18) -- cycle;
\draw[color=black] (9.1,6.8) node {nothing assumed about RH} ;\draw[color=black] (27,6.8) node {RH holds for $L(s,\chi_n \chi^*), \chi^* \mod d$};
\draw (4,18)-- (14,18);
\draw (4,18)-- (4,8);
\draw (4,8)-- (16,8);
\draw [color=zzttqq] (4,10)-- (6,10);
\draw [color=zzttqq] (6,10)-- (6,8);
\draw [color=zzttqq] (6,8)-- (4,8);
\draw [color=zzttqq] (4,8)-- (4,10);
\draw [color=zzttqq] (6,12)-- (6,10);
\draw [color=zzttqq] (6,10)-- (8,10); 
\draw [color=zzttqq] (8,10)-- (8,12);
\draw [color=zzttqq] (8,12)-- (6,12);
\draw [color=zzttqq] (8,14)-- (8,12);
\draw [color=zzttqq] (8,12)-- (10,12);
\draw [color=zzttqq] (10,12)-- (10,14);
\draw [color=zzttqq] (10,14)-- (8,14);
\draw [color=zzttqq] (10,16)-- (10,14);
\draw [color=zzttqq] (10,14)-- (12,14);
\draw [color=zzttqq] (12,14)-- (12,16);
\draw [color=zzttqq] (12,16)-- (10,16);
\draw [color=zzttqq] (12,18)-- (12,16);
\draw [color=zzttqq] (12,16)-- (16,16);
\draw [color=zzttqq] (16,18)-- (12,18);
\draw (22,18)-- (22,8);
\draw (22,8)-- (34,8);
\draw (22,18)-- (34,18);
\draw [color=zzttqq] (30,10)-- (30,8);
\draw [color=zzttqq] (30,8)-- (34,8);
\draw [color=zzttqq] (34,10)-- (30,10);
\draw [color=zzttqq] (22,12)-- (24,12);
\draw [color=zzttqq] (24,12)-- (24,10);
\draw [color=zzttqq] (24,10)-- (22,10);
\draw [color=zzttqq] (22,10)-- (22,12);
\draw [color=zzttqq] (24,14)-- (24,12);
\draw [color=zzttqq] (24,12)-- (26,12);
\draw [color=zzttqq] (26,12)-- (26,14);
\draw [color=zzttqq] (26,14)-- (24,14);
\draw [color=zzttqq] (26,16)-- (26,14);
\draw [color=zzttqq] (26,14)-- (28,14);
\draw [color=zzttqq] (28,14)-- (28,16);
\draw [color=zzttqq] (28,16)-- (26,16);
\draw [color=zzttqq] (28,16)-- (34,16);
%\draw [color=zzttqq] (30,16)-- (30,18);
\draw [color=zzttqq] (30,18)-- (28,18);
\draw [color=zzttqq] (28,18)-- (28,16);
\draw (2,20)-- (2,8);
\draw (4,8)-- (2,8);
\draw (2,20)-- (16,20);
\draw (20,20)-- (20,8);
\draw (22,8)-- (20,8);
\draw (20,20)-- (34,20);
\end{tikzpicture}

From the illustration we see that we may write the first case as
\begin{multline} \label{uy6666}
 A(\vs;\vb)= \Delta(\vs;\vb)+ \sum_{k_v \not \in I_v \text{ for }  1\leq v \leq n-1}\sum_{\substack{0 \leq k_1 < \cdots <k_n \\   d_j(k_j+\alpha_j) b_j^{-1} \in \cA_j( N_j) }}
 \prod_{j=1}^n \frac{a_j(k_j+\alpha_j)}{ (k_j+\alpha_j)^{s_j}} \\ = \Delta(\vs;\vb) + 
\sum_{v=1}^{n-1} \Delta_v
\p{\sum_{\substack{0 \leq k_1 < \cdots <k_v  \leq N_v }}  \prod_{j=1}^v \frac{a_j(k_j+\alpha_j)}{ (k_j+\alpha_j)^{s_j}}} \times \\ \times \p{\sum_{\substack{N_v \leq k_{v+1} < \cdots <k_n \\   d_j(k_j+\alpha_j) b_j^{-1} \in \cA_j( N_j)  }}  \prod_{j=v+1}^n \frac{a_j(k_j+\alpha_j)}{ (k_j+\alpha_j)^{s_j}}}, 
\end{multline}
where
\begin{gather} \label{DV3}
  \Delta_v=\begin{cases} 1, & b_1=\cdots=b_v=1, \\ 0, & \text{otherwise,} \end{cases}
\end{gather} 
and
 \begin{align}
   \Delta(\vs;\vb)&=    \sum_{\substack{0 \leq k_1 < \cdots <k_{n-1} \leq N_{n-1} \\  k_{n-1} \leq k_{n} \leq N \\ d_n(k_n+\alpha_n) b_n^{-1} \in \cA_n( N_n) }} \prod_{j=1}^n \frac{a_j(k_j+\alpha_j)}{ (k_j+\alpha_j)^{s_j}} \qquad &(\text{no RH condition or } b_n=1 ) \qquad \label{norhc}  \\
\Delta(\vs;\vb) &= 0  \qquad &(\text{RH condition and } b_n \neq 1)    \qquad
\label{rhc}
 \end{align}
By Lemma \ref{TTV}   we may estimate
\begin{gather}
 \label{DE3} 
 \max_{\vs \in K} \abs{\Delta(\vs;\vb)} \ll_{\varepsilon} b_n^{\varepsilon-1/2-\ddd} \leq  \prod_{j=1}^n b_j^{\varepsilon-1/2-\ddd},
\end{gather}
where the last inequality follows from the fact that for  the sum to be non empty\footnote{This is similar to the case when $\alpha_{n-1}$ is transcendental.} it is neccessary that $b_1=\cdots=b_{n-1}=1$.
 By Lemma \ref{EEEV} the second factors in the sum \eqref{uy6666} are bounded and by Lemma \ref{EEV2} with the Riemann hypothesis condition, and Lemma \ref{EEV3} if we assume the condition \eqref{Kkond} on the set $K$,  the third factors  in the sum \eqref{uy6666} are bounded in absolute values by 
\begin{gather}  \prod_{j=v+1}^n b_j^{\varepsilon-1/2-\ddd} \end{gather}
By \eqref{DV3} it follows that each term in the sum \eqref{uy6666} may be estimated by
\begin{gather}
\prod_{j=1}^n b_j^{\varepsilon-\ddd-1/2},
\end{gather}
for any $\varepsilon>0$. Together with \eqref{DE3} and the triangle inequality this implies that 
\begin{gather} \label{rer30}
 \max_{\vs \in K} \abs{A(\vs;\vb)} \ll_\ddddd \prod_{j=1}^n b_j^{-1/2-\ddddd}. \qquad (b_j \in \cB_j(N_j), 0 <\ddddd<\ddd)
\end{gather}
We now proceed to treat the case 
when $k_v \in I_v$ for some $1 \leq v \leq n-1$. Let us write $k=k_v$. For the sum to be non-empty it is neccessary for $b_1=\cdots=b_{v-1}=1$. We may write \begin{multline} \label{Av}
 A_v(\vs;\vb)= \sum_{\substack{ d_v(k+\alpha_v) b_v^{-1} \in \cA_v( N_v) \\ k \in I_v}}   \frac{a_{v}(k+\alpha_{v})}{ (k+\alpha_{v})^{s_{v}}} \times \\ \times \p{\sum_{0 \leq k_1 < \cdots <k_{v-1}<k}  \prod_{j=1}^{v-1} \frac{a_j(k_j+\alpha_j)}{ (k_j+\alpha_j)^{s_j}}}   \p{\sum_{\substack{k < k_{v+1}< \cdots <k_{n} \leq N \\   d_j(k_j+\alpha_j) b_j^{-1} \in \cA_j( N_j) }}  \prod_{j=v+1}^{n} \frac{a_j(k_j+\alpha_j)}{ (k_j+\alpha_j)^{s_j}}},
     \end{multline} 
where if $v=1$ the second factor in the product is replaced by $1$.  Since the second factor is bounded by Lemma \ref{EEEV} (say have absolute value less than $B$ on $\M^{v}$) and $|k+\alpha_v|^{-\Re(s_v)} \leq k^{-\Re(s_v)}$
it follows by the triangle inequality that
\begin{gather} \label{Av2}
  \abs{A_v(\vs;\vb)} \ll \sum_{\substack{ d_v(k+\alpha_v) b_v^{-1} \in \cA_v( N_v) \\ k \in I_v}}  k^{-\Re(s_v)}  \abs{\sum_{\substack{k < k_{v+1}< \cdots <k_{n} \leq N \\   d_j(k_j+\alpha_j) b_j^{-1} \in \cA_j( N_j) }}  \prod_{j=v+1}^{n} \frac{a_j(k_j+\alpha_j)}{ (k_j+\alpha_j)^{s_j}}} \qquad 
     \end{gather}  
If we assume the Riemann hypothesis condition  we may use Lemma \ref{EEV2} to estimate  the final factor in \eqref{Av2} for $1 \leq v \leq n-1$ and we obtain together with the fact that $k^{-\Re(s)} \leq k^{-1/2}$  if $s \in \M$ that
\begin{gather} \label{rer312a} 
\max_{\vs \in K} \max_{1 \leq v \leq n-1} 
\abs { A_v(\vs;\vb)} \ll
 \prod_{j=v+1}^n b_j^{\varepsilon-\ddd-1/2} \sum_{ \substack{d_v(k+\alpha_v) b_v^{-1} \in \cA_v( N_v) \\ k \geq \LL}} 
  k^{-1/2 }(k^{\varepsilon-\ddd-1/2} + N_n^{\varepsilon-\ddd}). \qquad
\end{gather}
Similarly if instead of  assuming the Riemann hypothesis condition, we assume the condition \eqref{Kkond} on the set $K$  and $1 \leq v \leq n-1$ we may instead use Lemma  \ref{EEV3} to show that the terms in the sum in \eqref{Av2} may be bounded  by  $$k^{-1-\ddddd} \prod_{j=v+1}^n b_j^{-1/2-\ddddd}$$  whenever 
\begin{gather} \label{XID} 0<\ddddd<\frac 1 n \p{\min_{s \in K } \Re(s_n+s_{n-1})-\frac 3 2}, \end{gather}
so that
\begin{gather} \label{rer312b} 
\max_{\vs \in K} \max_{1 \leq v \leq n-1} 
\abs { A_v(\vs;\vb)} \ll
 \prod_{j=v+1}^n b_j^{-1/2-\ddddd} \sum_{ \substack{d_v(k+\alpha_v) b_v^{-1} \in \cA_v( N_v) \\ k \geq \LL}}  k^{-1-\ddddd}
\end{gather}
for some $\ddddd>0$. First we note that if  $b_v=1$, then inner sum in  both \eqref{rer312a} and \eqref{rer312b} can be estimated by $\LL^{-\ddddd}$ for some $\ddddd>0$, and by estimating the first products trivially we obtain
\begin{gather}
 \label{rer333}  \max_{\vs \in K}   \abs { A_v(\vs;\vb)} \ll   \LL^{-\ddddd}
\end{gather}
This will be useful to prove \eqref{visa} but not be sufficient to prove \eqref{tr33}.  By Lemma \ref{gy2} we may instead estimate the first part of the sum in \eqref{rer312a} and the sum in \eqref{rer312b} by \begin{gather}\label{rer313}
\sum_{ d_v(k+\alpha_v) b_v^{-1}  \in \cA_v( N_v)}  k^{-1-\xi} \ll_{\varepsilon} b_v^{\varepsilon-\xi-1}, 
\end{gather}
and the second part of the sum in \eqref{rer312a} may be estimated by absolute summation and Lemma \ref{RHV} by 
\begin{gather} \label{rer314}
N_n^{\varepsilon-\ddd} \sum_{ d_v(k_v+\alpha_v) b_v^{-1} \in \cA_v( N_v)} k^{\varepsilon-\ddd-1/2}  \ll  b_v^{\varepsilon-\ddd-1/2}\exp(\sqrt{N_v})N_n^{\varepsilon-\ddd} \ll b_v^{\varepsilon-\ddd-1/2}, \qquad
\end{gather} 
since in the Riemann hypothesis case we  have that $N_n \geq \exp(\exp(N_v))$ for any $1 \leq v \leq n-1$. 
From \eqref{rer312a}, \eqref{rer313}, \eqref{rer314} when the Riemann hypothesis condition holds, and \eqref{rer312b}, \eqref{rer313}, when condition \eqref{Kkond} holds we have that
\begin{gather} \label{rer31a} \begin{split}
\max_{\vs \in K} \max_{1 \leq v \leq n-1} 
\abs { A_v(\vs;\vb)} &\ll_\varepsilon
 \prod_{j=v+1}^n b_j^{\varepsilon-\ddd-1/2} \p{ b_v^{\varepsilon-\ddd-1/2}+\sum_{ d_v(k_v+\alpha_v) b_v^{-1} \in \cA_v( N_v)} 
  k^{\varepsilon-1-\ddd}}
 \\ &\ll_{\varepsilon} \prod_{j=v}^n b_j^{\varepsilon-\ddd-1/2}. 
\end{split}  \, \, \, \, \, 
\end{gather}
From the observation that unless $b_1=\cdots=b_{v-1}=1$ then $A_v(\vs;\vb)=0$, the inequality \eqref{rer31}  implies that
\begin{gather} \label{rer31} 
\max_{\vs \in K} \max_{1 \leq v \leq n-1} 
\abs { A_v(\vs;\vb)} \ll \prod_{j=1}^n b_j^{-\ddddd-1/2}, \qquad (b_j \in \cB_j(N_j), 0 <\ddddd<\ddd). \qquad
\end{gather}
which implies  \eqref{rer31} when $\ddddd$ satisfies \eqref{XID}. 
The equation \eqref{yrte3}, the inequalities \eqref{rer30} and \eqref{rer31} and the triangle inequality imply the inequality \eqref{tr33}. It remains to prove \eqref{visa}.  By \eqref{yrte3}, \eqref{uy6666}, \eqref{norhc}, \eqref{rhc} and \eqref{rer333}    we get that 
\begin{gather}
   \label{hytr}
  \zeta_{\va,N}^{\bf{1},\bf{N},\Pri} (\vs;\val)  =\sum_{\substack{0 \leq k_1 < \cdots <k_{n-1} \leq N_{n-1} \\  k_{n-1} \leq k_n \leq N \\ d_n(k_n+\alpha_n)  \in \cA_n( N_n) }} \prod_{j=1}^n \frac{a_j(k_j+\alpha_j)}{ (k_j+\alpha_j)^{s_j}} + O(\LL^{\varepsilon-\ddd}).
\\ \intertext{We can rewrite this in the following manner}  \label{hytr4}
\begin{multlined} \zeta_\va^{[\LL]}(\vs;\val)  + \sum_{\substack{ M \leq k_n \leq N \\ d_n(k_n+\alpha_n)  \in \cA_n( N_n) }} \frac{a_n(k_n+\alpha_n)}{ (k_n+\alpha_n)^{s_n}} \times \\ \times
  \sum_{0 \leq k_1 < \cdots <k_{n-1} \leq \min(k_n,N_{n-1})}   \prod_{j=1}^{n-1} \frac{a_j(k_j+\alpha_j)}{ (k_j+\alpha_j)^{s_j}} + O(\LL^{\varepsilon-\ddd}). \end{multlined}
\end{gather}
Similarly as in the proof of Lemma \ref{TTV} and  \eqref{krtret}  we define $A(\vs)$ and $B_k(\vs)$  by \eqref{Mtineq}
and we may with $k=k_n$ rewrite the inner summation in \eqref{hytr4} so that
\begin{gather}  \label{hytr9} \begin{multlined}
    \zeta_{\va,N}^{\bf{1},\bf{N},\Pri} (\vs;\val)  = \zeta_\va^{[\LL]}(\vs;\val)  +  A(s)   \sum_{\substack{ M \leq k \leq N \\ d_n(k+\alpha_n)  \in \cA_n( N_n) }}  \frac{a_n(k+\alpha_n)}{ (k+\alpha_n)^{s_n}}   \\ +  \sum_{\substack{ M \leq k \leq N_{n-1} \\ d_n(k+\alpha_n)  \in \cA_n( N_n) }}  B_{k}(s)  \frac{a_n(k+\alpha_n)}{ (k+\alpha_n)^{s_n}}   \\  +  B_{N_{n-1}} (\vs)  \sum_{\substack{ N_{n-1}  \leq k \leq N \\ d_n(k+\alpha_n)  \in \cA_n( N_n) }}   \frac{a_n(k+\alpha_n)}{ (k+\alpha_n)^{s_n}} + O(\LL^{\varepsilon-\ddd}).  \end{multlined}
\end{gather}
The first sum on the right hand side of \eqref{hytr9}  may be estimated by $O(\LL^{\varepsilon-\eta})$   (by \eqref{OPP10} if we assume no Riemann hypothesis condition and by Lemma \ref{RHV2} if we assume the Riemann hypothesis condition)  and \eqref{Errorineq0}. The second sum can be estimated by $O(\LL^{-\ddd})$  by using \eqref{Errorineq} and estimating the sum by its absolute values. For a sufficiently large $N$ the third sum may be estimated by $N_{n-1}^{1/2+\varepsilon-\ddd}$ and since the factor $\max_{\vs \in \M^n} \abs{B_{N_{n-1}} (\vs)} \ll {(N_{n-1})}^{-1/2}$, by utilizing the fact that $N_{n-1} \geq M$  the corresponding term in \eqref{hytr9} may be estimated by $O(\LL^{\varepsilon-\ddd})$.   By    Lemma \ref{EEEV} the first term equals  $\zeta_\va(\vs;\val)+O(\LL^{-1/2})$. Combining these estimates we find that 
\begin{gather}
  \zeta_{\va,N}^{\bf{1},\bf{N},\Pri} (\vs;\val)  = \zeta_\va(\vs;\val)+O(\LL^{-\ddddd})
\end{gather}
for some $\ddddd>0$. This implies \eqref{visa}.
 \end{description} 
\qed

\section{A mean square result and a weak approximate functional equation \label{sec5}}

For the  Hurwitz zeta-function there exists a naive approximate functional equation. Just approximate the zeta-function by a Dirichlet polynomial ($\sigma \geq 1/2$):
\begin{gather} \label{appfeq}
 \zeta(\sigma+it; \alpha)=\sum_{n=0}^{\lfloor T \rfloor} (n+\alpha)^{-\sigma-it} +O\left(\frac{T^{1-\sigma}} t\right), \qquad (|t| \leq  T).
\end{gather}
This is sufficient to prove the right order in the mean square formula, and also it is sufficient to prove universality. In order to get a better error estimate a useful trick is to consider a smoothed version of this equality. Instead of the truncated version we will consider the smoothed Dirichlet polynomial
\begin{gather}
\zeta_1^{[\phi,T]} (\sigma+it; \alpha)  = \sum_{n=0}^\infty \phi\p{\frac {n+\alpha} T} (n+\alpha)^{-\sigma-it},
 \end{gather}
 where $\phi$ fulfill \eqref{phidef}. We remark that this is the one-dimensional specialization of   \eqref{zetaphidef}.
\begin{lem} \label{appfeq2} Let $\phi$ fulfill \eqref{phidef}. Then for any $\alpha,\ddddd,A>0$ and $-A<\sigma<A$ we have uniformly for $|t| \leq T$ that
\begin{gather*} 
 \zeta(\sigma+it; \alpha) = \zeta_1^{[\phi,T]} (\sigma+it; \alpha) + \begin{cases} O \p{T^{-A}}, & (T^\ddddd \leq |t| \leq  T) \\ O \p{ \p{1+\frac 1 {|\sigma+it-1|}}  \, T^{1-\sigma}}  , & \text{otherwise.} \end{cases}
\end{gather*}
\end{lem}
The way this can be proved is by using the normal Mellin-Barnes integral (here $c$ is large enough so the Dirichlet series is absolutely convergent)
\begin{gather*}
  \zeta^{[\phi,T]}(s; \alpha) =
	\frac 1 {2 \pi i} \int_{c-\infty i}^{c+\infty} \zeta(z+s;\alpha ) T^s \, \int_0^\infty x^{z-1}\phi (x) dx dz,
\end{gather*}
shifting the integration line from $\Re(z)=c$ to $\Re(z)= -2A \log T$, picking up the residue at $z=0$, applying the Stirling formula and the  functional equation to
estimate the Hurwitz zeta-function on that line. In the process we pick up the residue of the Hurwitz zeta-function at $z=1-s$ which if $s$ is bounded away from 1 can be estimated  by the first estimate and if $s$ is close to 1 can be estimated by the second one. Similarly we will prove the analogue in several complex variables. We recall \eqref{zetaphidef} that
\begin{gather*}
\zeta_{n}^{[\phi,T]} (\vs;\va) = \sum_{0=k_0 \leq k_1<k_2<\cdots<k_n} 
\prod_{j=1}^n (k_j+\alpha_j)^{-s_j} \phi\p{\frac{k_{j}-k_{j-1}} T}.
\end{gather*}
As in the one-variable case, since $\phi$ has compact support this is a Dirichlet polynomial in $n$ variables. We note that the approximate functional equation below holds for the restricted variables $s_j$. In particular we assume that the sign of the imaginary parts of the $s_j$ coincide. The reason for this is simplicity and it is sufficient for our purposes. It allows us to avoid the set where $\zeta_n(\vs; \val)$ has singularities.

\begin{lem} (Weak approximate functional equation) \label{LE9999}
  For any $A,\ddddd>0$ and $\phi$ that fulfills \eqref{phidef} we have that there exists some $B$ so that for $T \geq 1$, $T^\ddddd \leq \Im(s_i) \leq T$, $-A \leq \Re(s_j) \leq A$ for all $j=1,\ldots, n$ the following inequality holds 
	   $$\abs{\zeta_n(\vs;\val)-\zeta_n^{[\phi,T]}(\vs;\val)} <B T^{-A}.$$
\end{lem}

While not our main object of study in this paper it is nevertheless of some interest to see which mean square results Lemma \ref{LE9999} gives on the critical line. While there is some related work in the case $n=2$ starting with Matsumoto-Tsumora \cite{MatHir}, the following result is as far as we know new\footnote{The result might be new also for $n=2$ since the results  in \cite{MatHir} are related, but not the same. We have not yet checked this thoroughly though (especially the related papers). We should do that.} for $n \geq 3$. 
\begin{cor} Let $0<\ddddd<1$. Then for $T \geq 2$
 \begin{gather*}
     \int_{[T^{\ddddd},T]^n} \abs{\zeta_n \p{\frac 1 2 +i\vt;\val}}^2 d \vt = \frac {(T \log T)^n}{n!} \p{1+O_\ddddd \p{\frac 1 {\log T}}}. 
\end{gather*}
\end{cor}

\begin{proof}
  This follows from applying the weak approximate functional equation, Lemma  \ref{LE9999} and then  taking the mean square. Interchange the integration and summation. The main term will come from  the diagonal terms and the error estimate comes from the off-diagonal terms and the Montgomery-Vaughan inequality. 
\end{proof}

\noindent{\em Proof of Lemma \ref{LE9999}.}
We first consider the multiple Hurwitz zeta-function proper following Matsumoto \cite[p.3]{Matsumoto2}. For $n \geq 2$ the idea is to use the Mellin-Barnes formula ($0<c<\Re(s)$)
   $$\Gamma(s)(1+\lambda)^{-s}=\frac 1 {2 \pi i}   \int_{\Re(z)=c} \Gamma(s-z) \Gamma(z) \lambda^{-z} dz,  $$
  first used by Katsurada \cite{Katsurada} on these types of problems. By recursion formula \cite[p.3, Eq (2.4)]{Matsumoto2} of Matsumoto we get
\begin{multline} \label{gammainthepath}
    \zeta_n(\vs;\val)=\frac 1 {2 \pi i} \int_{\Gamma}  \frac{\Gamma(z) \Gamma(s_n-z) }{\Gamma(s_n)}\zeta_1(z; 1+\alpha_n-\alpha_{n-1}) \times \\ \times \zeta_{n-1}( s_1,\ldots,s_{n-2},s_{n-1}+s_n-z;  \alpha_1,\ldots,\alpha_{n-1}) dz, \qquad
 \end{multline}
which is valid if $\Gamma=c+i\R$ for $c>1$ if $\Re(s_n)$ is sufficiently large so that $\zeta_{n-1}$ is absolutely convergent. Matsumoto then used this to analytically continue $\zeta_n$ to $\C^n$ by  doing analytic continuation in  the variable $s_n$. When we take care not to cross any singularities of the integrand  when $\Re(s_n)$ goes from a sufficiently large real number to any real number, our integration path $\Gamma$ must deform and for any $s_n=\sigma_n+it_n$ where $T^\ddddd \leq t_n \leq T$ we can choose the integration path 
\begin{gather} \label{gammapath} \Gamma= I_1 \cup I_2 \cup I_3 \cup I_4 \cup I_5, \\ \intertext{where} \label{hejsan}
\begin{split}
    I_1&=  [-N-\infty i,-N-T^{\ddddd/2} i], \qquad I_2= [-N-T^{\ddddd/2} i,N-T^{\ddddd/2} i], \, \,\\  I_3&=[N-T^{\ddddd/2} i, N+T^{\ddddd/2} i],    \qquad  \, \, \,  \, \, \, \,  I_4= [N+T^{\ddddd/2} i, -N+T^{\ddddd/2} i], \\  I_5&= [-N+T^{\ddddd/2} i,-N+\infty i], 
 \end{split}  
\end{gather}
for a sufficiently large positive $N$ such that the factor $\zeta_{n-1}$ in the integrand is absolutely convergent on $I_1$ and $I_5$.

\vskip 8pt

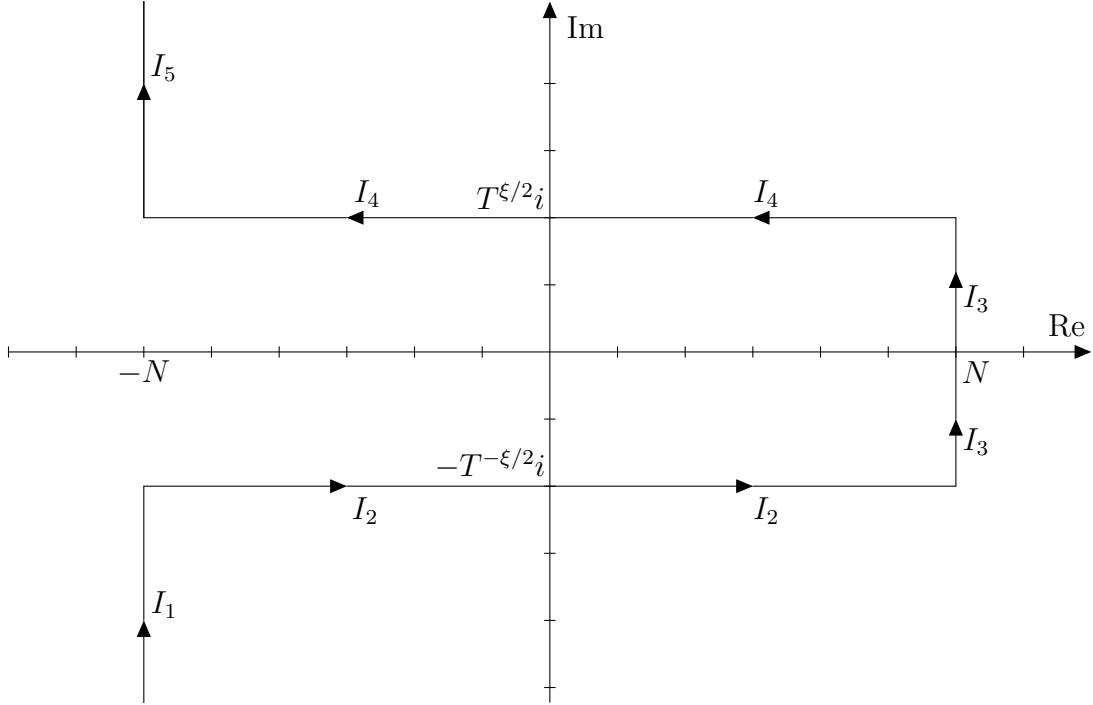
\begin{tikzpicture}[line cap=round,line join=round,>=triangle 45,x=0.89cm,y=1.78cm]
\draw[->,color=black] (-8,0) -- (8,0);
\foreach \x in {-8,-7,-6,-5,-4,-3,-2,-1,1,2,3,4,5,6,7}
\draw[shift={(\x,0)},color=black] (0pt,2pt) -- (0pt,-2pt);
\draw[color=black] (7.2,0.04) node [anchor=south west] { Re};
\draw[->,color=black] (0,-2.61) -- (0,2.61);
\foreach \y in {-2.5,-2,-1.5,-1,-0.5,0.5,1,1.5,2,2.5}
\draw[shift={(0,\y)},color=black] (2pt,0pt) -- (-2pt,0pt);
\draw[color=black] (0.09,2.42) node [anchor=west] { Im};
\clip(-8,-2.61) rectangle (8,2.61);
\draw (-6,1) -- (-6,2.61);
\draw [->] (-6,-1) -- (-3,-1);
\draw [->] (-3,-1) -- (3,-1);
\draw  (3,-1) -- (6,-1);
\draw [->] (-6,-3.61) -- (-6,-2);
\draw  (-6,-2) -- (-6,-1);
\draw [->] (6,-1) -- (6,-0.5);
\draw [->] (6,-0.5) -- (6,0.6);
\draw  (6,0.5) -- (6,1);
\draw [->] (6,1) -- (3,1);
\draw [->] (3,1) -- (-3,1);
\draw  (-3,1) -- (-6,1);
\draw [->] (-6,1) -- (-6,2);
\draw [->] (-6,2) -- (-6,2.81);
\draw [->] (-6,-3.18) -- (-6,-3);
\draw (-5.7,-1.88) node {$I_1$};
\draw (-2.73,-1.18) node {$I_2$};
\draw (3.2,-1.18) node {$I_2$};
\draw (6.3,0.40) node {$I_3$};
\draw (6.3,-0.67) node {$I_3$};
\draw (-2.7,1.18) node {$I_4$};
\draw (3.2,1.18) node {$I_4$}; 
\draw (-0.58,1.17) node {$T^{\ddddd/2} i$};
\draw (-0.88,-0.83) node {$-T^{-\ddddd/2} i$};
\draw (-5.71,2.12) node {$I_5$};
\draw (6.3,-0.15) node {$N$};
\draw (-6,-0.15) node {$-N$};
\end{tikzpicture}

We remark that because of the conditions $T^\ddddd<\Im(s_j)<T$, the integration path becomes simpler than in Matsumoto's treatment since less care avoiding the possible singularities of $\zeta_{n-1}$ is needed. 
We now apply the same method to get a recursion formula for the  modified version $\zeta_n^{[\phi,T]}(\vs;\va)$ of the multiple Hurwitz zeta-function
\begin{multline} \label{gammainthepath2}
    \zeta_n^{[\phi,T]}(\vs;\val)=\frac 1 {2 \pi i} \int_{\Gamma}  \frac{\Gamma(z) \Gamma(s_n-z) }{\Gamma(s_n)}\zeta_1^{[\phi,T]}(z; 1+\alpha_n-\alpha_{n-1}) \times \\ \times \zeta_{n-1}^{[\phi,T]}( s_1,\ldots,s_{n-2},s_{n-1}+s_n-z;  \alpha_1,\ldots,\alpha_{n-1}) dz. \qquad
 \end{multline}
In fact at each step it will be somewhat simpler since each modified multiple zeta-function is absolute convergent so we could use the integration path $\Gamma=c+i\R$ for any $c>0$. However nothing stops us from doing the same modification to the path as for the multiple Hurwitz  zeta function proper so we will use $\Gamma$ defined by \eqref{gammapath} and \eqref{hejsan} in both equations \eqref{gammainthepath},\eqref{gammainthepath2}. From the equations  \eqref{gammainthepath},\eqref{gammainthepath2} it now follows that
\begin{gather} \label{ai3}
 \zeta_n(\vs;\val) -  \zeta_n^{[\phi,T]}(\vs;\val) =\frac 1 {2 \pi i} \int_\Gamma B(z,s_n)\p{\zeta_{1} (z) \zeta_{n-1} (z) - \zeta_{1}^*(z) \zeta_{n-1}^{*}(z)} dz, \qquad
\end{gather}
where $\Gamma$ is defined by \eqref{gammapath} and where we have written the Gamma-factors as the Beta-function and the zeta-factors
\begin{align}
 \zeta_{1} (z)&=\zeta_1(z; 1+\alpha_n-\alpha_{n-1}), \\  \zeta_{n-1}(z)&= \zeta_{n-1}( s_1,\ldots,s_{n-2},s_{n-1}+s_n-z;  \alpha_1,\ldots,\alpha_{n-1}), \\
\intertext{and} 
 \zeta_{1}^* (z) &= \zeta_1^{[\phi,T]}(z; 1+\alpha_n-\alpha_{n-1}), \\ \zeta_{n-1}^{*}(z)&=\zeta_{n-1}^{[\phi,T]}( s_1,\ldots,s_{n-2},s_{n-1}+s_n-z;  \alpha_1,\ldots,\alpha_{n-1}), 
\end{align}
   are abbreviated versions of the corresponding factors that occurs in \eqref{gammainthepath},\eqref{gammainthepath2}. The proof of Lemma \ref{LE9999} now follows by the principle of induction. The base case  $n=1$  follows from Lemma \ref{appfeq2}. Let us now assume it is true for $n-1$ (the induction hypothesis).
 By letting $\Delta_1(z)=\zeta_{1}(z) -  \zeta_{1}^{*}(z)$ and $\Delta_{n-1}(z)=\zeta_{n-1}(z)-\zeta_{n-1}^{*}(z)$, the second two factors in the integrand can be rewritten as
\begin{gather}
 \zeta_{1}  (z) \zeta_{n-1}(z) - \zeta_{1}^{*}(z) \zeta_{n-1}^{*}(z) = A_1(s)+A_2(s)+A_3(s), \\ \intertext{where} 
A_1(s)= \Delta_1(z)\Delta_{n-1}(z), \qquad A_2(s)= \Delta_1(z) \zeta_{n-1}^{*}(z),   \qquad  A_3(s)  =  \zeta_{1}^{*}(z) \Delta_{n-1}(z).
\end{gather}
Thus
\begin{gather}
 \zeta_n(\vs;\val) -  \zeta_n^{[\phi,T]}(\vs;\val) = \sum_{k=1}^5 \sum_{j=1}^3 \frac 1 {2 \pi i} \int_{I_k} B(z,s_n) A_j(z) dz,
\end{gather}
and by using the triangle inequality on \eqref{ai3} it is sufficient to to prove that 
\begin{gather} \label{ineq}
  \abs{\frac 1 {2 \pi i}\int_{I_k} B(z,s_n) A_j(z) dz}\ll_C  T^{-C}, 
\end{gather}
for each $1 \leq k \leq 5$ and $1 \leq j \leq 3$.
We will now state some inequalities for the factors that will be used in the proof. By the induction hypothesis we have that
\begin{gather} \label{eq48}
  |\Delta_{n-1}(z)| \ll_C T^{-C} \, \, \, \, \, \, (-N \leq \Re(z) \leq N, \, \,   -T^{\ddddd/2}<\Im(z)<\Re(s_n)+T^{\ddddd/2}) \qquad \\ \intertext{
By Lemma \ref{appfeq2} it follows that that}  \label{eq49}
 |\Delta_1(z)| \ll_C T^{-C}, \qquad  (-N \leq \Re(z) \leq N, \, \, \,  T^{\ddddd/2}<|\Im(z)|<T). 
\end{gather}
The following is a standard estimate for the Hurwitz zeta-function that is a consequence of its functional equation,  which is weaker than the above but sufficient for our purposes, and also holds when $|\Im(z)|>T$.
\begin{gather} \label{eq50}
 |\Delta_1(z)|, |\zeta_1(z)|, |\zeta_1^*(z)| \ll  1+ (\log (1+|\Im z|)) |\Im(z)|^{1-\Re(z)} \, \, \, \, \, \,  (|\Im(z)| \geq 1) \, \, \, \, \, \, \, \, \, \,  \\
\intertext{By approximating the Dirichlet-polynomials trivially by their absolute values we have that}
 \abs{\zeta_{n-1}^*(z)} \ll T^{n(A+1)+N},          \qquad         (-N \leq \Re(z) \leq N). \, \, \, \, \, \, \label{eq51}
\\ \intertext{Finally because  $\zeta_{n-1}(z)$ is absolutely convergent  for $\Re(z)=-N$  we have}
 |\Delta_{n-1}(z)|, |\zeta_{n-1}(z)|, |\zeta_{n-1}^*(z)| \ll 1,                   \qquad (\Re(z)=-N).  \label{eq52}
\end{gather}
We are now ready to apply these estimates to prove the estimate \eqref{ineq} which by the principle of induction concludes our proof.
\begin{description}
  \item[$k=1,2$:] In this case Stirling's formula implies that $|B(z,s_n)| \ll T^{N} \exp\p{-|z|}$ which is sufficiently small on $I_1,I_2$ and furthermore exponentially decreasing as $\Im(z) \to -\infty$. For $k=1$ it is sufficient to estimate $\zeta_{n-1}(z)$, $\Delta_{n-1}(z)$ by \eqref{eq52}, and $\zeta_1^*(z)$, $\Delta_1(z)$ by \eqref{eq50} to obtain \eqref{ineq} for $1 \leq j \leq 3$. For $k=2$ we may instead estimate $\Delta_{n-1}(z)$ by \eqref{eq48} and $\zeta_{n-1}^*(z)$ by \eqref{eq51} to prove \eqref{ineq} for $1 \leq j \leq 3$.
  \item[$k=3:$]	 We note that the Stirling formula in this case gives $|B(z,s_n)| \ll T^{-\ddddd N/2}$. By using this and estimating $\Delta_{n-1}(z)$ with \eqref{eq48} and $\Delta_1(z),\zeta_1^*(z) \ll 1$ by the trivial estimate, the inequality \eqref{ineq} follows for $j=1,3$. The most difficult case to consider is possibly $k=3$, $j=2$.  Here we use the second case of  Lemma \ref{appfeq2} to estimate  $\Delta_1(z)$ and the  \eqref{eq51} to estimate $\zeta_n^*(z)$. Together these estimates gives us
$$\abs{\frac 1 {2 \pi i}\int_{I_3} B(z,s_n) A_2(z) dz}\ll T^{n(A+1)+1-\ddddd(N-1)/2},$$	
which is $\ll T^{-C}$ provided we choose $N$ sufficiently large (depending on $\ddddd,n,A$).
 	\item[$k=4$:] In this case Stirling's formula implies that $|B(z,s_n)| \ll T^{N}$. It is sufficient to estimate $\Delta_1(z),\Delta_{n-1}(z)$ by estimates \eqref{eq49} and \eqref{eq48}, and $\zeta_1^*(z)$ and $\zeta_{n-1}^*(z)$ by \eqref{eq50} and \eqref{eq51} to prove \eqref{ineq} for  $1 \leq j \leq 3$.
	\item[$k=5$:] Here we divide the range into two parts. The first part is when $\Re(z) \leq \Re(s_n)+T^{\ddddd/2}$. In this interval it is sufficient to estimate $\Delta_{n-1}(z)$ by \eqref{eq48} and $\Delta_{1}(z)$ by \eqref{eq49}, $\zeta_{n-1}^*(z)$ by \eqref{eq52}, $\zeta_1^*(z)$ by \eqref{eq50} and $|B(z;s_n)| \ll T^N$ by Stirling's formula. If $\Re(z) \geq \Re(s_n)+T^{\ddddd/2}$ we estimate $\Delta_{1}(z)$ by \eqref{eq52} instead, but in this range the Stirling formula gives us $|B(z;s_n)|\ll  T^N \exp \p{-|z-s_n|} \ll T^N \exp \p{-T^{\ddddd/2}}$ which is a sufficiently good estimate to give us what we need. Together these two cases gives us \eqref{ineq} for $1 \leq j \leq 3$.
\end{description}
\qed  

\section{Proofs of Theorems \ref{TH4} and \ref{TH5}}

Also for this case we follow the proof of Theorem \ref{TH1}. The fundamental lemma, Lemma \ref{LE4} may be used as is, but we need a new version of Lemma \ref{LE6}.

\begin{lem} \label{LE6v2}
 Let  $n\geq 2$ and $1 \leq m < n$ and the conditions on $\val$, $K$ and $\{v_1,\ldots,v_m\}$ of Theorem \ref{TH4} and Theorem \ref{TH5} respectively be satisfied.   Then for any $\varepsilon>0$ and polynomial $p$ in $n$ complex variables we have that 
\begin{multline}
     \liminf_{N \to \infty} \liminf_{T \to \infty}  \frac 1 {T^m} \operatorname{meas} {\Bigg\{ }\vt \in [0,T]^m: 
\\ \left. \max_{\vs \in K} \abs{\zeta_{\ett}^{[N]}\p{\sigma +i\sum_{j=1}^m t_j v_j;\val} -\zeta_{n}^{[N]}\p{\sigma +i\sum_{j=1}^m t_j v_j}} < \varepsilon  \right \} >0.
  \end{multline}
\end{lem}
\begin{proof}
This has a similar proof as Lemma \ref{LE6} although with some additional arithmetic input.  The arithmetic information we need in order to prove this is the following
\begin{enumerate}
 \item For Theorem \ref{TH4}:  If $v_1,\ldots,v_m$ are algebraic numbers linearly independent over $\Q$, then the set $$\{v_j \log p: 1\leq j \leq m \text{ and } p \text{ is a prime}\}$$ is linearly independent over $\Q$. This follows from  Baker's \cite[p. 10]{baker} Theorem 2.1.
\item For Theorem \ref{TH5}: If $\alpha_1,\ldots,\alpha_m$ are algebraically independent then the set $$\{\log(k+\alpha_j): 1\leq j \leq m: k \in \N\}\cup \{\log p: p \text{ is a prime} \}$$ is linearly independent over $\Q$.
\end{enumerate}
\end{proof}
We then need a suitable version of Lemma \ref{LE8}.

\begin{lem} \label{LE8v2}
 Let  $\phi$ fulfill \eqref{phidef} and let $0<\ddd,\ddddd<1$.  Let 
 $\{v_1,\ldots,v_m\}$ be linearly independent vectors such that  $v_j \in (\R^+)^n$. Then for any $1-m/(2n)+\ddd<\sigma_0<1$  there exist some $C>0$ such that if $1 \leq N \leq T$ then
\begin{gather}
    \frac 1 {T^n}\int_{\sigma_0}^1 \int_{{[T^{\ddddd},T]}^m} \abs{\zeta_{\ett}^{[N]} \p{\sigma +i\sum_{j=1}^m t_j v_j;\val}-\zeta_n^{[\phi,T]}\p{\sigma +i\sum_{j=1}^m t_j v_j;\val}}^2 d\vt d \sigma \leq \frac{C}  {N^{\ddd}} . 
\end{gather}
\begin{proof} This follows by using the approximate functional equation, Lemma \ref{LE7} in the much the same manner as we use in the proof of Lemma \ref{LE8}. By classical methods, squaring and treating the non-diagonal part in an elementary way. We get a smaller region  where we have the mean square estimate than in Lemma \ref{LE8}  since we take the mean square with respect to fewer variables.
 \end{proof}
\end{lem}
The proof of Theorems \ref{TH4} and \ref{TH5} now follows from Lemma \ref{LE4}, Lemma \ref{LE6v2}, Lemma \ref{LE7} and Lemma \ref{LE8v2} in much the same manner as Theorem \ref{TH1} follows from Lemma \ref{LE4}, Lemma \ref{LE6}, Lemma \ref{LE7} and Lemma \ref{LE8}. 

\section{Questions, generalizations and future work}

\subsection{What about discrete universality}
A discrete universality version of Theorem \ref{TH1} can certainly be proved without any extra difficulty.  For details see v1 of this paper on arXiv \cite{Aold}. In order to keep the current paper focused we have removed the discussion of this and expect to publish this more properly in a sequel.

\subsection{What about joint universality?}
The theorems of this paper may also be proved in joint universality versions for $\zeta_n(\vs,\val_1), \ldots, \zeta_n(\vs,\val_m)$ under certain conditions on  the $m$-tuple $(\val_1,\ldots,\val_m)$ where $\val_j \in (\R^+)^n$. The conditions become rather technical and we have therefore decided to remove the joint universality theorems from this version of the current paper. We expect to publish these results in a sequel. The interested reader can however find the results in v1 of the current paper on arXiv \cite{Aold}.

\subsection{What multiple zeta functions/$L$-functions/Dirichlet series are universal in several complex variables?}
Our first question is the following. There are a multitude of zeta functions and Dirichlet series in several complex variables. For what zeta-functions can we prove universality?
While we have decided to focus this paper on the Multiple Hurwitz zeta-function case, our method is quite general and should probably work for the multiple Dirichlet series
\begin{gather}
  Z(\vs)= \sum_{1 \leq k_1 < \cdots <k_n}  \prod_{j=1}^n a_{j,k_j}  \lambda_{j,k_j}^{-s_j} \label{rrtt} 
\end{gather}
provided the multiple Dirichlet series has some meromorphic continuation  beyond its region of absolute convergence and  that we can prove universality for the one variable Dirichlet series
\begin{gather}
  Z_j(s)= \sum_{k=1}^\infty a_{j,k} \lambda_{j,k}^{-s} 
\end{gather}
for each $1 \leq j \leq n$ and some region of the variable $s$.   Our method   should also be able to handle
\begin{gather} \label{rrtt2}
 \sum_{1 \leq k_1 < \cdots <k_n}   \prod_{j=1}^n a_{j,k_j}  (\lambda_{1,k_1}+\cdots+ \lambda_{j,k_j})^{-s_j}
\end{gather}
where for example $\lambda_{j,k}=k$ and $a_{j,k}=a_j(k)$ could be some multiplicative functions, such as for example Fourier coefficients of Hecke-Mass cusp forms. This type of  multiple zeta-function   has been studied by Matsumoto-Tanigawa \cite{Matsumoto5}  and is not of type \eqref{rrtt}.  There are however plenty of Dirichlet series in several complex variables that are neither of type \eqref{rrtt} nor \eqref{rrtt2}. Whether the methods of this paper are applicable should ultimately depend on how similar the given Dirichlet series/zeta-function in several complex variables is to these two examples.

 In particular there is  a  family of Dirichlet series called the  Weyl group multiple Dirichlet series \cite{WMD2}, that is quite different from the multiple zeta-functions. While the methods of this paper can not handle the Weyl group multiple Dirichlet series we are developing (starting in 2013 when visiting an ICERM program) a different method that can handle some of the simplest cases of Weyl group multiple Dirichlet series. The details are actually somewhat less complex than for this paper and the result will appear in a forthcoming paper \cite{Andersson8}. In contrast to what is stated in v1 of the current paper \cite{Aold} our argument will require the Generalized Riemann hypothesis rather than the weaker assumption of no Landau-Siegel zero. In recent papers \cite{Anew1,Anew2,Anew3} we developed universality results with scaling that are true in the  half-plane of absolute convergence. This will also allow us to prove corresponding unconditional universality result  for some Weyl group multiple Dirichlet series in several complex variables. Cases we can treat (both conditional and unconditional results) include
    the simplest case of a Weyl group multiple Dirichlet series, the Double Dirichlet series  originally intruduced by Siegel \cite{siegel} and further studied by Goldfeld-Hoffstein \cite{GolHof},   the Double Dirichlet series introduced by Friedberg-Hoffman-Lieman \cite{FriHoLi} and some Weyl group multiple Dirichlet series in more than two variables, such as the triple Dirichlet series from Brubaker's thesis \cite{Brubaker}.

\subsection{Can we obtain effective results?}  
Even in one variable the problem of obtaining effective results is quite difficult and largely unsolved. While the other steps in the proofs of universality can be made effective, the Pechersky rearrangement theorem is inherently ineffective\footnote{By the use of the Riesz representation theorem in its proof.}. There is however another method of Garunk\v{s}tis \cite[Proposition 3]{garunkstis} building on ideas of Good \cite{Good} that can be used as a replacement of the Pechersky rearrangment theorem. We have to pay a price though, and in this case the price is that the sets $K$ where we can prove universality will be much smaller.  The following \cite[Corollary 2]{garunkstis} result
		  illustrates the problem 
\begin{garunkstis} Let $0 <\varepsilon \leq 1/2$. Let the function $g(s)$ be analytic in the disc $\abs{s} \leq 0.06$ and assume $\max_{|s| \leq 0.06 } \abs{g(s)}\leq 1.$ Let $K=\{ s \in \C :|s-3/4| \leq 0.0001\}$. Then 
$$\liminf_{T \to \infty} \frac 1 T \mathop{\rm meas} \left \{t \in [0,T]:\max_{s \in K} \abs{\log \zeta(s+it)-g(s)}<\varepsilon \right \} \geq \exp \p{-\varepsilon^{-13}}. $$
Furthermore there exist some $0 \leq t \leq \exp \p{\exp\p{10 \varepsilon^{-13}}}$ such that the approximation$$\max_{s \in K} \abs{\log \zeta(s+it)-g(s)}<\varepsilon$$ holds.
\end{garunkstis}
In the same manner as in one variable it should be possible  to make Lemma \ref{LE7} and Lemma \ref{LE8} effective. It should also be possible to make Lemma \ref{LE6} effective if either $\alpha_n$  is transcendental or we assume the Generalized Riemann hypothesis. However if $\alpha_n$ is rational then  Lemma \ref{BNDDIR} is needed in the proof of Lemma \ref{LE6}, and since a version of the Pechersky rearrangement theorem is used to prove Lemma \ref{BNDDIR},  we must find some other approach such as the one of Garunk\v{s}tis. Also in order to obtain an effective version of our fundamental Lemma, Lemma \ref{LE4} we must use an effective replacement of the Pechersky rearrangement theorem such as the one found in  Garunk\v{s}tis. This may be possible, although perhaps cumbersome, assuming as in the one variable case, that we consider approximations on a much smaller set $K$.  Also note that since in our proof we use the Pechersky rearrangement theorems repeatedly, we must also use Garunk\v{s}tis' replacement  repeatedly and at the end our effective result should be very weak. Towers of exponentials should appear in such a result, where the length of the tower depends on $\varepsilon$ and the function we want to approximate.

In another direction Lamzouri-Lester-Radziwill \cite[Theorem 1.1]{LaLeRa} recently proved a result that may also be viewed as an effective version of a universality 
\begin{lam}
 Let $0<r<\frac 1 4$. Let $f$ be a non-vanishing continuous function on $|z| \leq (r+1/4)/2$ that is holomorphic in $|z|<(r+1/4)/2$. Let $\omega$ be a real-valued continuously differentiable function with compact support. Then, we have
\begin{align}
\frac 1 T \int_T^{2T} \omega \p{\max_{|z| \leq r} \abs{\zeta \p{\frac 3 4+it+z}-f(z)}}
=& \mathbb E\p{\omega \p{\max_{|z| \leq r} \abs{\zeta \p{\frac 3 4+z,X}-f(z)}}} \\&+ O \p{(\log T)^{-\frac{(3/4-r)}{11}+o(1)}},
\end{align}
where the constant in the $O$ depends on $f,\omega$ and $r$ and
\begin{gather} \label{ytr88} \zeta(s,X)=\prod_p \p{1 - \frac {X(p)}{p^s}}^{-1}
\end{gather}
where $X(p)$ are uniformly distributed independent random variables on the unit circle.
\end{lam}
Can we obtain some analogue of  Lamzouri-Lester-Radziwill theorem for multiple zeta-functions of several complex variables?  In several variables one may introduce something corresponding to \eqref{ytr88}. However, unlike in the one variable case, it does not seem easy to prove that this object has a limit distribution,  so it does not seem as their method is applicable on the corresponding problem for multiple zeta-functions in several complex variables.

\bibliographystyle{hplain}

\end{document}